\documentclass[aap,preprint]{imsart}

\usepackage[OT1]{fontenc}
\usepackage{mlmodern}
\usepackage{amsmath}
\usepackage{amssymb}
\usepackage{amsthm}
\usepackage{graphicx}
\usepackage{enumerate}
\usepackage{url}
\usepackage{bbm}
\usepackage{ifthen}
\usepackage{latexsym}
\usepackage{mathrsfs}
\usepackage{xcolor}
\usepackage{caption, subcaption}
\usepackage{stmaryrd}
\usepackage{dsfont}
\usepackage{bm}
\usepackage{booktabs}
\usepackage{esvect}

\usepackage[numbers]{natbib}

\usepackage[colorlinks,citecolor=blue,urlcolor=blue]{hyperref}

\startlocaldefs
\numberwithin{equation}{section}

\newcommand{\U}{\mathbb{U}}

\newcommand{\Lower}[2]{\smash{\lower #1 \hbox{#2}}}
\newcommand{\ben}{\begin{enumerate}}
\newcommand{\een}{\end{enumerate}}
\newcommand{\bi}{\begin{itemize}}
\newcommand{\ei}{\end{itemize}}

\newcommand{\E}{\mathbb{E}}

\newcommand{\MtP}{\mathrm{MtP}}

\renewcommand{\vec}[1]{\boldsymbol{#1}}

\newcommand{\vect}[1]{\bm{#1}}

\theoremstyle{plain}
\newtheorem{thm}{Theorem}[section]
\newtheorem{prop}{Proposition}[section]
\newtheorem{lem}{Lemma}[section]
\newtheorem{cor}{Corollary}[section]

\theoremstyle{definition}
\newtheorem{rem}{Remark}[section]

\newtheoremstyle{example}{\topsep}{\topsep}%
     {}
     {}
     {\rmfamily}
     {}
     {\newline}
     {\thmname{#1}\thmnumber{ #2}\thmnote{ #3}}
\theoremstyle{example}
\newtheorem{exam}{Example}[subsection]
\newtheorem{setup}{Setup}

\endlocaldefs

\begin{document}

\begin{frontmatter}
\title{Coagulation-Fragmentation Duality of Infinitely Exchangeable Partitions from Coupled Mixed Poisson Species Sampling Models}
\runtitle{Coupled Mixed Poisson Duality}

\begin{aug}
\author[A]{\fnms{Lancelot}~\snm{F. James}\ead[label=e1]{lancelot@ust.hk}}
\address[A]{Department of ISOM, HKUST\printead[presep={,\ }]{e1}}
\end{aug}

\begin{abstract}
We generalize the celebrated coagulation-fragmentation duality of Jim Pitman~\cite{Pit99Coag}, originally established for the PD($\alpha$,$\theta$) laws of Pitman and Marc Yor~\cite{PY97}, resolving a problem open for over two decades. Our framework extends the duality to processes driven by arbitrary non-negative L\'evy subordinators and, for the first time, to multi-group settings with coupled dynamics. The solution is a novel four-component system built from the Poisson Hierarchical Indian Buffet Process (PHIBP), a framework developed for modeling complex microbiome species sampling~\cite{hibp25}, which circumvents intractable analysis on traditional partition spaces.

The architecture of this construction embeds naturally within Bertoin's~\cite{BerFrag} continuous-time fragmentation framework, resolving a foundational impasse he highlighted~\cite[Ch. 4 p. 213]{BerFrag}: where time-reversal fails, we introduce \emph{simultaneous structural duality}, in which fragmentation and coalescence evolve in physical time while maintaining a pointwise dual relationship via coupled subordinators. This opens new modeling avenues for complex ancestral processes such as recombination graphs. The compound Poisson representation underlying the construction provides the first tractable path to explicit joint EPPF laws for a broad class of coupled processes, enabling exact simulation. Section~7 develops the \emph{cloud duality principle} via the h-biased L\'evy--It\^o coupled duality constructor, extending the framework to arbitrary Polish spaces and demonstrating that the duality arises from point-process regrouping rather than features specific to interval partitions.
\end{abstract}



\end{frontmatter}
\tableofcontents

\section{Introduction}
The study of exchangeable random partitions, which provide a mathematical framework for modeling the clustering of data, has deep roots in population genetics, Bayesian statistics, and the theory of stochastic processes. A cornerstone of this field is the two-parameter Poisson-Dirichlet family as developed and discussed by Jim Pitman and Marc Yor~\cite{PY97} and its remarkable coagulation-fragmentation duality, discovered by Pitman~\cite{Pit99Coag}. Extending this elegant static duality to processes driven by more general non-negative L\'evy subordinators—beyond the special stable subordinator beta-gamma algebra underlying the two-parameter family—has remained a significant open problem for over two decades. Moreover, extending such duality to multi-group settings with $J \geq 1$ structured populations has, to our knowledge, never been achieved even for the classical two-parameter case.

Beyond this static challenge lies an equally fundamental question about continuous-time dynamics. The profound utility of Pitman's duality is most evident in this setting---though its applications are far wider, as detailed in, e.g.,~\cite{Haas2,HoJamesLau2025,Pit06,Wood}---where it provides the generative logic for seminal processes ranging from the Bolthausen-Sznitman coalescent to cascades in spin-glass theory. This very success, however, exposes a deep-seated obstacle. As Bertoin highlights in the direct context of discussing Pitman's~\cite{Pit99Coag} duality~\cite[See p. 213]{BerFrag}, while the structural similarities between fragmentation and coalescence ``might suggest that these two classes of processes should be related by time-reversal,'' in reality, ``the situation is far less simple.'' The reason for this impasse is a fundamental asymmetry: the natural mechanism for such a link, subordination, preserves the structure of coalescents but fails to preserve the essential branching property of fragmentations. This has left the full dynamic potential of fragmentation-coalescence duality largely untapped, creating a major open question for extending these elegant relationships beyond the specific parameter-encoded dynamics of the classical two-parameter case.

This paper resolves both challenges—extending Pitman's duality to arbitrary subordinators and, for the first time, to multi-group settings with coupled dynamics across $J$ populations—and providing a framework for utilizing this duality in continuous time through a fundamentally new perspective. Rather than pursuing the time-reversal route that the classical duality might suggest, we introduce a novel notion of \emph{simultaneous structural duality}: at each fixed time $t$, fragmentation and coalescence dynamics coexist and maintain their dual relationship, with one process fragmenting forward while its dual coalesces backward at the same instant. This is achieved through a constructive approach built upon a hierarchical mixed Poisson perspective: we introduce, for the first time, a four-component coupled dynamic system in which the static, fixed-time partitions—the traditional object of study—are revealed to be merely a single projection of a far richer, dynamic structure.

The hallmark of this approach and its core contribution is twofold. First, it sidesteps the notoriously complex and often obscure analysis required on spaces of mass or integer partitions. Instead of confronting this complexity directly, our framework shifts the problem to an explicitly constructed four-component system where the dynamics are transparent. Second, it provides the first demonstration of how Pitman-type duality can be leveraged in continuous time through pointwise coupling rather than time-reversal, yielding concrete modeling capabilities unavailable in previous frameworks. Once this shift in perspective is made, the desired duality properties follow as a natural consequence of the architecture. 
This system simultaneously defines: (i) the fine-grained partition, (ii) its coagulation operator, (iii) a forward-in-time system of coupled, time-homogeneous fragmentation processes in the sense of Bertoin~\cite{BerFrag}, and (iv) a dual, backward-in-time structured coalescent. All four components are governed by the same underlying compositional structure, which yields their exact compound Poisson representations. By placing these objects on a common probability space with explicit joint laws, the construction naturally extends the duality in two fundamental directions: first, to arbitrary driving subordinators, and second, to the previously uncharacterized multi-group setting.
To substantiate this claim and build our framework from first principles, this introduction will proceed as follows. First, we establish the crucial link between the mixed Poisson framework and the structure of exchangeable partitions, motivating why our approach is able to connect back to these more complex spaces (Section~\ref{sec:background}). Second, we detail the classical duality of Pitman that serves as our point of departure (Section~\ref{sec:pitman-duality}). With this context in place, we then introduce the general blueprint for our duality framework: a four-part coupled system that elevates the principles of the Poisson Hierarchical Indian Buffet Process (PHIBP)~\cite{hibp25}. We outline this blueprint and its main contributions in Section~(\ref{sec:phibp-framework}). We conclude with a discussion of the broader implications for time-dynamic processes—particularly the ability to observe coagulation and fragmentation phenomena simultaneously at every time point—and an outline of the paper (Section~\ref{sec:implications}).

To be clear, our primary contribution is Theorem~\ref{coupledsupertheorem} and its many implications. That is:

\begin{enumerate}
\item \textbf{The envisionment and construction of the four-part coupled system}. Part 1 of our core innovation is the coupled infinite sum representations of Theorem~\ref{coupledsupertheorem}, which corresponds to a specific PRM structure we term the \emph{L\'evy-It\^o coupled duality constructor}. Its architecture achieves two major goals: First, it establishes structural coupled duality frameworks across spaces of counts, mass, and integer partitions, while simultaneously constructing the continuous process-level bridges (e.g., $\Lambda$ Fleming-Viot processes) that realize these dualities dynamically. Second, and more broadly, it unlocks the extension of Pitman's duality to arbitrary subordinators and to multiple groups ($J \ge 1$). Section~\ref{sec:dualitymachine} presents a general blueprint for extensions to abstract Polish spaces using the h-biased framework in~\cite{PPY92,PY92}, where the duality is characterized through the \emph{cloud duality principle}---a geometry-versus-sampling dichotomy at the level of point clouds.

To clarify the power of this approach, consider an analogy to Lukacs's~\cite{Lukacs55CharGamma} classic characterization of the Gamma distribution. Given independent Gamma variables $X$ and $Y$, it is true that $X/(X+Y)$ is independent of $X+Y$ and hence random variables $S \overset{d}{=} X+Y$. However, this distributional statement is far weaker than the structural identity that holds on the original probability space:
\[
\frac{X}{X+Y} \times (X+Y) = X \quad \text{versus} \quad \frac{X}{X+Y} \times S \overset{d}{=} X.
\]

While many processes can be constructed to exhibit Pitman's fragmentation-coagulation duality between $\mathrm{PD}(\alpha, \theta)$ and $\mathrm{PD}(\beta, \theta)$ marginals in a distributional sense, our framework selects and generalizes the specific structural version of Pitman's duality (detailed for the classic case in Section~\ref{sec:tagged_fragmentation_example}) that preserves the duality pointwise on the underlying probability space. In the continuous-time setting, this distinction is paramount: it is precisely this structural feature that enables duality to operate in lock-step at each time point $t$. This gives rise to our concept of \textit{simultaneous structural duality}, which is crucial for a time-dynamic construction addressing points raised in Bertoin~\cite[Ch.~4, p.~213]{BerFrag}, and, more importantly, provides new continuous-time capabilities in both theoretical modelling and practical executable applications. This provides, for instance, a tractable new framework for the notoriously complex Ancestral Recombination Graph (ARG) problem, a point we elaborate on in Section~\ref{section1:Simtimefragcoag}.

\item \textbf{An explicit compound Poisson representation for this entire system}. This corresponds to Part 2 of our contribution and  provides a transparent, and arguably the only feasible, path to the calculation of joint EPPFs for the new classes of coupled processes. This delivers the explicit partition calculus—the hard formulas for the joint EPPFs—that represents a direct extension of Pitman's duality.

\item \textbf{A direct path to exact simulation and a deeper conceptual understanding}. The compound Poisson laws immediately yield algorithms for exact sampling, making these complex models practically executable. Furthermore, our mixed Poisson framework illuminates the "absolute rates" perspective—a natural yet often obscured component of species sampling models—connecting them to broader phenomena in combinatorial stochastic processes. This ability to execute complex computational inference procedures can be extrapolated from its demonstration in~\cite{hibp25} for a static microbiome analysis.
\end{enumerate}

Before embarking on this detailed development, we provide below a technical preview of the core construction and its continuous-time extension. Because our framework synthesizes ideas from several distinct areas—L\'evy processes, exchangeable partition theory, fragmentation-coalescence duality, mixed Poisson representations, and specialized tools developed in prior work, including the PHIBP framework~\cite{hibp25} that catalyzed this investigation—the full development necessarily introduces substantial background material that may initially obscure the central contributions. This preview serves to clarify what is achieved: it highlights the key architectural innovation (the four-component coupled system), demonstrates how it resolves both the static and continuous-time duality challenges, and provides a roadmap for readers navigating the technical layers that follow.

\subsection{A Four-Component Coupled System from Subordinators.}
As a preview to Theorem \ref{coupledsupertheorem}, we construct from $J+1$ independent L\'evy subordinators $(\sigma_0, \sigma_1, \dots, \sigma_J)$ a novel four-component coupled system of mixed Poisson variables for $j=1,\dots,J,$ $y\in[0,1]$ and $\gamma_{j}>0$:
\begin{equation} \label{coupledcoarsetofinevectorQUAD}
\begin{pmatrix} I_j(\gamma_j, y) \\ \mathscr{A}_j(\gamma_j, y) \\ (\mathscr{F}^{(\lambda_l)}_{j,l}(\gamma_j, y))_{l\ge 1} \\ Z_j(\gamma_j, y) \end{pmatrix}
:=
\begin{pmatrix}
\sum_{l=1}^{\infty}\sum_{k=1}^{\infty}\mathscr{P}_{j,k,l}(s_{j,k,l}\gamma_{j})\mathbb{I}_{\{U_{j,k,l}\leq y\}} \\
\sum_{l=1}^{\infty}\sum_{k=1}^{\infty}\mathbb{I}_{\{\mathscr{P}_{j,k,l}(s_{j,k,l}\gamma_{j})>0\}}\mathbb{I}_{\{Y_{l}\leq y\}} \\
\left(\sum_{k=1}^{\infty}\mathscr{P}_{j,k,l}(s_{j,k,l}\gamma_{j})\mathbb{I}_{\{U_{j,k,l}\leq y\}}\right)_{l\ge 1}\\
\sum_{l=1}^{\infty}\left[\sum_{k=1}^{\infty}\mathscr{P}_{j,k,l}(s_{j,k,l}\gamma_{j})\right]\mathbb{I}_{\{Y_{l}\leq y\}}
\end{pmatrix}
\end{equation}

where $\mathscr{P}_{j,k,l}(s_{j,k,l}\gamma_{j}) \sim \mathrm{Poisson}(s_{j,k,l}\gamma_{j})$ are independent Poisson variables and the jumps are based on the subordinator $\sigma_{0},$ up till time $t=1,$ $\sigma_{0}(1),$ acting as the main clock. More precisely we work with the multivariate process of composed subordinators $(\sigma_{j}\circ \sigma_{0}; j\in[J], \sigma_{0})$ with properties formalized in Section~\ref{sec:coupledmixedPoisson}.

Identifying this precise four-component structure is itself a central contribution of this work—there are no obvious indications a priori that this particular construction provides the necessary conceptual framework to formulate the correct duality object. Here, $I_j$ represents the fine partition (subspecies clusters), $\mathscr{A}_j$ tracks distinct blocks (species) appearing in $Z_j$, the collection $(\mathscr{F}^{(\lambda_l)}_{j,l})_{l \ge 1}$ provides independent fragmentation operators decomposing each coarse species $l$, and $Z_j$ aggregates to the coarse partition of species abundances. The construction—building these four components from the same Poisson atoms underlying the subordinators—yields coupled mass partitions and bridges satisfying Pitman-type duality on their respective measure spaces (Theorem~\ref{coupledsupertheorem}). 

Specifically, the construction is driven by the jumps of the underlying subordinators, which we interpret as a hierarchy of mean abundance rates. At the highest level, the jumps $(\lambda_l)_{l\ge 1}$ of $\sigma_0$ represent the global mean abundance rates for each potential species $Y_l$ across the meta-community. These global rates are then used to derive the local mean abundance rate for species $Y_l$ within group $j$, denoted $\sigma_{j,l}(\lambda_l)$. This local rate itself decomposes into a sum of finer-grained rates, $\sigma_{j,l}(\lambda_{l})=\sum_{k=1}^{\infty}s_{j,k,l}$, where the atoms $(s_{j,k,l})_{k\ge 1}$ are the jumps of $\sigma_j$ restricted to specific  intervals of length $\lambda_l$, and represent the abundance rates of subspecies markers $U_{j,k,l}$ associated with species $Y_l$ in group $j$.
This satisfies $\sigma_{j,l}(\lambda_{l})=\sum_{k=1}^{\infty}s_{j,k,l},$ hence  $Z_j(\gamma_j, y)=\sum_{l=1}^{\infty}\mathscr{\tilde{P}}_{j,l}(\sigma_{j,l}(\lambda_{l})\gamma_{j})\mathbb{I}_{\{Y_{l}\leq y\}}.$ Where $\mathscr{\tilde{P}}_{j,l}(\sigma_{j,l}(\lambda_{l})\gamma_{j}):=\sum_{k=1}^{\infty}\mathscr{P}_{j,k,l}(s_{j,k,l}\gamma_{j})\sim \mathrm{Poisson}(\sigma_{j,l}(\lambda_{l})\gamma_{j})$, denoting, unconditionally, a mixed Poisson distribution. 

This nested jump structure, combined with iid Uniform$[0,1]$ variables $(U_{j,k,l})$ and $(Y_{l})$, defines the complete generative system in \eqref{coupledcoarsetofinevectorQUAD}, with equivalent representations in Section~\ref{sec:coupledconstruct}. From this representation, the mean rates and species labels uniquely determine the coupled mass partitions $\mathbf{P}_j$, $\mathbf{Q}_{0}$, $\mathbf{Q}_{j,l},$ and $\mathbf{V}_j$, and their corresponding bridges, from which all duality properties follow. The explicit constructions for these mass partitions and bridges are detailed in Section \ref{sec:PKdualitylaws}, built from the representation in \eqref{coupledcoarsetofinevectorQUAD}. As demonstrated in Section \ref{sec:background}, this framework operates more generally from the thinned zero sets of the subordinators (see also Section \ref{sec:posterior_decomp}). 

The compound Poisson representation given in Theorem~\ref{coupledsupertheorem}, built upon the remarkable family of mixed truncated Poisson (MtP) distributions from \cite{James2017, hibp25, Pit97}, is the crucial mechanism that makes this blueprint analytically tractable. It is this representation, with the allocation $\mathscr{A}_j$ serving as the key conditioning variable, that unlocks the explicit joint EPPF calculations in Sections~\ref{Sec:jointEPPF}-\ref{subsec:marginalized_duality}—a task that would otherwise be intractable. This achievement provides for coupled systems the kind of explicit joint EPPFs that Pitman~\cite{Pit99Coag} established for the Poisson-Dirichlet process, thereby translating an abstract duality framework into a constructive engine for (joint) likelihood-based analysis of complex systems, with direct consequences for key problems in population genetics (such as the analysis of Ancestral Recombination Graphs~(ARG) and multi-type processes), and with clear extensions to community ecology, microbiome analysis, and other fields reliant on hierarchical, count-based rate models

\begin{rem}
An important feature of our framework: because we allow arbitrary subordinators which may take the value zero, the mixed Poisson constructs remain well-defined and robust, whereas division by zero on such events renders mass partitions and bridges undefined. This demonstrates a concrete advantage of the mixed Poisson perspective beyond mere computational convenience, as demonstrated in our development of robust microbiome models in~\cite{hibp25} that avoid the pitfalls of compositional approaches (relative rates). See Section~\ref{sec:background} for more on these relations and calculations of EPPFs. 
\end{rem}

\subsection{Simultaneous Fragmentation-Coalescence: A New Form of Continuous-Time Duality}\label{section1:Simtimefragcoag}
As shown in Section~\ref{sec:DynamicCouplin}, by replacing $\sigma_{0}(1)$ with $(\sigma_{0}(t):t\ge 0),$ and invoking Bertoin's~\cite{BerFrag} homogeneous fragmentation~(HFP) framework,  the construction elevates seamlessly to continuous time, yielding coupled $\Lambda$-Fleming-Viot processes, structured $\Lambda$-coalescents~\cite{DonnellyKurtz1999, Pit99Coag, sagitov1999}, and homogeneous fragmentation processes (HFPs). The critical innovation—achievable only through pointwise coupling of the underlying subordinators—is that Pitman-type duality is preserved simultaneously at each time $t$, not merely in distribution. This addresses the problem noted by Bertoin~\cite[Ch 4, p.213]{BerFrag} through a novel conceptual framework: unlike previous uses of the duality result in the purely stable case (\cite{BerFrag,BerLegall00,Pit99Coag,Pit06}), which exploit parameter variation (e.g., $\alpha = e^{-t}$) to encode time evolution in the context of Bolthausen-Sznitman/U-coalescent dynamics and flows of bridges, our approach treats time as an independent physical parameter, enabling continuous-time dynamics where both fragmentation and coalescence processes evolve simultaneously in real time while maintaining pointwise duality at each instant.

Consider the coupled pair $(I_\alpha(\sigma_{\beta/\alpha}(t)), Z_\beta(t))$ for $0 < \beta < \alpha < 1$. The process $I_\alpha(\sigma_{\beta/\alpha}(t))$ arises from time-changing an $\alpha$-stable Homogeneous Fragmentation Process (HFP) with an independent $\beta/\alpha$ stable subordinator. By construction, this process yields a $\mathrm{PD}(\alpha,0)$ partition at each time $t,$ and its total mass evolves as the composition $\sigma_{\alpha}(\sigma_{\beta/\alpha}(t)).$ Its dual, $Z_\beta(t)$, is constructed from the composed subordinator $\sigma_{\beta} = \sigma_{\alpha}\circ\sigma_{\beta/\alpha}$ and therefore evolves marginally as a $\beta$-stable HFP.

Let $(\Pi_{I,m}(t), \Pi_{Z,m}(t))$ denote the respective partitions of $[m]=\{1,\dots,m\}$ at time $t$. We adopt the notation where $\pi \preceq \pi'$ means “$\pi$ is a coarser partition than $\pi'$.” The coupling is constructed to satisfy, for all $t \ge 0$ and $s \ge 0$:

\begin{align} \label{eq:ordering}
  \Pi_{Z,m}(t) &\preceq \Pi_{I,m}(t) && \text{(duality: coarse-fine nesting at time $t$)} \\
  \Pi_{Z,m}(t) &\preceq \Pi_{Z,m}(t+s) && \text{(fragmentation: $Z$ forward in time)} \nonumber \\
  \Pi_{I,m}(t) &\preceq \Pi_{I,m}(t+s) && \text{(fragmentation: $I$ refining forward)} \nonumber
\end{align}

The second and third inequalities confirm that both processes are fragmentations, becoming progressively finer as time $t$ increases. The first inequality establishes the crucial duality: at any fixed moment in time, the partition of the $Z$ process is a coarsening of the partition of the $I$ process. Critically, these relationships hold in lock-step: as time progresses, both $I$ and $Z$ fragment simultaneously, yet the duality $\Pi_{Z,m}(t) \preceq \Pi_{I,m}(t)$ persists at every instant. Consequently, because $Z_\beta(t)$ is a $\beta$-stable HFP—a fragmentation process forward in time—it is equivalent to a familiar coalescent process (specifically, a $\mathrm{Beta}(2-\beta, \beta)$-coalescent) when viewed backward in time.

This duality provides a principled and coherent framework for ancestral inference, particularly relevant for population genetics. Consider a scenario where the coarse process $Z_\beta(t)$ represents the genealogy at a single genetic locus, whose time-reversal is a standard Beta-coalescent. Let the finer process $I_\alpha(\sigma_{\beta/\alpha}(t))$ represent the full Ancestral Recombination Graph (ARG)—whose modern inferential treatment owes much to the seminal work of Griffiths and Marjoram\cite{griffiths1997ancestral, hein2005gene}—a notoriously complex object from across a genetic region. Although there are methods to analyze these objects separately, their joint distribution has been difficult to characterize and exploit.

Our framework provides a tractable structure for this joint distribution. The lock-step coupling, which ensures $\Pi_{Z,m}(t) \preceq \Pi_{I,m}(t)$ at all times $t$, creates a formal conditional relationship. For instance, conditioning on a realization of the marginal genealogy $Z_\beta$ that exhibits a recent coalescence time for two lineages immediately prunes the vast state space of possible recombination histories for the ARG $I_\alpha(\sigma_{\beta/\alpha}(\cdot))$. This simplification of the conditional structure makes the joint likelihood more tractable, paving the way for more efficient simulation and parameter estimation algorithms. This ability to rigorously leverage the marginal process to constrain the full ancestral graph represents a significant advance over static or discretely-linked models. Crucially, our compound Poisson representations deliver the explicit joint EPPF formulas that enable—for the first time in this general setting—both direct parametric inference and exact simulation. Furthermore, this foundation makes it possible to adapt more sophisticated computational inference schemes to the continuous-time domain, as has been demonstrated for related static microbiome models~\cite{hibp25}.

As indicated, this simultaneous observation of dual dynamics—impossible in parameter-encoded time frameworks—enables unique modeling capabilities where both processes can be observed and utilized at each instant. Furthermore, our construct allows richer modelling capabilities using our full 4-point coupled system that can be represented as $(I_{j}, \mathscr{A}_{j}, (\mathscr{F}^{(\lambda_l)}_{j,l})_{l \ge 1}, Z_j, j\in [J])([0,t]).$ The pattern extends to arbitrary subordinators (Section~\ref{sec:DynamicCouplin}): pointwise coupling of subordinators induces pointwise coupling of the associated continuous-time partition processes, preserving duality structure at each $t$. Our approach also introduces to the continuous-time literature an effective mean-abundance–rates formulation via mixed Poisson jumps—as opposed to relative rates via mass partitions—bringing to continuous time the flexibility that PHIBP provides in static settings~\cite{hibp25}. 

\subsection{Exchangeable Partitions and the Mixed Poisson Process Perspective}\label{sec:background}
Mass partitions \(\mathbf{V} = (V_{l})_{l\ge 1}\) on the space 
$\mathcal{P}_{\infty} := \{ (p_k)_{k \ge 1} : p_1 > p_2 > \dots > 0, \sum_{k=1}^{\infty} p_k = 1 \}$
and their corresponding exchangeable random partitions of \(\mathbb{N}\) play important roles in theoretical and applied probability—notably excursion theory, Bayesian statistics, population genetics, and machine learning. Often referred to as Kingman's paint-box construction, these infinitely exchangeable random partitions are derived from conditionally independent and identically distributed (i.i.d.) sampling from a random bridge \(G(y) = \sum_{l=1}^{\infty} V_{l} \mathbb{I}_{\{Y_{l}\leq y\}}\) for \(y\in [0,1]\), where \((Y_{l})_{l \ge 1} \overset{\text{iid}}{\sim} \mathrm{U}[0,1]\) denotes a sequence of i.i.d. Uniform[0,1] random variables, independent of \(\mathbf{V}\), as articulated in the works of \cite{Aldous,BerFrag,kingman1975,Pit02,Pit06}. Here and throughout, we adopt the convention that variables denoted by \(Y\), \(U\), or their variants (e.g., \(\tilde{Y}_j\), \(U_{j,k}\)) represent i.i.d. random variables drawn from a $\mathrm{Uniform}[0,1]$ distribution, often serving as unique latent tags for species or clusters.

The law of the restrictions of such a partition to the first \(n\) integers is expressed in terms of an exchangeable partition probability function (EPPF), say \(\mathrm{p}\). For any \(n \ge 1\), let $(\hat{Y}_{1},\ldots,\hat{Y}_{n})$ denote an exchangeable sample of size $n$ from $G$. This sampling mechanism produces a partition of \([n] := \{1, \dots, n\}\) into \(K_n\) blocks, say \((N_1, \dots, N_{K_n})\). If there are \(K_n=r\) unique values $(\tilde{Y}_1, \dots, \tilde{Y}_r)$ in the first \(n\) samples from \(G\), the blocks are formed by \(N_j = \{i \in [n] : \hat{Y}_i = \tilde{Y}_j\}\). For a realization with block sizes \(|N_j| = n_j\) such that \(\sum_{j=1}^{r} n_j = n\), the joint probability is given by the EPPF:
\begin{equation}
\label{EPPF}
\mathbb{P}(\Pi_{n}=(N_{1},\ldots,N_{r}))=\mathrm{p}(n_{1},\ldots,n_{r}).
\end{equation}
By convention, blocks are typically listed according to the order of their least elements. Further details on consistency properties can be obtained from the cited works. Our approach employs a special sub-class derived from the larger family of finite Gibbs partitions, described in detail in~\citep[Chapter 1]{Pit06}. For a fixed \(n\), a general finite Gibbs EPPF, denoted by \(p^{[n]}\), where "Gibbs" implies a product structure, is connected to the work of Kolchin~\citep{Kolchin}. This connection is detailed in~\cite[Theorem 1.2, p. 25]{Pit06}. 

We now connect to our setup via the case \(J=1\), using \(Z(\gamma_{1}, y)\) from~\eqref{coupledcoarsetofinevectorQUAD}. Since a composed subordinator \(((\sigma_{1} \circ \sigma_{0})(s): 0 \leq s \leq 1)\), is simply a subordinator, \(Z\) is a familiar one-dimensional mixed Poisson process with mean rates of abundance \(z_{l} = \sigma_{1,l}(\lambda_{l})=\sum_{k=1}^{\infty}s_{1,k,l}\), where the jumps of the composed subordinator are not ranked, and each jump is associated with a species \(Y_{l}\).  Thus, we are in the setting articulated by Pitman~\cite{PitmanPoissonMix}, which relates to species sampling models of Fisher and McCloskey.  Now, assuming \(\sigma_{1}(\sigma_{0}(1)) > 0\), we may form mass partitions \(\mathbf{V} = \left(V_{l} = z_{\pi(l)}/\sigma_{1}(\sigma_{0}(1))\right)_{\{l \geq 1\}}\), where \(\pi(l)\) denotes the ranked ordering of the components. Hence, \(G\) with atoms \(Y_{l}\) corresponds to the related bridge. 

\begin{rem}
This demonstrates how our system represented in~\eqref{coupledcoarsetofinevectorQUAD} builds mass partitions such as $\mathbf{V}$ and coupled counterparts as described in~Theorem~\ref{coupledsupertheorem}.
\end{rem}

For $n > 0$, conditioning $\mathbf{V}$ or $G$ on $Z(\gamma_{1}, 1) = n$ leads to a proper conditional distribution with law denoted by $\mathbb{P}^{[n]}(\gamma_{1}),$ notationally suppressing dependence on the L\'evy density of $\sigma_{1}\circ\sigma_{0}.$ Now crucially, $Z(\gamma_{1}, y)$ admits a compound Poisson representation
$$\sum_{\ell=1}^{\varphi} \mathscr{F}^{(H_\ell)}_{1,\ell}(\gamma_{1}, 1)\mathbb{I}_{\{\tilde{Y}_{\ell}\leq y\}}$$
where $\varphi$ is Poisson-distributed, the terms are iid, and $\mathscr{F}^{(H_\ell)}_{1,\ell}(\gamma_{1}, 1)$ represents the count of species of type $\tilde{Y}_{\ell}$ at time $\gamma_{1}$. Furthermore, $\mathscr{F}^{(H_\ell)}_{1,\ell}(\gamma_{1}, 1)$ follow common mixed truncated Poisson (MtP) distributions. It follows that the conditional EPPF $p^{[n]}(n_{1},\ldots,n_{r}|\gamma_{1})$ of $\mathbf{V}|Z(\gamma_{1},1)=n$ can be expressed in terms of finite Gibbs partition blocks in exchangeable order:
\begin{equation}
\label{KolchinfiniteGibbs}
\frac{n!}{r! \prod_{l=1}^{r} n_{l}!} \, \mathrm{p}^{[n]}(n_{1}, \ldots, n_{r}|\gamma_{1}) = 
\mathbb{P}\left(
(\mathscr{F}^{(H_\ell)}_{1,\ell}(\gamma_{1}, 1) = n_{\ell};\ell \in [r]),\varphi = r \, \middle| \, 
\sum_{\ell=1}^{\varphi} \mathscr{F}^{(H_\ell)}_{1,\ell}(\gamma_{1}, 1) = n
\right).
\end{equation}
where $Z(\gamma_{1},1)\overset{d}=\mathscr{P}(\sigma_{1}(\sigma_{0}(1))\gamma_{1}),$ a mixed Poisson process at time $\gamma_{1}.$
Hence, we will work with joint laws of \((\mathscr{F}^{(H_\ell)}_{1,\ell}(\gamma_{1}, 1) = n_{\ell};\ell \in [\varphi]),\varphi \) and the corresponding law of the n-th arrival time of the mixed Poisson process say $\Gamma_{n}/\sigma_{1}(\sigma_{0}(1))$ is expressed as 
\begin{equation}
\label{jointNT}
\frac{n!}{r!\prod_{l=1}^{r}n_{l}!} \mathrm{p}^{[n]}(n_{1},\ldots,n_{r}|\gamma_{1})\mathbb{P}(\frac{\Gamma_{n}}{\sigma_{1}(\sigma_{0}(1))}\in d\gamma_{1})
\end{equation}
where,
$$
\mathbb{P}(\frac{\Gamma_{n}}{\sigma_{1}(\sigma_{0}(1))}<\infty)=\int_{0}^{\infty}\frac{n}{\gamma}\mathbb{P}(\mathscr{P}(\sigma_{1}(\sigma_{0}(1))\gamma_{1})=n)d\gamma=\mathbb{P}(\sigma_{1}(\sigma_{0}(1))>0)>0.
$$
 In the following remark we highlight how to obtain the unconditional EPPF of $\mathbf{V}.$ 

\begin{rem}[Recovery of Unconditional EPPFs]
\label{Remark:PitEPPFadjustment}
By integrating out over the distribution of 
$\frac{\Gamma_{n}}{\sigma_{1}(\sigma_{0}(1))}$ we recover the EPPF corresponding to that of the unconditional distribution of the mass partitions $\mathbf{V}$. Crucially this sets up the precise bridge between mass partitions, their infinitely exchangeable random partitions of $\mathbb{N}$, and corresponding mixed Poisson processes. We note further that we are not requiring the stronger condition of $\mathbb{P}(\sigma_{1}(\sigma_{0}(1))>0)=1$ which is the case for infinite activity subordinators that admit proper probability densities, such as the case for general Poisson-Kingman models. To allow for summation to one, i.e. a proper EPPF we simply follow Pitman's~\citep[Proposition 6.1]{PitmanPoissonMix} and divide by $\mathbb{P}(\sigma_{1}(\sigma_{0}(1))>0)$. Or more precisely for our general framework $\mathbb{P}(\sigma_{j}(\sigma_{0}(1))>0;\forall j\in[J])$.
\end{rem}

This mechanism, albeit without true recognition of its underlying principle, is clearly exhibited in~\cite[Propositions 1 and 3]{JLP2}, where the authors derive EPPFs, conditional and unconditional, corresponding to those derived in~\cite{Pit02} for Poisson-Kingman models. Those same derivations and more are discussed in~\citep[Sections~5 and 8]{James2002}, where this author attributes learning about and employing the \textit{gamma trick} in more broad contexts through the works of Pitman and Yor~\cite{PY92,PY97,PY2001} and Tsilevich, Vershik, and Yor~\cite{tsilevich2000distinguished} (see also, for example, \citep[p.~27, Eq.~(2.t)]{PPY92}). Pitman’s work, attributing this to a missing or hidden component of species sampling modelling due to a framework of Fisher/McCloskey, provides a formal interpretation in terms of conditioning on the number of animals/species/customers, n, tagged/arriving at a specified time according to a mixed Poisson process or, equivalently, conditioning on the corresponding \(n\)-th arrival time. This results in the law \(\mathbb{P}^{[n]}\) as employed formally with Pitman’s interpretation in recent works such as~\citep{JamesStick} with many pertinent details for our present setting, \citep[Section~2.2.1]{HJL2}, and \citep[Section~6]{HoJamesLau2025}. These principles have been a key influence on the centerpiece of the present work on coagulation and fragmentation, the Poisson-Hierarchical Indian Buffet Process (PHIBP)~\cite{hibp25}; in the development of the PHIBP, however, the influence stems more from the overarching interpretations in~\cite{PitmanPoissonMix} rather than from the explicit use of the laws \(\mathbb{P}^{[n]}\). We now can  give a full overview of the implications of our results in regard to identification of pieces for associated structured  coupled process constructions exhibiting duality, and transparent calculations of joint partition distributions, which are otherwise opaque from other perspectives.

\begin{rem}[Architectural Strategy]\label{archstrat}
The construction of the four-component coupled mixed Poisson system addresses a central challenge in this work. As demonstrated in Section~\ref{sec:background}, this system, with one representation in \eqref{coupledcoarsetofinevectorQUAD}, provides the representation linking our coupled construct to mass partitions, bridges, and other quantities on their respective spaces that are induced to exhibit a structural form of Pitman-type duality over $J \ge 1$ groups. A crucial and subtle aspect of this architecture is its inherently dynamic nature. This is not merely a static correspondence between partition structures, but a structural coupling that is realized simultaneously in continuous time, as detailed in Section~\ref{sec:DynamicCouplin}. This dynamic perspective is essential, as the general form of this duality cannot be recovered through classical time-reversal arguments. However, from the perspective of joint partition calculus, this correspondence establishes a 'soft duality'—a conceptual linkage not yet enabling concrete calculations of joint EPPFs. The pathway to explicit computation is made accessible through the coupled compound Poisson representation in Theorem~\ref{coupledsupertheorem}, which enables the derivations in Sections~\ref{sec:posterior_decomp} through \ref{sec:finetocoarse}. These leverage analytical results from \cite{hibp25} concerning the coarse partition and fragmentation operators, culminating in the joint EPPF computations of Sections~\ref{Sec:jointEPPF} and \ref{subsec:marginalized_duality}. See, for instance, Proposition~\ref{prop:conditional_law_of_fragments} for the more general extension of \eqref{KolchinfiniteGibbs}, which is necessarily more complex but obeys the same principles.
\end{rem}

\subsection{Duality for the Two-Parameter Poisson-Dirichlet Distribution}\label{sec:pitman-duality}
The two-parameter Poisson-Dirichlet distribution $\mathrm{PD}(\beta,\theta)$ for $0\leq \beta<1, \theta>-\beta$ is the most distinguished class of laws on $\mathcal{P}_{\infty}$. Its origins lie in excursion theory, developed primarily by Pitman and Yor~\cite{BPY,PY92,PY97} and Perman, Pitman and Yor~\cite{PPY92}, who derived the tractable $\mathrm{GEM}(\beta,\theta)$ distributions via size-biased sampling. This class has gained considerable importance across fields~\cite{CraneEwens}. Ishwaran and James~\citep{IJ2001, IJ2003} termed the framework the ``Pitman-Yor process,'' denoted $G_{\beta,\theta}$, which saw prolific usage in Bayesian statistics and machine learning. The special case $G_{0,\theta}$ is the Dirichlet process~\cite{Ferg1973}, with underlying mass partition $\mathbf{V} \sim \mathrm{PD}(0,\theta)$, central to population genetics~\cite{Pit96, TavareEwens}.

The EPPF from sampling $G_{\beta,0}$ for a partition of $[n]$ into $K^{[\beta]}_{n}=r$ blocks is
\begin{equation}
p_{\beta}(n_1,\ldots,n_{r}):=\frac{\beta^{r-1}\Gamma(r)}{\Gamma(n)}\prod_{j=1}^{r}(1-\beta)_{n_{j}-1},
\label{canonEPPF1}
\end{equation}
where $(x)_n = {\Gamma(x+n)}/{\Gamma(x)}$ is the Pochhammer symbol. For $G_{\beta,\theta}$, the EPPF is
\begin{equation*}
p_{\beta,\theta}(n_1,\ldots,n_{r})=\frac{\Gamma(n)\Gamma(\frac{\theta}{\beta}+r)\Gamma(\theta+1)}{\Gamma(r)\Gamma(\frac{\theta}{\beta}+1)\Gamma(\theta+n)}p_{\beta}(n_1,\ldots,n_{r}).
\end{equation*}
Furthermore, $\mathbb{P}_{\beta,0}(K^{[\beta]}_{n}=k)=\mathbb{P}^{(n)}_{\beta}(k)=\mathbb{P}_{\beta,0}^{(n)}(k)=\frac{\beta^{k-1}\Gamma(k)}{\Gamma(n)} S_\beta(n,k)$, with $S_\beta(n,r) =  \frac{1}{\beta^r k!} \sum_{j=1}^r (-1)^j \binom{r}{j} (-j\beta)_n$ denoting the generalized Stirling number, as described in~\cite[eq(3.19) p. 66]{Pit06}.

Pitman's~\cite{Pit99Coag} celebrated coagulation-fragmentation duality connects partitions governed by $\mathrm{PD}(\alpha,\theta)$ and $\mathrm{PD}(\beta,\theta)$ for $0 \leq \beta < \alpha < 1$. Given $\mathbf{V} \sim \mathrm{PD}(\beta,\theta)$, fragmenting each mass via i.i.d.\ $\mathbf{Q}^{(l)} \sim \mathrm{PD}(\alpha, -\beta)$ yields $\mathbf{P} = \mathrm{Frag}_{\alpha,-\beta}(\mathbf{V}) \sim \mathrm{PD}(\alpha,\theta)$. Conversely, merging $\mathbf{P} \sim \mathrm{PD}(\alpha,\theta)$ via $\mathbf{Q}_{0} \sim \mathrm{PD}(\frac{\beta}{\alpha}, \frac{\theta}{\alpha})$ yields $\mathbf{V} = \mathrm{Coag}_{\beta/\alpha, \theta/\alpha}(\mathbf{P}) \sim \mathrm{PD}(\beta,\theta)$.

While laws on $\mathcal{P}_{\infty}$ are informative, Kingman's correspondence provides a more tangible description via EPPFs on integer partitions of $[n]$. 

\begin{setup}[Canonical Partition Duality Framework]
\label{setup:canonical_duality}

\textbf{Coagulation:} Start with $\pi^{(1)}=(C_1, \ldots, C_K)$ from $p_{\alpha,\theta}$ with counts $(c_{1},\ldots,c_{K})$. Draw partition $(X_1, \ldots, X_r)$ of $[K]$ from $p_{\beta/\alpha, \theta/\alpha}$ with counts $(x_{1},\ldots,x_{r})$. Merge blocks: $N_l = \bigcup_{j \in X_l} C_j$. Result: $\pi^{(2)}=(N_1, \ldots, N_r)$ follows $p_{\beta,\theta}$.

\textbf{Fragmentation:} Start with $\pi^{(2)}=(N_1, \ldots, N_r)$ from $p_{\beta,\theta}$. Fragment each $N_l$ via $p_{\alpha,-\beta}$ into sub-blocks $(C_{l,1}, \ldots, C_{l,x_{l}})$. Result follows $p_{\alpha,\theta}$.

The duality~\citep[Theorem 12]{Pit99Coag} is:
\begin{equation}
\label{Pitmanduality}
    p_{\frac{\beta}{\alpha},\frac{\theta}{\alpha}}(x_{1},\ldots,x_{r}) \times p_{\alpha,\theta}(c_{1},\ldots,c_{K})
= \prod_{l=1}^{r} p_{\alpha,-\beta}(c_{l,1},\ldots,c_{l,x_{l}})\times p_{\beta,\theta}(n_{1},\ldots,n_{r}).
\end{equation}
\end{setup}

For bridges, coagulation is subordination:
\begin{equation}
G_{\beta,\theta}(y) \overset{d}{=} F_{\alpha,\theta}\left(G_{\frac{\beta}{\alpha},\frac{\theta}{\alpha}}(y)\right), \quad y \in [0,1],
\label{PYcoagdual}
\end{equation}
with fragmentation as~\citep[Chapters 4--5]{BerFrag}:
\begin{equation}
G_{\alpha,\theta}(y) \overset{d}{=} \sum_{l=1}^{\infty} V_{l}F^{(l)}_{\alpha,-\beta}(y),
\label{PYfragdual}
\end{equation}
where $(F^{(l)}_{\alpha,-\beta})_{l\ge 1}$ are i.i.d.\ $\mathrm{PD}(\alpha,-\beta)$ bridges independent of $\mathbf{V} \sim \mathrm{PD}(\beta,\theta)$.

As noted in~\citep[Section 6]{Pit99Coag}, these descriptions offer practically no clues for extensions. The required independence is an artifact of the stable-beta-gamma algebra~\cite{PY97}.

For the PHIBP case ($J=1$), we obtain tractable joint EPPFs:
$p_{\text{coag}}\times p_{\text{fine}}=p_{\text{frag}}\times p_{\text{coarse}}$,
maintaining the same notation and constructions.

\subsection{A General Blueprint for Duality: The PHIBP Framework}\label{sec:phibp-framework}
We extend Pitman's~\cite{Pit99Coag} coagulation-fragmentation duality for coupled mass partitions $(\mathbf{V},\mathbf{P})$ following $(\mathrm{PD}(\beta,\theta),\mathrm{PD}(\alpha,\theta))$ to arbitrary subordinators and multi-group settings ($J \geq 1$). Our results apply to Poisson-Kingman distributions~\citep{BerFrag,JLP2,kingman1975,PPY92,Pit02,Pit06} and more generally (Remark~\ref{Remark:PitEPPFadjustment}). The central innovation reveals a dynamic perspective latent within the PHIBP framework developed in~\cite{hibp25}. While that work established PHIBP as a species sampling model for hierarchical count data, we show its underlying random measures implicitly define a fully coupled dynamic system. The key is the Allocation process, which provides the conditioning variable connecting fine and coarse partitions to their dual operators. We formalize this structure in Theorem~\ref{coupledsupertheorem}, elevating PHIBP from a law on partitions to a four-component system: fine partition, coagulation operator, forward-time fragmentation process, and backward-time structured coalescent.
This provides a complete blueprint. For $J=1$, we obtain explicit analogues of Pitman's $\mathrm{PD}(\alpha,\theta)$ duality for arbitrary subordinators. For $J>1$, we enter new territory—no precedent exists even for the PD case. The duality is governed by conditionally independent operators in the Fine-to-Coarse generative framework (the "Poissonized world"), from which complex marginal dependencies of Coarse-to-Fine laws emerge through integration.

\subsection{Broader Implications and Outline}\label{sec:implications}

The static duality laws at $t=1$ are marginals of a new class of continuous-time processes. The four components extend to coupled, time-homogeneous fragmentations~\citep{BerFrag} preserving duality at every instance, dually generating structured coalescents with simultaneous across-group mergers. For population genetics, this induces enriched simultaneous structual duality for foundational $\Lambda$-coalescents~\citep{Johnston2023, Pit99Coag} and subdivided population models~\citep{mohle2024multi, Taylor2009}, while extending to correlated merger-migration events from punctuated demographic histories~\citep{PraEtheridge2025Deme}. For microbiomes, it provides temporal dynamics of species assembly across coupled communities. To achieve tractable duals comparable to~\cite{Pit99Coag}, we identify four coupled components, as seen in one form in ~\eqref{coupledcoarsetofinevectorQUAD}, guaranteeing structural and distributional duality and provide explicit marginals and conditional distributions. Independence seen classically does not hold generally.

\begin{rem}
The core contribution is the coupled construction itself: building a unified probability space supporting all four components simultaneously. 
This resolves long-standing obstacles to extending Pitman's duality. See~ Remarks~\ref{rem:nontrivial} and \ref{rem:Architect} for further comments on challenges.  Section~\ref{sec:DynamicCouplin} shows this lifts to continuous time, where the structural coupling is crucial.  
\end{rem}

\subsection{Outline}

The paper's argument unfolds as follows. Section~\ref{sec:coupledmixedPoisson} establishes the architectural core: Theorem~\ref{coupledsupertheorem} constructs the four-component coupled system and provides its unified compound Poisson representation. This representation enables the explicit derivations of the coagulation, fragmentation, fine, and coarse EPPF components (Sections~\ref{sec:posterior_decomp}--\ref{sec:finetocoarse}), culminating in the general joint EPPF duality formulas (Sections~\ref{Sec:jointEPPF}--\ref{subsec:marginalized_duality}). Section~\ref{sec:section:PK} contextualizes these components by linking them to the family of Poisson-Kingman distributions and formally identifying the conditional $\mathbb{P}^{[n]}$ laws of the mass partitions. Section~\ref{sec:general_gibbs_duality} revisits our recent work~\cite{HoJamesLau2025}, which revealed that the celebrated independence in the classical Poisson-Dirichlet duality is a fragile exception. For general Gibbs partitions, a change of measure exposes a fundamental structural dependence. Building on this insight, we complete the unfinished business of~\cite{HoJamesLau2025} by explicitly characterizing the coagulation operator, establishing a true duality on par with~\cite{Pit99Coag} for this large class. Section~\ref{sec:tagged_fragmentation_example} provides crucial cross-validation by re-deriving stable-case duality using the coupled Poissonian construction, revealing structural properties like fragmentation invariance that are opaque from marginal viewpoints. It identifies the Poissonized form as essential for lifting to the dynamic setting. Section~\ref{sec:DynamicCouplin} realizes the framework's full dynamic potential by lifting to continuous time, yielding coupled processes---$\Lambda$-Fleming-Viot processes, structured $\Lambda$-coalescents, and homogeneous fragmentations---that exhibit simultaneous duality at every time point and extend to multi-group $J$-spider models, formally addressing Bertoin's commentary on Pitman-type duality in continuous time. Finally, Section~\ref{sec:dualitymachine} develops the $h$-biased L\'evy-It\^o coupled duality constructor, extending our framework to abstract Polish spaces. This articulates the \emph{cloud duality principle}---showing that duality arises from point-process regrouping and the geometry-versus-sampling dichotomy, rather than features specific to interval partitions.

\section{Four Part Coupled PHIBP Duality}\label{sec:coupledmixedPoisson}
We introduce the stochastic components underlying our framework, constructed from subordinators, Poisson-Kingman distributions, and mixed Poisson processes connecting to finite Gibbs EPPFs. We define our specific versions here; see~\cite{BerFrag, Pit97, Pit02, Pit06} for broader context and~\cite{HJL2, JamesStick, PitmanPoissonMix} for related distributions. Our construction uses $J+1$ independent subordinators: $(\sigma_j : j \in \{1, \dots, J\})$ for group-specific effects and $\sigma_0$ as a common tethering mechanism, each with associated Poisson processes. For $j \in \{0, 1, \dots, J\}$, the subordinator $\sigma_j = (\sigma_j(t) : t \ge 0)$ is a L\'evy process with no drift or Gaussian component, characterized by Poisson random measure $(s_{j,k}, U_{j,k})_{k \ge 1}$ on $(0, \infty) \times [0, 1]$ with mean measure $\tau_j(s) ds \, \U(du)$, where $\tau_j$ is its L\'evy measure. The ranked jumps $(s_{j,k})_{k \ge 1}$ satisfy $\sigma_j(1) = \sum_{k=1}^{\infty} s_{j,k}$, with Laplace transform
$$
\E[e^{-\gamma_j \sigma_j(t)}] = e^{-t\psi_j(\gamma_j)}, \quad \psi_j(\gamma_j) = \int_0^\infty (1 - e^{-\gamma_j s}) \tau_j(s) ds,
$$
where $\int_0^\infty \min(s, 1) \tau_j(s) ds < \infty$. The exponential cumulants are
$$
\psi_j^{(c)}(\gamma) = \int_0^\infty s^{c-1} e^{-s\gamma} \tau_j(s) ds.
$$

For each $j$, we construct iid pairs $(C_{j,k}, S_{j,k})$ with joint distribution
\begin{equation}
\label{MtPsimple}
\frac{s^{c_{j,k}}e^{-s\gamma_j}\tau_{j}(s)}{\psi^{(c_{j,k})}_{j}(\gamma_j)} \times \frac{(\gamma_j)^{c_{j,k}}\psi^{(c_{j,k})}_{j}(\gamma_j)}{\psi_{j}(\gamma_j)c_{j,k}!},
\end{equation}
denoted $C_{j,k} \sim \mathrm{MtP}(\tau_{j},\gamma_j)$. This admits the conditional construction $C_{j,k} | S_{j,k} = s \sim \mathrm{tP}(s\gamma_{j})$ (zero-truncated Poisson) with $S_{j,k}$ having marginal density $f_{S_{j}}(s)=\frac{(1-e^{-s\gamma_{j}})\tau_{j}(s)}{\psi_{j}(\gamma_{j})}$. The combinatorial structure connects to Bell polynomials via
\[
\psi_{j}(\gamma_{j}) = \sum_{k=1}^{\infty}\frac{\gamma^{k}_{j}\psi^{(k)}_{j}(\gamma_{j})}{k!},
\]
with exponentially tilted moments
\begin{equation}
\label{expmomentid}
\E[\sigma_{j}(\lambda)^{n} e^{-\gamma_{j} \sigma_{j}(\lambda)}] = e^{-\lambda\psi_{j}(\gamma_{j})} \Xi^{[n]}(\lambda\tau_{j}, \gamma_{j}),
\end{equation}
where $\Xi^{[n]}(\lambda\tau_{j}, \gamma_{j}) = \sum_{r=1}^{n} \Xi^{[n]}_{r}(\lambda\tau_{j}, \gamma_{j})$ and
\begin{equation}
\label{xisumrepbasic}
\Xi^{[n]}_{r}(\lambda\tau_{j}, \gamma_{j}) := \frac{n!\lambda^{r}}{r!} \sum_{(n_1, \ldots, n_r)}\prod_{l=1}^{r} \frac{\psi^{(n_{l})}_{j}(\gamma_{j})}{n_{l}!}=\lambda^{r}\Xi^{[n]}_{r}(\tau_{j}, \gamma_{j}),
\end{equation}
summing over ordered $r$-tuples $(n_1, \ldots, n_r)$ with $\sum_{l=1}^{r} n_l = n$.

For the tethering mechanism $j=0$, we denote the Laplace exponent by $\Psi_0$ (cumulants $\Psi_0^{(c)}$) and label Poisson measure points as $(\lambda_l, Y_l)_{l \ge 1}$ where $\sigma_0(1) = \sum_{l=1}^\infty \lambda_l$ and $(Y_l)$ are iid Uniform$[0,1]$. Setting $\gamma_{0}=\sum_{j=1}^{J}\psi_{j}(\gamma_{j})$, we highlight iid pairs $(H_{l},\tilde{X}_{l})$ with joint distribution
\begin{equation}
\label{HXdecomp}
\frac{\lambda^{x_{l}}e^{-\lambda\gamma_0}\tau_{0}(\lambda)}{\Psi^{(x_{l})}_{0}(\gamma_0)}\times \frac{\gamma_0^{x_{l}}\Psi^{(x_{l})}_{0}(\gamma_0)}
{x_{l}!\Psi_{0}(\gamma_0)},
\end{equation} 
where $\tilde{X}_{l}\sim \mathrm{MtP}(\tau_{0},\gamma_0)$ and $H_{l}$ has marginal density
\begin{equation}
\label{denH}
 f_{H}(\lambda) := \frac{(1-e^{-\lambda\gamma_0})\tau_{0}(\lambda)}{\Psi_{0}(\gamma_0)}.
\end{equation}

We employ these subordinators via two constructions. First, composing subordinators as in~\cite{hibp25}: the collection $(\sigma_{j,l}(\lambda_l))_{j \in [J]}$ from jumps $(\lambda_l)_{l \ge 1}$ of $\sigma_0$ forms Poisson random measure points on $\mathbb{R}_+^J \times \mathbb{R}_+$ with mean measure~\citep[Theorem 30.1]{Sato2013}
\begin{equation}
\label{eq:joint_measure_def}
\nu(d\mathbf{t}, d\lambda) = \left( \prod_{j=1}^J \mathbb{P}(\sigma_j(\lambda) \in dt_j) \right) \tau_0(\lambda) d\lambda,
\end{equation}
where
$$
\Psi_{0}(\gamma_0) = \int_0^\infty \left(1 - \mathbb{E}\left[\exp\left(-\sum_{j=1}^J \gamma_j \sigma_j(\lambda)\right)\right]\right) \tau_0(\lambda) d\lambda.
$$
For multi-index $\vv{n} = (n_1, \dots, n_J)$ and $\vv{\gamma} = (\gamma_1, \dots, \gamma_J)$, the joint exponential cumulant is
\begin{equation}
\label{jointexpcumulants}
\left(\Psi_{0}\circ \sum_{j=1}^{J}\psi_{j}\right)^{(\vv{n})}(\vv{\gamma})
:= \int_{0}^{\infty} \mathbb{E}\left[ \left(\prod_{j=1}^{J}[\sigma_{j}(\lambda)]^{n_{j}}\right)e^{-\sum_{j=1}^{J}\sigma_{j}(\lambda)\gamma_{j}} \right] \tau_{0}(\lambda)d\lambda.
\end{equation}

Following~\cite{PY97}[Proposition 33], we decompose the joint measure of $((\sigma_{j,l}(\lambda_l))_{j \in [J]}, \lambda_{l})_{l\ge 1}$, factoring \eqref{eq:joint_measure_def} as
\begin{equation}
    \label{eq:joint_measure_decomp}
    \nu(d\mathbf{t}, d\lambda) = f_{\pi}(\lambda|\mathbf{t}, \tau_{0})d\lambda \, \Lambda_{[J]}(dt_{1},\ldots,dt_{J}),
\end{equation}
where $\Lambda_{[J]}$ is the joint L\'evy measure of ranked jumps $((\sigma_{j,\pi(l)}(\lambda_{\pi(l)}))_{l\ge 1})_{j\in[J]}$ of $(\sigma_{j}\circ\sigma_{0}; j\in[J])$:
\begin{equation}
    \label{eq:levy_measure_def}
    \Lambda_{[J]}(dt_{1},\ldots,dt_{J})=\int_{0}^{\infty}\left(\prod_{j=1}^{J}\mathbb{P}(\sigma_{j}(\lambda)\in dt_{j})\right)\tau_{0}(\lambda)d\lambda,
\end{equation}
with concomitant jumps $(\lambda_{\pi(l)})_{l\ge 1}$ conditionally iid with density
\begin{equation}
    \label{eq:conditional_density_f_pi}
    f_{\pi}(\lambda|\mathbf{t}, \tau_{0})=\frac{\left(\prod_{j=1}^{J}\mathbb{P}(\sigma_{j}(\lambda)\in dt_{j})\right)\tau_{0}(\lambda)}{\Lambda_{[J]}(dt_{1},\ldots,dt_{J})}.
\end{equation}

\begin{rem} 
This evokes~\citep[Proposition 33]{PY97}, where, for $\theta>0,$ it establishes a sub-case $\mathrm{PD}(0,\theta)$-$\mathrm{PD}(\alpha,\theta)$ of Pitman's duality via gamma/generalized gamma subordinators which is otherwise represented in~\cite[Propositions 21 and 22]{PY97}. While the framework extends to general subordinators, tractability in~\citep{PY97} relies on stable beta-gamma algebra. For general subordinators, this yields a 'soft result' lacking explicit conditional laws for EPPFs; see, however, Remark~\ref{rem:stablefFragdist} for the pure stable case $\mathrm{PD}(\beta,0)$-$\mathrm{PD}(\alpha,0)$.
\end{rem}

By independent increments, $\sigma_{j,l}(\lambda_l) \overset{d}{=} \sigma_j(\lambda_l)$ has L\'evy density $\lambda_l \tau_j(s)ds$ with jumps
$$
\sigma_{j,l}(\lambda_l) = \sum_{k=1}^{\infty} s_{j,k,l}.
$$
The complete jump set $(s_{j,k})_{k\ge 1}$ is the union over $l$ of $(s_{j,k,l})_{k \ge 1}$. Alternatively, conditioning on $\sigma_0(1) = b$, the time-changed processes $(\sigma_j(b \cdot y) : y \in [0,1])$ are independent subordinators with L\'evy measure $b \cdot \tau_j(s)ds$, yielding the same $(s_{j,k})_{k\ge 1}$.

\subsection{Coupled Mixed Poisson Processes in the PHIBP Framework}\label{sec:coupledconstruct}
We introduce a novel set of coupled, vector-valued processes, indexed by $j=1, \dots, J$, which admit two distinct but related representations. The first, which we term the "fine-to-coarse" view, is the collection of processes $(I_j, \mathscr{A}_j, Z_j)_{j=1,\dots,J}$. The second, the "coarse-to-fine" view, is given by $(I_{j}, (\mathscr{F}^{(\lambda_l)}_{j,l})_{l \ge 1}, Z_j)_{j=1,\dots,J}$. For the fine-to-coarse coupling we consider, first given $\sigma_{0}$, let $(s_{j,k},w_{j,k}, U_{j,k})$ be the points of a Poisson random measure whose intensity depends on $\sigma_0,$ specifically it has mean measure $\tau_{j}(s)ds\sigma_{0}(dw)\mathbb{U}(du)$ on the product space $(0, \infty) \times [0, 1] \times [0, 1]$.
 We first introduce infinite collections of i.i.d. standard Poisson processes, which we denote by  $((\mathscr{P}_{j,k})_{k \geq 1})_{j \in [J]}$. Using the points $s_{j,k}$, we define random variables  $\mathscr{P}_{j,k}(s_{j,k}\gamma_{j}) \sim \mathrm{Poisson}(s_{j,k}\gamma_{j})$, which represent the  number of sub-species of type ${j,k}$ sampled up till time $\gamma_{j}$ with mean rate $s_{j,k}$. Foreshadowing our lifting of results to Bertoin's~\citep{BerFrag} continuous-time homogeneous fragmentation setting, we may view these processes as living in spaces of counting processes.

The variables $(\mathscr{A}_j, Z_j)_{j \in [J]}$ appear explicitly in~\citep{hibp25} and constitute the transparent part of the PHIBP, where $(\mathscr{A}_j)_{j \in [J]}$ plays an important role as the allocation process. The other variables mentioned are implicitly contained therein. The idea of coupling here is novel and allows us to put all quantities of interest on the same probability space. More importantly, it  establishes the coag/frag duality for a vast array of processes with implications for processes constructed from its basic structures. Results equivalent in distribution can also be achieved by decoupled versions, but one loses the fine interpretations demonstrated here, at the further expense of more complex proofs for weaker constructions. Crucially, without coupling, there is no possibility of elevation to meaningful continuous time  joint constructs. 

Now for $\gamma_{j}>0,$ and $y \in [0,1]$ construct
$$
(I_{j}(\gamma_{j},y),Z_{j}(\gamma_{j},y))=(\sum_{k=1}^{\infty}\mathscr{P}_{j,k}(s_{j,k}\gamma_{j})\mathbb{I}_{\{U_{j,k}\leq y\}},\sum_{k=1}^{\infty}\mathscr{P}_{j,k}(s_{j,k}\gamma_{j})\mathbb{I}_{\{w_{j,k}\leq y\}}) 
$$
Construct $(\mathscr{A}_{j}, j\in[J])$ as Bernoulli processes via: 
$$
\mathscr{B}_{j,k}(1-e^{-s_{j,k}\gamma_{j}}):=\mathbb{I}_{\{\mathscr{P}_{j,k}(s_{j,k}\gamma_{j})>0\}}
$$
and construct the coupled Allocation processes
\begin{equation}
\label{AllocationprocessBernoullirep}
\mathscr{A}_{j}(\gamma_{j},y)=\sum_{k=1}^{\infty}\mathscr{B}_{j,k}(1-e^{-s_{j,k}\gamma_{j}})\mathbb{I}_{\{w_{j,k}\leq y\}}=\sum_{k=1}^{\infty}\mathbb{I}_{\{\mathscr{P}_{j,k}(s_{j,k}\gamma_{j})>0\}}\mathbb{I}_{\{w_{j,k}\leq y\}}
\end{equation}

Using this,the coupled fine-to-coarse process vector $(I_j, \mathscr{A}_j, Z_j)$ is defined for each group $j \in {1, \dots, J}$, for any total time horizon $\gamma_j > 0$, and for all $y \in[0,1]$ as:

\begin{equation} \label{coupledfinetocoarsevector}
\begin{pmatrix} I_j(\gamma_j, y) \\ \mathscr{A}_j(\gamma_j, y) \\ Z_j(\gamma_j, y) \end{pmatrix}
:=
\begin{pmatrix}
\sum_{k=1}^{\infty}\mathscr{P}_{j,k}(s_{j,k}\gamma_{j})\mathbb{I}_{\{U_{j,k}\leq y\}} \\
\sum_{k=1}^{\infty}\mathbb{I}_{\{\mathscr{P}_{j,k}(s_{j,k}\gamma_{j})>0\}}\mathbb{I}_{\{w_{j,k}\leq y\}} \\
\sum_{k=1}^{\infty}\mathscr{P}_{j,k}(s_{j,k}\gamma_{j})\mathbb{I}_{\{w_{j,k}\leq y\}}
\end{pmatrix}
\end{equation}

We can now re-index the collection of processes $\left( \mathscr{P}_{j,k}(s_{j,k}\gamma_{j}) \right)_{k \ge 1}$ by sorting them according to the jumps $(\lambda_l)_{l \ge 1}$ of the process $\sigma_0$. This yields the equivalent, doubly-indexed representation:

\begin{equation} \label{eq:doubly_indexed_representation}
\left( \left( \mathscr{P}_{j,k,l}(s_{j,k,l}\gamma_{j}) \right)_{k \ge 1} \right)_{l \ge 1}
\end{equation}

Now we construct for each $(j,l)$ the Fragmenting mixed Poisson process,
\begin{equation}
\label{eq:FragmentingProcess}
\mathscr{F}^{(\lambda_l)}_{j,l}(\gamma_j, y) := \sum_{k=1}^{\infty} \mathscr{P}_{j,k,l}(s_{j,k,l}\gamma_{j}) \mathbb{I}_{\{U_{j,k,l}\leq y\}},
\end{equation}
such that 
\begin{equation}
I_j(\gamma_j, y) = \sum_{l=1}^{\infty} \mathscr{F}^{(\lambda_l)}_{j,l}(\gamma_j, y) = \sum_{l=1}^{\infty} \left[ \sum_{k=1}^{\infty} \mathscr{P}_{j,k,l}(s_{j,k,l}\gamma_{j}) \mathbb{I}_{\{U_{j,k,l}\leq y\}} \right]
\end{equation}
and 
\begin{equation}
Z_{j}(\gamma_j, y) = \sum_{l=1}^{\infty} \left[ \sum_{k=1}^{\infty} \mathscr{P}_{j,k,l}(s_{j,k,l}\gamma_{j}) \right] \mathbb{I}_{\{Y_{l}\leq y\}} = \sum_{l=1}^{\infty} \mathscr{F}^{(\lambda_l)}_{j,l}(\gamma_j, 1) \mathbb{I}_{\{Y_{l}\leq y\}},
\end{equation}
where $\mathscr{F}^{(\lambda_l)}_{j,l}(\gamma_j, 1) = \sum_{k=1}^{\infty} \mathscr{P}_{j,k,l}(s_{j,k,l}\gamma_{j})
=\mathscr{\tilde{P}}_{j,l}(\sigma_{j,l}(\lambda_{l})\gamma_{j}).$

These constructions lead to our coupled coarse to fine  vector $(I_{j}, (\mathscr{F}^{(\lambda_l)}_{j,l})_{l \ge 1}, Z_j)$ is defined for each group $j \in {1, \dots, J}$, for any total time horizon $\gamma_j > 0$, and for all $y \in[0,1]$ as:

\begin{equation} \label{coupledcoarsetofinevector}
\begin{pmatrix} I_j(\gamma_j, y) \\ (\mathscr{F}^{(\lambda_l)}_{j,l}(\gamma_j, y))_{l\ge 1} \\ Z_j(\gamma_j, y) \end{pmatrix}
:=
\begin{pmatrix}
\sum_{l=1}^{\infty}\sum_{k=1}^{\infty}\mathscr{P}_{j,k,l}(s_{j,k,l}\gamma_{j})\mathbb{I}_{\{U_{j,k,l}\leq y\}} \\
(\sum_{k=1}^{\infty}\mathscr{P}_{j,k,l}(s_{j,k,l}\gamma_{j})\mathbb{I}_{\{U_{j,k,l}\leq y\}})_{l\ge 1}\\
\sum_{l=1}^{\infty}[\sum_{k=1}^{\infty}\mathscr{P}_{j,k,l}(s_{j,k,l}\gamma_{j})]\mathbb{I}_{\{Y_{l}\leq y\}}
\end{pmatrix}
\end{equation}

Note as in~\citep{hibp25} these represent data processes and have as we shall show next highly tractables representations in terms of compound Poisson distributions. Before we do that we extend the univariate MtP variable to multivariate variables otherwise appearing in~\citep[Propositiom 3.2]{hibp25}
That is for each $\tilde{X}_{l}\sim \mathrm{MtP}(\tau_{0},\sum_{j=1}^{J}\psi_{j}(\gamma_{j)},$ there is a multivariate vector $(X_{j,l}, j\in[J])$ such  that $\tilde{X}_{l}=\sum_{j=1}^{J}X_{j,l}=1,2,\ldots$ and  $(X_{j,l}, j\in[J])\mid \tilde{X}_{l}=x{_l}$ has Multinomial distribution:
\begin{equation}
\label{multigroupX}
(X_{1,l}, \dots, X_{J,l}) \mid \tilde{X}_{l}=x_l \sim \mathrm{Multi}\left(x_l; q_{1},\ldots, q_{J}\right),      
\end{equation}
where the probabilities $q_{j} = \frac{\psi_{j}(\gamma_{j})}{\sum_{v=1}^{J}\psi_{v}(\gamma_{v})}$ represent the relative sampling effort for each group.
We also import the Poisson variable from that work,  $\varphi \sim \mathrm{Poisson}\left(\Psi_{0}\left(\sum_{j=1}^{J}\psi_{j}(\gamma_{j})\right)\right)$ with pmf
\begin{equation}
\label{Poissonvphi}
\mathbb{P}(\varphi=r)=\frac{{[\Psi_{0}(\sum_{j=1}^{J}\psi_{j}(\gamma_{j}))]}^{r}{\mbox e}^{-\Psi_{0}(\sum_{j=1}^{J}\psi_{j}(\gamma_{j}))}}{r!}
\end{equation}
which in that context is the total number of distinct species observed across all $J$ groups, that is for $\varphi=r$ there are $r$ iid species tags denoted $(\tilde{Y}_{1},\ldots,\tilde{Y}_{r}).$ 

Given our explicit constructions above, we now state what constitutes the cornerstone of this work: the existence and complete characterization of a four-component coupled mixed Poisson process system. This construction is far from trivial—it represents the resolution to a two-decade challenge of extending Pitman's duality beyond its canonical setting. Its computational tractability via explicit compound Poisson representations, and its role as the generative engine for our entire duality framework, are precisely what have remained elusive in previous approaches. We state clearly the major implications that follow from this construction:

\begin{thm}[Unified Compound Poisson Representation]
\label{coupledsupertheorem}
The explicit constructions previously introduced in the fine-to-coarse vector~\eqref{coupledfinetocoarsevector} and the coarse-to-fine vector~\eqref{coupledcoarsetofinevector}, establish the existence of a four-component coupled mixed Poisson process system $(I_{j}, \mathscr{A}_{j}, (\mathscr{F}^{(\lambda_l)}_{j,l})_{l \ge 1}, Z_j)_{\{j=1,\dots,J\}}$ with explicit compound Poisson structure.

For each group $j \in \{1, \dots, J\}$, the unified representation simultaneously encodes all four components through their compound Poisson structure:

\begin{equation} \label{eq:unified_4_tier_representation}
\begin{pmatrix}
I_j(\gamma_j, y) \\
\\
\mathscr{A}_j(\gamma_j, y) \\
\\
\left(\mathscr{F}^{(H_{\ell})}_{j,\ell}(\gamma_j, y)\right)_{\ell\ge 1} \\
\\
Z_j(\gamma_j, y)
\end{pmatrix}
:=
\begin{pmatrix}
\sum_{k=1}^{\mathscr{A}_{j}(\gamma_{j},1)}C_{j,k}\mathbb{I}_{\{\tilde{U}_{j,k}\leq y\}} = \sum_{\ell=1}^{\varphi}\mathscr{F}^{(H_{\ell})}_{j,\ell}(\gamma_j, y) \\
\\
\sum_{\ell=1}^{\varphi}X_{j,\ell}\mathbb{I}_{\{\tilde{Y}_{\ell}\leq y\}} \\
\\
\begin{cases}
    \sum_{k=1}^{X_{j,\ell}} C_{j,k,\ell} \mathbb{I}_{\{\tilde{U}_{j,k,\ell}\leq y\}} & \text{for } \ell \in \{1, \dots, \varphi\} \\
    0 & \text{for } \ell > \varphi
\end{cases} \\
\\
\sum_{\ell=1}^{\varphi}\left( \sum_{k=1}^{X_{j,\ell}} C_{j,k,\ell} \right) \mathbb{I}_{\{\tilde{Y}_{\ell}\leq y\}} = \sum_{\ell=1}^{\varphi}\mathscr{F}^{(H_{\ell})}_{j,\ell}(\gamma_j, 1)\mathbb{I}_{\{\tilde{Y}_{\ell}\leq y\}}
\end{pmatrix}
\end{equation}
Where the variables $\varphi,$ $(H_{\ell})$ and $(X_{j,\ell}, j\in[J])$ are specified in~\eqref{HXdecomp}, \eqref{Poissonvphi} and~\eqref{multigroupX}.
The count variables $(C_{j,k})$ and $(C_{j,k,\ell})$ are iid $\mathrm{MtP}(\tau_{j},\gamma_{j})$ and represent the same collection of random variables, re-indexed, and otherwise independent across $j\in[J].$  Similarly $(\tilde{U}_{j,k,l})$ are just a re-arrangement of the $(\tilde{U}_{j,k}).$ 

This system $(I_{j}, \mathscr{A}_{j}, (\mathscr{F}^{(\lambda_l)}_{j,l})_{l \ge 1}, Z_j)_{j=1,\dots,J}$ immediately implies the following manifestations of our general notion of duality on corresponding spaces:
\begin{enumerate}
\item (Coupled Mass Partitions ) The jumps $(s_{j,k})_{k\ge 1},$ $(\lambda_{l})_{l\ge 1},$ $(\sigma_{j,l}(\lambda_{l}))_{l\ge 1}$ and $((s_{j,k,l})_{k\ge 1})_{l\ge 1}$ acting as mean rates of abundance yield, via normalization, the complete set of coupled mass partitions with components $\mathbf{P}_{j} = (P_{j,k})_{k\ge 1},$ $\mathbf{Q}_0 = (Q_{0,l})_{l\ge 1},$ $\mathbf{V}_{j} = (V_{j,l})_{l\ge 1},$ and $\mathbf{Q}_{j,l} = (Q_{j,k,l})$ that underlie the duality. See Section~\ref{sec:section:PK}. 

\item (Coupled Bridges) This construction naturally induces the corresponding bridges: the fine-grained bridges $F_j(y) = \sum_{k=1}^{\infty}P_{j,k}\mathbb{I}_{\{U_{j,k}\leq y\}},$ the coagulating bridge $G_0(y) = \sum_{l=1}^{\infty}Q_{0,l}\mathbb{I}_{\{Y_l\leq y\}},$ the coarse-grained bridges $G_j(y) = (F_j \circ G_0)(y) = \sum_{l=1}^{\infty}V_{j,l}\mathbb{I}_{\{Y_l\leq y\}},$ and the fragmenting bridges $F_{j,l}(y) = \sum_{k=1}^{\infty}Q_{j,k,l}\mathbb{I}_{\{U_{j,k,l}\leq y\}}.$ See Section~\ref{sec:section:PK}

\item (Explicit EPPF Duality via Compound Poisson Count Calculus) The coupled system is pivotal for two reasons: (i) it provides a direct link to a set of mass partitions and their bridges, forming a Pitman duality at that abstract level, and (ii) crucially, its compound Poisson representation provides the tractable computational pathway that is absent when working solely with the mass partitions and bridges. While the mass partitions themselves encode the duality in principle, they resist direct combinatorial analysis. The transparent compound Poisson distributions of the count components—$(X_{j,\ell}),$ $(C_{j,k}),$ equivalently $(C_{j,k,\ell}),$ $(\mathscr{F}^{(H_{\ell})}_{j,\ell}(\gamma_j, 1))$ and $\varphi$—transform this abstract duality into explicit, computable joint Exchangeable Partition Probability Functions through straightforward conditional probability arguments. It is precisely this compound Poisson representation that renders the generalized duality analytically tractable as demonstrated in Sections~\ref{sec:posterior_decomp} to \ref{subsec:marginalized_duality}

\item (Dynamic Enablement and Computational Duality in Time:) 
This coupled system provides the static snapshot ($t=1$) of a continuous-time L\'evy process construct. Its dynamic lift, via a single Poisson random measure with intensity $dt \otimes \nu(d\mathbf{t}, d\lambda)$, simultaneously generates in lockstep: structured $\Lambda$-Fleming-Viot processes, their dual structured $\Lambda$-coalescents, and coupled homogeneous fragmentation processes in the sense of Bertoin. Crucially, the construction introduces a novel modeling feature: the pathwise generation of process pairs where one is guaranteed to be a refinement of the other at all times. Our framework's analytical tractability not only establishes this pathwise stochastic ordering, but elevates it to a hard computational primitive. (See Section~\ref{sec:DynamicCouplin})
\end{enumerate}
\end{thm}

\begin{proof}
The difficulty of this result is in envisioning its construction, which is clearly detailed in  \eqref{coupledfinetocoarsevector}  and \eqref{coupledcoarsetofinevector}. That is to say how and why does one choose the components. The explicit constructions, that is to say, the architecture, naturally establish existence of all components mentioned. 
The compound Poisson representations for $\mathscr{A}_j(\gamma_j, y)$ and $Z_j(\gamma_j, y)$ can be read from \citep[Proposition 3.2 and Theorem 3.1]{hibp25}. For further clarity, since representations in~\cite{hibp25} do not involve indicators, with respect to the exposition and notation in~\citep[Proposition 3.2]{hibp25}   we note that given $\lambda_{l},$ $\xi_{j,l}:=\sum_{k=1}^{\infty}\mathbb{I}_{\{\mathscr{P}_{j,k,l}(s_{j,k,l}\gamma_{j})>0\}}$  appearing in \eqref{coupledcoarsetofinevectorQUAD} is the total sum of a Bernoulli process and has a distribution $\mathrm{Poisson}(\psi_{j}(\gamma_{j})\lambda_{l})$, hence $(\mathscr{A}_{j}, j\in[J])$ is a multivariate Poisson IBP in the sense of ~\cite[Section 5]{James2017}. Since the processes $\mathscr{A}_j(\gamma_j, y)$ and $Z_j(\gamma_j, y)$ contain all the relevant variables the coupling constructions serve to conclude the representation for the other components. Alternatively, one can condition on $\sigma_{0}$ and apply~\cite{James2017} directly. These are effectively derived from thinning arguments similar to those described in the forthcoming~Section~\ref{sec:posterior_decomp}. As to implications, the mass partitions and bridges are directly read off and hence structurally connected to $(I_{j}, \mathscr{A}_{j}, (\mathscr{F}^{(\lambda_l)}_{j,l})_{l \ge 1}, Z_j)_{j=1,\dots,J}$  via the subordinator jumps and atoms. The point about corresponding EPPF's is a crucial feature of this work and is direct, once the precise compound Process components are identified, from the discussion in Section~\ref{sec:background} as articulated in ~\cite{PitmanPoissonMix}, as $(s_{j,k})_{k\ge 1},$ $(\lambda_{l})_{l\ge 1},$ $(\sigma_{j,l}(\lambda_{l}))_{l\ge 1}$ and $((s_{j,k,l})_{k\ge 1})_{l\ge 1}$ are precisely interpreted as absolute mean rates of types or species, and demonstrated specifically in~\cite{James2002, JamesStick, JLP2}. See also section~\ref{sec:genPnlaws}. The L\'evy-It\^o nature of our construction enables the direct lift to the continuous time setting with dynamics in effect following a coupled version of Bertoin's~\cite{BerFrag} Homogeneous Fragmentaton (HFP) framework, which otherwise exhibits our notion of duality at every time point. This is expressed in more detail in Section~\ref{sec:DynamicCouplin}, but is a fairly automatic consequence once specified. 
\end{proof}

\begin{rem}[On the Non-Triviality of This Construction]
\label{rem:nontrivial}
The brevity of the proof belies the depth of challenge. Pitman's coagulation-fragmentation duality \cite{Pit99Coag} for $\mathrm{PD}(\alpha,\theta)$ has resisted extension for over two decades due to: (i) reliance on stable-beta-gamma algebra unavailable for general subordinators; (ii) absence of multi-group characterizations; (iii) disintegration of tractable components—the classical duality's explicit laws vanish for general subordinators, leaving components analytically opaque.

The breakthrough lies in construction design. Equations \eqref{coupledfinetocoarsevector} and \eqref{coupledcoarsetofinevector} provide equivalent representations coupling all four components—fine partition $I_j$, coagulation operator (via $\mathscr{A}_{j}$), fragmenting processes $\mathscr{F}^{(H_{\ell})}_{j,\ell}$, and coarse partition $Z_j$—on the same probability space with explicit interdependence.
The construction exploits the connection between mixed Poisson processes and EPPF structure (Section~\ref{sec:background}), bypassing notoriously difficult analysis on mass and integer partition spaces. Instead, it operates in the analytically simpler coupled mixed Poisson framework where explicit laws emerge naturally. While our prior work \cite{hibp25} established the compound Poisson structure of $\mathscr{A}_j$ and $Z_j$ as a distributional fact, it did not expose the underlying constructive mechanism. The key insight here is to reveal that mechanism: we build the allocation process from the ground up using the indicator representation $\mathbb{I}_{\{\mathscr{P}_{j,k}(s_{j,k}\gamma_{j})>0\}},$ which does not appear in~\cite{hibp25}. This shows that the component $\xi_{j,l}$, previously presented simply as a $\mathrm{Poisson}(\psi_{j}(\gamma_{j})\lambda_{l})$ variable, is in fact the literal count of active fine-scale clusters: $\xi_{j,l}:=\sum_{k=1}^{\infty}\mathbb{I}_{\{\mathscr{P}_{j,k,l}(s_{j,k,l}\gamma_{j})>0\}}$. The fact that this sum over the fundamental atoms $(s_{j,k,l})$ collapses to a clean Poisson variable is the subtle, deceptive, engine of the entire construction, making clear its role in the coagulation process dictated by $(\lambda_{l})$. This crucially provides \emph{pointwise coupling}, not distributional equivalence. Proving distributional equivalence of cluster sizes $(C_{j,k})_{k \geq 1}$ and $(C_{j,k,\ell})_{k,\ell}$ requires lengthy verification of marginals and conditionals. Pointwise equality reveals these are \emph{the same random variables}, reindexed by species assignment, bypassing bookkeeping and enabling tractable partition-valued duality. Critically, distributional equivalence cannot support the continuous-time extension (Section~\ref{sec:DynamicCouplin}), which requires all four components evolving simultaneously—guaranteed only by pointwise coupling.
Section~\ref{sec:DynamicCouplin} shows this static construction lifts to continuous time, yielding a four-part dual architecture with implications beyond this work. The explicit EPPFs (Sections~\ref{sec:fragdualitycomp}--\ref{subsec:marginalized_duality}) are not merely finite-sample results but cross-sections of richer dynamical theory.
\end{rem}

\begin{rem}
As demonstrated via prediction rules in \cite{hibp25}, we can add arbitrary numbers of groups or future realizations within existing groups after observing the coupled processes in Theorem \ref{coupledsupertheorem}.
\end{rem}

\subsection{Conditioning and Decompositions of the Process}
\label{sec:posterior_decomp}

In this section, we describe the conditional behavior of the basic subordinators and the analytic forms of the distributions from Theorem~\ref{coupledsupertheorem}. In many applications, such as the microbiome setting of~\citep{hibp25}, the process of counts and species tags, $(Z_{j}, j\in[J])$, is considered observable, whereas the other processes are latent. Our duality results rely on the law of the process, conditional on these counts. The results below are primarily deduced from~\cite{hibp25}, see also Section 5 of~\citep{James2017}.
Specifically, we are interested in the law of the subordinators and related constructs given the observed counts $n_{j,\ell}$ for $j \in [J]$ and $\ell \in [r]$, associated with $r$ iid species tags $(\tilde{Y}_{1},\ldots,\tilde{Y}_{r})$. These counts arise from the fragmenters, where $\mathscr{F}^{(H_{\ell})}_{j,\ell}(\gamma_j, 1)=n_{j,\ell}$. Conditioning on these observations transforms the law of the latent hierarchical process, inducing the state-dependence that is central to our general duality. The foundation for this conditioning is the decomposition of the L\'evy measure $\nu$ based on whether a jump results in an observation. A jump of the parent process with size $\lambda$ and type $\mathbf{t}$ fails to produce any counts across all $J$ groups with probability $e^{-\sum_{j=1}^{J}t_{j}\gamma_{j}}$. This naturally partitions the L\'evy measure $\nu(d\mathbf{t},d\lambda)$ into two components:
$$
\nu(d\mathbf{t},d\lambda) = \underbrace{e^{-\sum_{j=1}^{J}t_{j}\gamma_{j}}\nu(d\mathbf{t},d\lambda)}_{\text{part yielding no observations}} + \underbrace{\left(1-e^{-\sum_{j=1}^{J}t_{j}\gamma_{j}}\right)\nu(d\mathbf{t},d\lambda)}_{\text{part yielding at least one observation}}.
$$
The total mass of the second term, which corresponds to the rate of observing new species, is given by the identity:
\begin{equation}
\label{eq:integral_identity}
\int_{\mathbb{R}_{+}^{J+1}} \left(1-e^{-\sum_{j=1}^{J}t_{j}\gamma_{j}}\right) \nu(d\mathbf{t},d\lambda) = \Psi_{0}\left(\sum_{j=1}^{J}\psi_{j}(\gamma_{j})\right).
\end{equation}
This implies that $\varphi$, the total number of distinct species observed, follows a $\mathrm{Poisson}\left(\Psi_{0}\left(\sum_{j=1}^{J}\psi_{j}(\gamma_{j})\right)\right)$ distribution.

This decomposition of $\nu$ corresponds to a thinning of the underlying Poisson point process that generates the jumps of the parent subordinator $\sigma_0$. We can thus partition the set of all species into two disjoint collections:
\begin{itemize}
    \item The set of $\varphi$ \textit{observed} species, associated with parent jump sizes $(H_\ell)_{\ell=1}^\varphi$ and locations $(\tilde{Y}_\ell)_{\ell=1}^\varphi$.
    \item The infinite set of \textit{unobserved} species, associated with parent jump sizes $(\lambda'_l)_{l \ge 1}$ and locations $(Y'_l)_{l \ge 1}$.
\end{itemize}
This allows us to represent the compound subordinator $\sigma_j \circ \sigma_0$ as a sum of contributions from these two distinct sets of jumps. Let the jump sizes of the child process $\sigma_j$ corresponding to observed and unobserved parent jumps be denoted by $\tilde{\sigma}_{j,\ell}(H_\ell)$ and $\sigma'_{j,l}(\lambda'_l)$, respectively. The process can then be written as:
\begin{equation}
\label{eq:process_decomposition}
(\sigma_j \circ \sigma_0)(y) = \sum_{l=1}^{\infty} \sigma'_{j,l}(\lambda'_{l}) \mathbb{I}_{\{Y'_{l} \le y\}} + \sum_{\ell=1}^{\varphi} \tilde{\sigma}_{j,\ell}(H_{\ell}) \mathbb{I}_{\{\tilde{Y}_{\ell} \le y\}}.
\end{equation}
furthermore
\begin{equation}
\label{eq:process_decomposition2}
\sigma_{0}(y) = \sum_{l=1}^{\infty} \lambda'_{l} \mathbb{I}_{\{Y'_{l} \le y\}} + \sum_{\ell=1}^{\varphi} H_{\ell} \mathbb{I}_{\{\tilde{Y}_{\ell} \le y\}}.
\end{equation}

And importantly, for each observed species $\ell \in \{1, \ldots, \varphi\}$, we specify the law governing its properties. This process begins with the parent jump size $H_{\ell}$ and the corresponding vector of child jump sizes, $(\tilde{\sigma}_{j,\ell}(H_{\ell}))_{j \in [J]}$.
A crucial feature of the model is that a jump from the underlying L\'evy process is only \emph{observed} if it generates at least one individual in at least one of the $J$ features. Given a vector of child jump sizes $\mathbf{t} = (t_1, \ldots, t_J)$, the total rate of generating individuals is $\sum_{j=1}^{J} \gamma_j t_j$. The probability of observing at least one individual is therefore $\left(1 - e^{-\sum_{j=1}^{J} t_j \gamma_j}\right)$.

We define the joint law of an observed parent jump $H_{\ell}$ and its child jumps $(\tilde{\sigma}_{j,\ell}(H_{\ell}))$ by conditioning the base L\'evy measure $\nu(\text{d}\mathbf{t},\text{d}\lambda)$ on this observation event. This yields a new probability measure, where the denominator is the total mass of all such observable jumps, ensuring it integrates to one:
\[
\mathbb{P}\left(H_{\ell} \in \text{d}\lambda, (\tilde{\sigma}_{j,\ell}(H_{\ell}))_{j \in [J]} \in \text{d}\mathbf{t}\right) = \frac{\left(1-e^{-\sum_{j=1}^{J}t_{j}\gamma_{j}}\right) \nu(\text{d}\mathbf{t},\text{d}\lambda)}{\Psi_{0}\left(\sum_{j=1}^{J}\psi_{j}(\gamma_{j})\right)}.
\]
To simplify, we can derive the marginal distribution for the parent jump size $H_{\ell}$ alone. This is achieved by integrating the joint law above over all possible child jump size vectors $\mathbf{t}$, which gives:
\[
\mathbb{P}(H_{\ell}\in \text{d}\lambda)=\frac{\left(1-e^{-\lambda \sum_{j=1}^{J}\psi_{j}(\gamma_{j})}\right)\tau_{0}(\lambda)\text{d}\lambda}{\Psi_{0}\left(\sum_{j=1}^{J}\psi_{j}(\gamma_{j})\right)}.
\]
Now, we can assemble the complete picture. For each observed species $\ell$, we describe the full joint distribution for:
\begin{enumerate}
    \item The specific count vector that was observed, $(\mathscr{F}^{(H_{\ell})}_{j,\ell}(\gamma_{j},1)=n_{j,\ell}, j\in[J])$.
    \item The underlying child jump sizes that generated those counts, $(\tilde{\sigma}_{j,\ell}(H_{\ell})=t_{j}, j\in[J])$.
    \item The parent jump size, $H_{\ell}$.
\end{enumerate}
This full distribution can be understood as a product of conditional probabilities that reflect the generative story. The final expression, which we present below, combines the probability of the counts given the jumps, the probability of the child jumps given the parent jump, and the probability of the parent jump.

For each $\ell \in \{1, \ldots, \varphi\}$, the joint distribution is given by:
\begin{equation}
\label{fulljoint}
\frac{\prod_{j=1}^{J} (\gamma_{j})^{n_{j,\ell}} t_{j}^{n_{j,\ell}} e^{-t_{j} \gamma_{j}}}{\left(1 - e^{-\sum_{j=1}^{J} \gamma_{j} t_{j}}\right) \prod_{j=1}^{J} n_{j,\ell}!} \cdot \frac{\left(1 - e^{-\sum_{j=1}^{J} \gamma_{j} t_{j}}\right) \prod_{j=1}^{J} \mathbb{P}(\sigma_{j}(\lambda)\in \text{d}t_{j})}{\left(1 - e^{-\lambda \sum_{j=1}^{J} \psi_{j}(\gamma_{j})}\right)} \mathbb{P}(H_{\ell}\in \text{d}\lambda).
\end{equation}
Here, the first term is the probability mass function of a zero-truncated multivariate Poisson distribution. It gives the probability of observing the specific count vector $(n_{j,\ell})$ under the explicit condition that the total count is at least one, i.e., $\sum_{j=1}^{J}n_{j,\ell} \ge 1$. These points, which can be read from~\cite{hibp25},  leads directly to a structural decomposition of the subordinators $(\sigma_j \circ \sigma_0, \sigma_0)$ themselves, which we formalize in the following proposition. This result is a direct application of the characterization in~\citep[Theorem 4.1, eq (4.4)]{hibp25}.

\begin{lem}[Decomposition of the Process Given the Counts, cf. Theorem 4.1 in~\citep{hibp25}]
\label{prop:posterior_process_decomp}
Let $(\tilde{\sigma}_0, \tilde{\sigma}_1, \ldots, \tilde{\sigma}_J)$ denote the process $(\sigma_0, \sigma_1, \ldots, \sigma_J)$ under the conditional law given the fragmenting count vectors $(\mathscr{F}^{(H_{\ell})}_{j,\ell}(\gamma_j, 1)=n_{j,\ell}, j\in[J])_{\ell\in[r]}$, the unique species tags $(\tilde{Y}_{1},\ldots, \tilde{Y}_{r}),$ and their number $\varphi=r$. This conditional process admits the following decomposition:
\begin{align}
    \tilde{\sigma}_0 &\overset{d}{=} \sigma'_{0,J} + \sum_{\ell=1}^r H_\ell \delta_{\tilde{Y}_\ell} \\
    \tilde{\sigma}_j \circ \tilde{\sigma}_0 &\overset{d}{=} \sigma'_{j}\circ \sigma'_{0,J} + \sum_{\ell=1}^r \tilde{\sigma}_{j,\ell}(H_\ell) \delta_{\tilde{Y}_\ell}.
\end{align}
The components of this decomposition are defined as follows:
\begin{enumerate}
    \item[1.] The Unobserved Component. This part corresponds to species not present in the sample.
    \begin{itemize}
        \item $\sigma'_{0,J}$ is the portion of $\sigma_{0}$ corresponding to jumps that go unobserved, with jumps denoted $(\lambda'_{l})_{l\ge 1}$. Its L\'evy density is thinned accordingly:
        $$
        \tau_{0,J}(\lambda)=e^{-\lambda\sum_{j=1}^{J}\psi_{j}(\gamma_{j})}\tau_{0}(\lambda).
        $$
        \item Each $\sigma'_{j}$ is an independent subordinator whose L\'evy density is thinned by the sampling effort $\gamma_j$:
        $$
        \tau_{j,\gamma_{j}}(s) = e^{-s \gamma_{j}} \tau_j(s).
        $$
        \item The thinned subordinator $\sigma'_j$ has the following fragmentation representation for $y\in [0,1]$:
           $$
            \sigma'_{j}(y)=\sum_{l=1}^{\infty}\sigma'_{j,l}(\lambda'_{l}y)=\sum_{l=1}^{\infty}\sum_{k=1}^{\infty}s'_{j,k,l}\mathbb{I}_{\{U'_{j,k,l}\leq y\}}
           $$
           where the jumps $(s'_{j,k,l})_{k\ge 1}$ are determined by the L\'evy measure $\lambda'_{l}e^{-s\gamma_{j}}\tau_{j}(s)ds$.
    \end{itemize}
    \item[2.] The Observed Component. This part corresponds to the $r$ species observed in the sample, with tags $(\tilde{Y}_{l})_{\ell=1}^r$.
    \begin{itemize}
        \item The global abundance rates $(H_\ell)_{\ell=1}^r$ are conditionally independent random variables whose laws are determined by the observed counts $\vv{n}_\ell = (n_{j,\ell})_{j \in [J]}$. The precise density for each $H_\ell$ is given in Proposition~\ref{prop:Hl_conditional_density} below.
        \item Given $H_{\ell}=\lambda$, the local abundance rates $(\tilde{\sigma}_{j,\ell}(H_\ell))_{j\in[J]}$ are conditionally independent. The law of each $\tilde{\sigma}_{j,\ell}(H_\ell)$ is a size-biased distribution, conditioned on its corresponding count $n_{j,\ell}$:
        \begin{equation}
        \label{densitypowerbiasnew}
        \mathbb{P}(\tilde{\sigma}_{j,\ell}(H_\ell) \in dt_{j} | H_\ell = \lambda, n_{j,\ell}) = \frac{t_j^{n_{j,\ell}} e^{-\gamma_j t_{j}}}{\mathbb{E}[\sigma_j(\lambda)^{n_{j,\ell}} e^{-\gamma_j \sigma_j(\lambda)}]} \mathbb{P}(\sigma_j(\lambda) \in dt_{j}).
        \end{equation}
    \end{itemize}
\end{enumerate}
\end{lem}

To fully specify the conditional process, we provide the distributions for the observed counts and the resulting conditional law for the latent jumps $H_\ell$.  

\begin{lem}[Marginal and Conditional Count Distributions, cf. Prop 4.3 in~\citep{hibp25}]\label{post:marginalofN}
For each observed species $\tilde{Y}_{\ell}$, the distributions of the count vector $(\mathscr{F}^{(H_{\ell})}_{j,\ell}(\gamma_j, 1)=n_{j,\ell}, j\in[J])$, where $\sum_{j=1}^{J}n_{j,\ell} \ge 1$, are as follows:
\begin{enumerate}
    \item[1.] Given $H_{\ell}=\lambda$, the conditional probability of the count vector is:
    \begin{equation}
    \label{Ngroupldistforfragconditional}
    \mathbb{P}((\mathscr{F}^{(H_{\ell})}_{j,\ell}(\gamma_j, 1)=n_{j,\ell}, j\in[J]) | H_{\ell}=\lambda) = \frac{\prod_{j=1}^{J}{(\gamma_{j})}^{n_{j,\ell}}}{\prod_{j=1}^{J}n_{j,\ell}!} \frac{e^{-\lambda\sum_{v=1}^{J}\psi_{v}(\gamma_{v})}\prod_{j=1}^{J}\Xi^{[n_{j,\ell}]}(\lambda\tau_{j},\gamma_{j})}{1-e^{-\lambda\sum_{v=1}^{J}\psi_{v}(\gamma_{v})}}.
    \end{equation}
    \item[2.] The marginal (unconditional) probability of the count vector is:
    \begin{equation}
    \label{Ngroupldistforfrag}
    \mathbb{P}((\mathscr{F}^{(H_{\ell})}_{j,\ell}(\gamma_j, 1)=n_{j,\ell}, j\in[J]) ) = \frac{\prod_{j=1}^{J}{(\gamma_{j})}^{n_{j,\ell}}}{\prod_{j=1}^{J}n_{j,\ell}!}
    \frac{\left(\Psi_{0}\circ \sum_{j=1}^{J}\psi_{j}\right)^{(\vv{n}_{\ell})}(\vv{\gamma})}{\Psi_{0}(\sum_{j=1}^{J}\psi_{j}(\gamma_{j}))}.
    \end{equation}
\end{enumerate}
\end{lem}

\begin{lem}[Conditional Distribution of Global Jumps]
\label{prop:Hl_conditional_density}
For each observed species $\tilde{Y}_{\ell}$ with count vector $\vv{n}_{\ell}=(n_{j,\ell})_{j \in [J]}$, the conditional probability density function of the associated global jump $H_{\ell}$ given these counts $(\mathscr{F}^{(H_{\ell})}_{j,\ell}(\gamma_j, 1)=n_{j,\ell}, j\in[J])$ is:
\begin{equation}\label{eq:Hl_density2}
\mathbb{P}(H_{\ell}\in d\lambda|\vv{n}_{\ell},\vv{\gamma}) = \frac{\tau_{0}(\lambda) \prod_{j=1}^{J}\mathbb{E}\left[\sigma_{j}(\lambda)^{n_{j,\ell}} e^{-\sigma_{j}(\lambda)\gamma_{j}}\right]d\lambda}{\left(\Psi_{0}\circ \sum_{j=1}^{J}\psi_{j}\right)^{(\vv{n}_{\ell})}(\vv{\gamma})}.
\end{equation}
\end{lem}
\begin{proof} This result does not appear explicitly in~\cite{hibp25}. However one may use Lemma~\ref{post:marginalofN} part 1., with the marginal of $H_{\ell},$ to obtain the joint distribution and then obviously divide by the marginal distribution described in~\eqref{Ngroupldistforfrag}
\end{proof}

\begin{rem}One may also use the joint distribution~\eqref{fulljoint} to directly verify Lemmas  \ref{post:marginalofN} and \ref{prop:Hl_conditional_density}. Furthermore, we note that while the Lemmas~\ref{prop:posterior_process_decomp},  \ref{post:marginalofN} and the forthcoming Lemma~\ref{prop:GibbsEPPF} are obtained from~\cite{hibp25} the interpretation of the fragmenters $(\mathscr{F}^{(H_{\ell})}_{j,\ell}(\gamma_{j},1))$ is not given and the results appear using a more generic notation of $N_{j,\ell}.$
\end{rem}

\subsection{The Fragmentation Component of Duality}\label{sec:fragdualitycomp}
As we mentioned, the work of~\citep{hibp25} already contains many components of the coarse to fine duality characterizations including non-obvious analytic results as we described above.  We now use Proposition 4.1 of that work to identify the fragmentaton of the coarse to fine side of the duality

\begin{lem}[The Fragmentation Component of the Duality, cf.~Proposition 4.1 in~\citep{hibp25}]
\label{prop:GibbsEPPF}
Consider the total count $\mathscr{F}^{(H_{\ell})}_{j,\ell}(\gamma_{j},1) = \sum_{k=1}^{X_{j,\ell}} C_{j,k,\ell}$ for $j=1,\ldots,J$ and $\ell=1,\ldots,r$. The conditional distributions of the number of blocks $X_{j,\ell}$ and the block sizes $\{C_{j,k,\ell}\}_{k=1}^{X_{j,\ell}}$, given the total count $\mathscr{F}^{(H_{\ell})}_{j,\ell}(\gamma_{j},1)=n_{j,\ell}$ and the parent jump size $H_{\ell} = \lambda$, have the following properties:
\begin{enumerate}
    \item The distributions are independent across the feature-species indices $(j,\ell)$.

    \item For any given pair $(j, \ell)$ with a total count $n_{j,\ell} > 0$, the joint probability mass function for the ordered composition of counts $\mathbf{c}_{j,\ell} := (c_{j,1,\ell}, \ldots, c_{j,x_{j,\ell},\ell})$ such that $\sum_{k=1}^{x_{j,\ell}} c_{j,k,\ell} = n_{j,\ell}$ is:
    \begin{equation} \label{eq:ordered_counts_pmf}
    \mathbb{P}\left(\mathbf{C}_{j,\ell} = \mathbf{c}_{j,\ell}, X_{j,\ell}=x_{j,\ell} \mid \mathscr{F}^{(H_{\ell})}_{j,\ell}(\gamma_{j},1)=n_{j,\ell}, H_\ell=\lambda\right) = \frac{n_{j,\ell}!}{x_{j,\ell}! \prod_{k=1}^{x_{j,\ell}} c_{j,k,\ell}!} p^{[n_{j,\ell}]}(\mathbf{c}_{j,\ell}|\lambda\tau_{j},\gamma_{j}).
    \end{equation}
    This distribution corresponds to arranging the blocks of a finite Gibbs partition in an exchangeable order. 
    
    \item The term $p^{[n_{j,\ell}]}(\cdot)$ is the probability of the underlying unordered partition, given by the finite Gibbs Exchangeable Partition Probability Function (EPPF):
    \begin{equation} \label{eq:fragGibbsEPPF}
    p^{[n_{j,\ell}]}(\mathbf{c}_{j,\ell}|\lambda\tau_{j},\gamma_{j}) = \frac{\lambda^{x_{j,\ell}}}{\Xi^{[n_{j,\ell}]}(\lambda\tau_{j},\gamma_{j})} \prod_{k=1}^{x_{j,\ell}}\psi^{(c_{j,k,\ell})}_{j}(\gamma_{j}).
    \end{equation}

    \item The marginal conditional distribution for the number of blocks, $X_{j,\ell}$, given $n_{j,\ell}>0$, is:
    \begin{equation} \label{eq:fragNumBlocks}
    \mathbb{P}\left(X_{j,\ell}=x_{j,\ell} \mid \mathscr{F}^{(H_{\ell})}_{j,\ell}(\gamma_{j},1)=n_{j,\ell}, H_\ell=\lambda\right) = \frac{\lambda^{x_{j,\ell}}\Xi^{[n_{j,\ell}]}_{x_{j,\ell}}(\tau_{j},\gamma_{j})}{\Xi^{[n_{j,\ell}]}(\lambda\tau_{j},\gamma_{j})},
    \end{equation}
    for $x_{j,\ell} \in \{1, \ldots, n_{j,\ell}\}$. If $n_{j,\ell}=0$, then $X_{j,\ell}=0$ by definition.
\end{enumerate}
\end{lem}

Now using this result and Lemma~\ref{prop:Hl_conditional_density} we arrive at a description of the joint fragmentation operator in the coarse to fine portion of the duality

\begin{prop}[The Full J-Group Poissonized Fragmentation Operator]
\label{prop:FullFragOperator}
Consider the setting from Propositions~\ref{prop:GibbsEPPF} and~\ref{prop:Hl_conditional_density}. The J-group Poissonized fragmentation operator is expressed as the product of the combinatorial factors for arranging partitions and the laws of the underlying unordered partitions for each feature:
\begin{equation}\label{eq:full_frag_operator_expression}
\prod_{\ell=1}^{r} \prod_{j=1}^{J} \frac{n_{j,\ell}!}{x_{j,\ell}! \prod_{k=1}^{x_{j,\ell}} c_{j,k,\ell}!} \times \prod_{\ell=1}^{r} p_{\mathrm{Frag}}(\mathbf{c}_{\ell}|\vv{n}_{\ell},\vv{\gamma}).
\end{equation}
The term $p_{\mathrm{Frag}}(\mathbf{c}_{\ell}|\vv{n}_{\ell},\vv{\gamma})$ is the law of the underlying collection of \textit{unordered} partitions for a given feature $\ell$. It has the analytical form:
\begin{equation}\label{eq:p_frag_analytic_final}
p_{\mathrm{Frag}}(\mathbf{c}_{\ell}|\vv{n}_{\ell},\vv{\gamma}) = \frac{\Psi^{(\tilde{x}_{\ell})}_{0}\left(\sum_{j=1}^{J}\psi_{j}(\gamma_{j})\right)\prod_{j=1}^{J}\prod_{k=1}^{x_{j,\ell}}\psi_{j}^{(c_{j,k,\ell})}(\gamma_{j})}
{\left(\Psi_{0}\circ \sum_{j=1}^{J}\psi_{j}\right)^{(\vv{n}_{\ell})}(\vv{\gamma})}.
\end{equation}
Furthermore, this law has an integral representation that reveals its generative structure by averaging over the latent jump size $H_\ell$:
\begin{equation}\label{eq:p_frag_integral_final}
p_{\mathrm{Frag}}(\mathbf{c}_{\ell}|\vv{n}_{\ell},\vv{\gamma}) = \int_{0}^{\infty}\left(\prod_{j=1}^{J}p^{[n_{j,\ell}]}(\mathbf{c}_{j,\ell}|\lambda \tau_{j},\gamma_{j})\right)\mathbb{P}(H_{\ell}\in d\lambda|\vv{n}_{\ell},\vv{\gamma}).
\end{equation}
\end{prop}

\begin{proof}This result equates to the conditional distribution of the collections ($\{C_{j,k,\ell}\}_{k=1}^{X_{j,\ell}},X_{j,\ell})_{j\in[J[,\ell\in[r]}$ given the counts $(\mathscr{F}^{(H_{\ell})}_{j,\ell}(\gamma_{j},1)=n_{j,\ell})_{j\in[J],\ell\in[r]},$ and $\varphi=r.$
The conditional independence and subsequent expressions for this quantity are established by applying Lemmas~\ref{prop:GibbsEPPF} then integrating with respect to the conditional distribution given in Lemma~\ref{prop:Hl_conditional_density}.
\end{proof}

\begin{rem}\label{rem:lawofFrags}
This integral representation in \eqref{eq:p_frag_integral_final} reflects the structure of underlying bridges. Conditional on the latent abundance $H_{\ell} = \lambda$, the vector of fragmentation bridges $(\tilde{F}_{1,\ell}, \ldots, \tilde{F}_{J,\ell})$ becomes a collection of independent random measures. The mass partition of each bridge, $(\tilde{Q}_{j,k,\ell})_{k\ge 1}$, follows the law $\mathbb{P}^{[n_{j,\ell}]}(\lambda\tau_{j},\gamma_{j})$, which in turn generates the Gibbs EPPF $p^{[n_{j,\ell}]}(\cdot|\lambda \tau_{j},\gamma_{j})$. The bridges and masses will be formally discussed in Section~\ref{sec:section:PK}.
\end{rem}

For future reference, write the group $J$ fragmentation operator as 
\begin{equation}
\label{eq:p_frag_eppf_J_groups}
p_{\text{frag}}(\pi_{\mathbf{n}}^{(1)} | \pi_{\mathbf{n}}^{(2)},\boldsymbol{\gamma})= \prod_{\ell=1}^{r} p_{\mathrm{Frag}}(\mathbf{c}_{\ell}|\vv{n}_{\ell},\vv{\gamma})=\prod_{\ell=1}^{r} \frac{\Psi^{(\tilde{x}_{\ell})}_{0}\left(\sum_{j=1}^{J}\psi_{j}(\gamma_{j})\right)\prod_{j=1}^{J}\prod_{k=1}^{x_{j,\ell}}\psi_{j}^{(c_{j,k,\ell})}(\gamma_{j})}
{\left(\Psi_{0}\circ \sum_{j=1}^{J}\psi_{j}\right)^{(\vv{n}_{\ell})}(\vv{\gamma})}
\end{equation}

\subsection{Analysis of the Denominator for J=1}
To unpack the structure of the normalizing constant in the denominator of \eqref{eq:Hl_density2}, we now dissect this term in the fundamental case of $J=1$. Central to this analysis is the concept of a composed L\'evy measure, which arises from the two-stage structure of our process. We define the composed measure $\Lambda_{[1]}$ as:
\begin{equation}
\Lambda_{[1]}(ds) := \left( \int_{0}^{\infty} \mathbb{P}(\sigma_{1}(\lambda) \in ds) \, \tau_{0}(\lambda)d\lambda \right).
\end{equation}
This measure can be interpreted as a mixture of the laws of the group-specific subordinators $\sigma_1(\lambda)$, weighted by the global L\'evy measure $\tau_0(\lambda)$, and arises from the composition of two subordinators~\cite{Bertoin1999}. The Laplace exponent corresponding to $\Lambda_{[1]}$ is precisely the composition $(\Psi_0 \circ \psi_1)$.
For $J=1$, the denominator of \eqref{eq:Hl_density2} simplifies to the term $(\Psi_{0}\circ \psi_{1})^{(n_{l})}(\gamma_1)$, where $n_l$ is the count for the single group. This crucial quantity admits several equivalent representations that reveal the interplay between the underlying processes:
\begin{equation} \label{eq:composed_derivative_reps}
\begin{aligned}
    (\Psi_{0}\circ \psi_{1})^{(n_{l})}(\gamma_1)
    &= \int_{0}^{\infty}s^{n_{l}}e^{-\gamma_{1}s}\Lambda_{[1]}(ds) \\
    &= \sum_{x_{l}=1}^{n_{l}} \frac{n_{l}!}{x_{l}!}
        \Psi^{(x_{l})}_{0}(\psi_{1}(\gamma_1))
        \sum_{(c_{1,l} \ldots, c_{x_{l},l})}
        \prod_{k=1}^{x_{l}} \frac{
             \psi^{(c_{k,l})}_{1}(\gamma_{1})
        }{
             c_{k,l}!
        } \\
    &= \int_{0}^{\infty}
        \mathbb{E}\left[
            \sigma^{n_{l}}_{1}(\lambda)
            e^{-\gamma_{1}\sigma_{1}(\lambda)}
        \right]
        \tau_{0}(\lambda) \, d\lambda = \int_{0}^{\infty}
       e^{-\lambda\psi_{1}(\gamma_{1})}\Xi^{[n_{l}]}(\lambda\tau_{1},\gamma_{1})
        \tau_{0}(\lambda) \, d\lambda.
\end{aligned}
\end{equation}
The second equality can be verified by referring to \eqref{expmomentid} and \eqref{xisumrepbasic}, writing $\Xi^{[n_{l}]}(\lambda\tau_{1},\gamma_{1})=\sum_{x_{l}=1}^{n_{l}}\lambda^{x_{l}}\Xi^{[n_{l}]}_{x_{l}}(\tau_{1},\gamma_{1})$, and integrating with respect to $\tau_{0}.$

This analysis provides the explicit form of the conditional density for $J=1$. Substituting the simplified numerator and denominator into the general form from Proposition~\ref{prop:Hl_conditional_density}, we obtain:

\begin{equation} \label{eq:Hl_density_J1}
\begin{split}
    \frac{\mathbb{P}(H_{\ell}\in d\lambda|n_{l}, \gamma_1)}{d\lambda} &= \frac{\tau_{0}(\lambda) \mathbb{E}\left[\sigma_{1}(\lambda)^{n_{l}} e^{-\sigma_{1}(\lambda)\gamma_{1}}\right]}{(\Psi_{0}\circ \psi_{1})^{(n_{l})}(\gamma_1)} \\
    &= \frac{n_{\ell}!\sum_{x_{\ell}=1}^{n_{\ell}}\mathbb{P}(H_{\ell}\in d\lambda|X_{\ell}=x_{\ell})\mathbb{P}(X_{\ell}=x_{\ell})\mathbb{P}\left(\sum_{k=1}^{x_{\ell}}C_{k,\ell}=n_{\ell}\right)}{(\Psi_{0}\circ \psi_{1})^{(n_{l})}(\gamma_1)d\lambda}
\end{split}
\end{equation}

where $(H_{\ell},X_{\ell})$ has distribution specified in~\eqref{HXdecomp} for $J=1$ and $(C_{k})\overset{iid}\sim \MtP(\tau_{1},\gamma_{1}).$
This leads to the following proposition which we shall use later
\begin{prop}\label{HgivenXN} Consider the case of $J=1$ then
\begin{enumerate}
\item  The conditional distribution of $H_{\ell}$ given $\tilde{X}_{\ell}=x_{\ell},$ and $\mathscr{F}^{(H_{\ell})}_{\ell}(\gamma_{1},1)=n_{\ell}$ is equivalent to $\mathbb{P}(H_{\ell}\in d\lambda|\tilde{X}_{\ell}=x_{\ell})$ as specified in ~\eqref{HXdecomp}
\item  The conditional distribution of $\tilde{X}_{\ell}$ given $\mathscr{F}^{(H_{\ell})}_{\ell}(\gamma_{1},1)=n_{\ell}$ is 
\begin{equation}
\label{XgivenF}
    \mathbb{P}(\tilde{X}_{\ell}=x_{\ell}\mid \mathscr{F}^{(H_{\ell})}_{\ell}(\gamma_{1},1)=n_{\ell})=\frac{n_{\ell}!\mathbb{P}(X_{\ell}=x_{\ell})\mathbb{P}\left(\sum_{k=1}^{x_{\ell}}C_{k,\ell}=n_{\ell}\right)}{(\Psi_{0}\circ \psi_{1})^{(n_{l})}(\gamma_1)}
\end{equation}
\end{enumerate}
\end{prop}

\subsection{Distributions of $(I_{j}(\gamma_{j},1),\mathscr{A}_{j}(\gamma_{j},1),Z_{j}(\gamma_{j},1))_{j\in[J]}$ }

\begin{rem}
We now describe the distribution of pertinent totals. We note again Theorem~\ref{coupledsupertheorem} makes the multivariate mixed Poisson laws for these variables transparent. While the calculations below are definitely necessary, the derivations are not difficult and follow from straightforward arguments for Poisson and randomized (mixed) Poisson distributions with the usage of identities in \eqref{expmomentid} and \eqref{xisumrepbasic}.
\end{rem}

\begin{prop}
It follows that $(I_{j}(\gamma_{j},1)=Z_{j}(\gamma_{j},1))_{j\in[J]}$ and, furthermore, the vector is equivalent to a mixed Poisson vector described as follows, 
allowing for $n_j \ge 0$, it can be expressed by conditioning on the value of the shared random time $\sigma_0(1)=b$ and then integrating over its distribution:
\begin{align}
\mathbb{P}(\mathscr{P}_{j}(\sigma_{j}(\sigma_{0}(1))\gamma_{j})&=n_{j}, \forall j) 
=\frac{\prod_{j=1}^{J}\gamma^{n_{j}}_{j}}{\prod_{j=1}^{J}n_{j}!} \mathbb{E}\left[\prod_{j=1}^{J}[\sigma_{j}(\sigma_{0}(1))]^{n_{j}} e^{-\sum_{i=1}^{J}\sigma_{i}(\sigma_{0}(1))\gamma_i}\right]\label{eq:prob_integral_form_1}\\
&= \frac{\prod_{j=1}^{J}\gamma^{n_{j}}_{j}}{\prod_{j=1}^{J}n_{j}!} \int_{0}^{\infty} \mathbb{E}\left[ \prod_{j=1}^{J} {[\sigma_{j}(b)]}^{n_{j}} e^{-\sigma_{j}(b)\gamma_{j}} \right] \mathbb{P}(\sigma_{0}(1)\in db) \label{eq:prob_integral_form_2} \\
&= \frac{\prod_{j=1}^{J}\gamma^{n_{j}}_{j}}{\prod_{j=1}^{J}n_{j}!} \int_{0}^{\infty} \left( \prod_{j=1}^{J} \Xi^{[n_{j}]}(b\tau_{j},\gamma_{j}) \right) e^{-b\sum_{i=1}^{J}\psi_{i}(\gamma_{i})} \mathbb{P}(\sigma_{0}(1)\in db), \label{eq:prob_integral_form_3}
\end{align}
where $\Xi^{[n_{j}]}(b\tau_{j},\gamma_{j}) = \mathbb{E}[(\sigma_j(b))^{n_j} e^{-\sigma_j(b)\gamma_j}]$. This term admits a decomposition over the number of blocks, $K_j$, in the partition of $n_j$:
\begin{equation}
\Xi^{[n_{j}]}(b\tau_{j}, \gamma_{j}) = \sum_{K_{j}=1}^{n_{j}}\Xi^{[n_{j}]}_{K_{j}}(b\tau_{j},\gamma_{j}) \quad \text{where} \quad \Xi^{[n_{j}]}_{K_{j}}(b\tau_{j}, \gamma_{j}) := \frac{n_{j}!b^{K_{j}}}{K_{j}!} \sum_{(c_{j,1}, \ldots,c_{j, K_{j}})}\prod_{k=1}^{K_{j}} \frac{\psi^{(c_{j,k})}(\gamma_{j})}{c_{j,k}!}.
\end{equation}

This identity connects the moments of the subordinator to sums over partitions, whose block sizes $(C_{j,k})$ follow a mixed truncated Poisson distribution, $\mathrm{MtP}(\tau_{j},\gamma_{j})$.
\end{prop}

The relationship between the count probabilities and the arrival times for those counts can be formalized, as in the univariate case.
\begin{cor}[Joint Density of Arrival Times]\label{jointdenwaitingtimes}
Consider the random vector $\mathbf{T}_{1,\mathbf{n}} = (T_{1,n_1}, \ldots, T_{1,n_J})$ of waiting times, whose components are defined for a vector of positive counts $\mathbf{n}=(n_1, \ldots, n_J)$ as:
\begin{equation} \label{eq:time_definition_multi}
T_{1,n_{j}} = \frac{\Gamma_{j,n_j}}{\sigma_j(\sigma_0(1))}.
\end{equation}

Here, $(\Gamma_{j,n_j})$ are independent standard Gamma random variables with shapes $(n_j)$ and scale 1. The dependence across the components of $\mathbf{T}_{1,\mathbf{n}}$ is induced entirely by the shared random variable $\sigma_0(1).$ The joint density of this random time vector is directly related to the joint probability of observing $(n_1, \ldots, n_J)$ counts from a set of mixed Poisson processes:
\begin{equation}
\label{eq:groupswitchingresult2}
\frac{\mathbb{P}(T_{1,n_{1}}\in d\gamma_{1},\ldots, T_{J,n_{J}}\in d\gamma_{J})}{d\gamma_{1}\cdots d\gamma_{J}} = \frac{\prod_{j=1}^{J}n_{j}}{\prod_{j=1}^{J}\gamma_{j}} \mathbb{P}\left(\mathscr{P}_{j}(\sigma_{j}(\sigma_{0}(1))\gamma_{1})=n_{j};j\in[J]\right)
\end{equation}
\end{cor}

\begin{prop}[Joint Law of Species Counts]
The distribution of the allocation vector $(\mathscr{A}_{j}(\gamma_{j},1))_{j\in[J]}$ at $y=1$, which counts the number of distinct species within each group, is equivalent to that of a mixed Poisson vector $(\mathscr{P}_{j}(\sigma_{0}(1)\psi_{j}(\gamma_{j})))_{j\in[J]},$ which may also be deduced from~\cite{hibp25}. The joint probability mass function for observing a vector of counts $(K_{1},\ldots, K_{J})$, where $K_j \ge 0$ and $\tilde{K} = \sum_{j=1}^{J} K_{j}$ is the total count of distinct species, can be expressed by conditioning on the shared random measure $\sigma_{0}(1)$:
\begin{equation}
\label{jointAllocation}
\mathbb{P}\left(\mathscr{A}_{j}(\gamma_{j},1)=K_{j}; j\in[J]\right) = \frac{\prod_{j=1}^{J}\left(\psi_{j}(\gamma_{j})\right)^{K_{j}}}{\prod_{j=1}^{J}K_{j}!} \mathbb{E}\left[[\sigma_0(1)]^{\tilde{K}} e^{-\sigma_0(1)\sum_{i=1}^{J}\psi_i(\gamma_i)}\right].
\end{equation}
Where the expectation is given by the formula:
$$
\mathbb{E}\left[[\sigma_0(1)]^{\tilde{K}} e^{-\sigma_0(1)\sum_{i=1}^{J}\psi_i(\gamma_i)}\right]=e^{-\Psi_{0}\left(\sum_{j=1}^{J}\psi_{j}(\gamma_{j})\right)}
\times 
 \sum_{r=1}^{\tilde{K}}\frac{\tilde{K}!}{r!} \sum_{(x_1, \ldots, x_r)}\prod_{l=1}^{r} \frac{\Psi^{(x_{l})}_{0}\left(\sum_{j=1}^{J}\psi_{j}(\gamma_{j})\right)}{x_{l}!}
$$
\end{prop}
\subsection{Conditional Laws of the Coarse Partition}\label{Sec:mixturedecomp} We now come to important results that are not explicitly in~\cite{hibp25} but follow as consequences of our developments above. We note that although executed seamlessly within our approach, its derivation by other means would be challenging. See Section~\ref{sec:section:PK} for an indirect approach which is however still enabled by our framework. 

\begin{prop}[Conditional Law of the Coarse Partition]
\label{prop:conditional_law_of_fragments}
\textit{Given the vector of total observed counts from the coarse process, $(Z_j(\gamma_j, 1) = n_j)_{j \in [J]}$, the joint conditional law for observing exactly $\varphi = r$ distinct species, where the fragmenting processes $\mathscr{F}^{(H_{\ell})}_{j,\ell}$ yield the specific count vectors $(\mathscr{F}^{(H_{\ell})}_{j,\ell}(\gamma_j, 1) = n_{j,\ell})_{j \in [J]}$ for each species $\ell \in [r]$, is given by:}
\[
\begin{aligned}
&\mathbb{P}\bigg( \Big(\mathscr{F}^{(H_{\ell})}_{j,\ell}(\gamma_{j},1) = n_{j,\ell}\Big)_{j \in [J], \ell \in [r]}, \varphi=r \;\Big|\; \Big(\sum_{\ell=1}^{\varphi}\mathscr{F}^{(H_{\ell})}_{j,\ell}(\gamma_{j},1) = n_j\Big)_{j \in [J]} \bigg) \\
&\quad= \frac{1}{r!} \left( \prod_{j=1}^{J} \frac{n_j!}{\prod_{\ell=1}^{r} n_{j,\ell}!} \right) \times p_{\text{coarse}}(\pi_{\mathbf{n}}^{(2)} | \boldsymbol{\gamma})
\end{aligned}
\]
where the conditional EPPF is, 
\begin{equation}
\label{eq:p_coarse_eppf_J_groups}
p_{\text{coarse}}(\pi_{\mathbf{n}}^{(2)} | \boldsymbol{\gamma}) := \frac{\exp\left(-\Psi_0\left(\sum_{j=1}^{J}\psi_{j}(\gamma_j)\right)\right) \prod_{\ell=1}^{r} \left(\Psi_0 \circ \sum_{j=1}^{J}\psi_{j}\right)^{(\vv{n}_{\ell})}(\vv{\gamma})}{\mathbb{E}\left[\prod_{j=1}^{J}[\sigma_{j}(\sigma_{0}(1))]^{n_{j}} e^{-\sum_{j=1}^{J}\sigma_{j}(\sigma_{0}(1))\gamma_j}\right]}.
\end{equation}
\textit{where the counts must satisfy the consistency constraint $\sum_{\ell=1}^{r} n_{j,\ell} = n_j$ for each group $j \in [J]$. The formula holds for all nonnegative counts, including the case where $n_j=0$. Furthermore, for any group $j$ with a positive total count, $n_j>0$, this conditioning on the count is equivalent to conditioning on the $n_j$-th arrival time being realized at $\gamma_j$, as established in Corollary~\ref{jointdenwaitingtimes}.}
\end{prop}

\begin{proof}
It follows from~\eqref{Poissonvphi} and \eqref{Ngroupldistforfrag} that 
$$
\mathbb{P}(\varphi=r)\prod_{\ell=1}^{r}\mathbb{P}((\mathscr{F}^{(H_{\ell})}_{j,\ell}(\gamma_{j},1)=n_{j,\ell}, j\in[J]) )
$$
is equal to
$$
\frac{\prod_{j=1}^{J}{(\gamma_{j})}^{n_{j}}}
{r!\prod_{j=1}^{J}n_{j,\ell}!}
{\exp\left(-\Psi_0\left(\sum_{j=1}^{J}\psi_{j}(\gamma_j)\right)\right) \prod_{\ell=1}^{r} \left(\Psi_0 \circ \sum_{j=1}^{J}\psi_{j}\right)^{(\vv{n}_{\ell})}(\vv{\gamma})}.
$$
Divide this quantity by ~\eqref{eq:prob_integral_form_1} to complete the result. 
\end{proof}

We close this section with a mixture representation of processes with laws given $(Z_j(\gamma_j, 1) = n_j)_{j \in [J]}$, which follows as a Corollary of Lemma~\ref{prop:posterior_process_decomp} combined with Proposition~\ref{prop:conditional_law_of_fragments}.

\begin{cor} Consider the processes composed of $(\sigma_{j},j\in[J]),\sigma_{0}$ and other relevant quantites as described in Lemma~\ref{prop:posterior_process_decomp}. Then their distributions given the coarser information in  $(Z_j(\gamma_j, 1) = n_j)_{j \in [J]}$, may be described generatively by mixing further over the conditional distributions of  $\mathscr{F}^{(H_{\ell})}_{j,\ell}(\gamma_{j},1) = \sum_{k=1}^{X_{j,\ell}} C_{j,k,\ell}$ for $j=1,\ldots,J$ and $\ell=1,\ldots,r$ and $\varphi=r$ subject to their conditioning on $(Z_j(\gamma_j, 1) = n_j)_{j \in [J]}$, as expressed in Proposition~\ref{prop:conditional_law_of_fragments}.
\end{cor}

\subsection{Fine to Coarse Duality Calculations}\label{sec:finetocoarse}
We now use our results above to obtain the fine to coarse conditional EPPF type formulae. Theoerem~\ref{coupledsupertheorem} makes this more transparent,  as the Allocation process's representation in terms of $(X_{j,\ell})$ and $\varphi$ is key to linking $I_{j}$ to $Z_{j}$. We first identify the coag operator by again an argument that reduces to elementary conditional probability. The proof again is self-evident but we will point out the clear Markovian dependence of the coag operator only through the $(\mathscr{A}_{j},j\in[J])$

\begin{prop}[The Full J-Group Poissonized Coagulation Operator]
\label{prop:FullCoagOperator}
The J-group Poissonized coagulation operator
is given by the conditional distributions of $(X_{j,\ell}=x_{j,\ell}, j\in[J]),\ell\in [r],$ and $\varphi=r$ given the Allocation process counts $(\mathscr{A}_{j}(\gamma_{j},1)=K_{j}, j\in[J])$ with $\tilde{K}=\sum_{j=1}^{J}K_{j}=1,2,\ldots,$ and hence has the form for $\tilde{x}_{\ell}
=\sum_{j=1}^{J}x_{j,\ell}$, $\sum_{\ell=1}^{r}\tilde{x}_{l}=\tilde{K}$:
\[
    \mathbb{P}(\pi_{\vect{n}}^{(2)} | \pi_{\vect{n}}^{(1)}, \vect{\gamma}) = \left(\frac{\prod_{j=1}^{J}K_{j}!}{r!\prod_{\ell=1}^{r}\prod_{j=1}^{J}x_{j,\ell}!} \right)\times  p_{\mathrm{coag}}(\pi_{\vect{n}}^{(2)}\mid \pi_{\vect{n}}^{(1)}, \vect{\gamma})   
    \]
where $p_{\mathrm{coag}}(\pi_{\vect{n}}^{(2)}\mid \pi_{\vect{n}}^{(1)}, \vect{\gamma})$ is a conditional EPPF expressed as,
\begin{equation}\label{eq:full_coag_operator_expression}
 p^{[\tilde{K}]}(\tilde{x}_1, \ldots, \tilde{x}_r | \tau_0, \sum_{j=1}^J \psi_j(\gamma_j))=
    \frac{e^{-\Psi_{0}(\sum_{j=1}^J \psi_j(\gamma_j)} \prod_{\ell=1}^{r} \Psi_0^{(\tilde{x}_{\ell})}\left(\sum_{j=1}^J \psi_j(\gamma_j)\right)}{\mathbb{E}\left[[\sigma_0(1)]^{\tilde{K}} e^{-\sigma_0(1)\sum_{j=1}^J\psi_j(\gamma_j)}\right]}
\end{equation}
\end{prop}
\begin{proof}
It follows from~\eqref{Poissonvphi} and \eqref{multigroupX} and \eqref{HXdecomp} that
$$
\mathbb{P}(\varphi=r)\prod_{\ell=1}^{r}\mathbb{P}((X_{j,\ell}=x_{j,\ell}, j\in[J]) )= \frac{e^{-\Psi_0(\sum_{j=1}^J \psi_j(\gamma_j))}\prod_{j=1}^{J}{(\psi_{j}(\gamma_{j}))}^{K_{j}}}{r!\prod_{\ell=1}^{r}\prod_{j=1}^{J}x_{j,\ell}!}\cdot 
{\prod_{\ell=1}^{r} \Psi_0^{(\tilde{x}_\ell)}(\sum_{j=1}^J \psi_j(\gamma_j))}
$$

Divide by \eqref{jointAllocation} to conclude the result.
\end{proof}

\begin{prop}[Conditional Law of the Fine Counts over J groups]
\label{prop:conditional_law_of_finecounts}
Given the vector of total observed counts from the coarse process, $(I_j(\gamma_j, 1)=\sum_{k=1}^{\mathscr{A}_{j}(\gamma_{j},1)}C_{j,k} = n_j)_{j \in [J]}$, the joint conditional law for observing allocation counts $(\mathscr{A}_{j}(\gamma_{j},1)=K_{j}, j\in[J])$ and subspecies cluster size counts $(C_{j,k}=c_{j,k}, k\in[K_{j}])_{j\in[J]}$
is given by 
$$
  \mathbb{P}(\pi_{\vect{n}}^{(1)}\mid \vect{\gamma})=\left(\frac{\prod_{j=1}^{J}n_{j}!}{\prod_{j=1}^{J}K_{j}!\prod_{k=1}^{K_{j}}c_{j,k}!} \right)\times  p_{\text{fine}}(\pi_{\mathbf{n}}^{(1)} | \boldsymbol{\gamma})
$$
where $p_{\text{fine}}(\pi_{\mathbf{n}}^{(1)} | \boldsymbol{\gamma})$ is equal to
   \begin{equation}
    \label{pfinePKEPPFcond}
    p_{\text{fine}}((c_{j,k}), (K_j) | \vect{n}, \vect{\gamma}) := \frac{\mathbb{E}\left[ [\sigma_0(1)]^{\tilde{K}} e^{-\sigma_0(1)\sum_{v=1}^J \psi_v(\gamma_v)} \right] \prod_{j=1}^J \left(\prod_{k=1}^{K_j} \psi_j^{(c_{j,k})}(\gamma_j)\right)}{\mathbb{E}\left[\prod_{j=1}^J[\sigma_j(\sigma_0(1))]^{n_j} e^{-\sum_{v=1}^J \sigma_v(\sigma_0(1))\gamma_v}\right]}.
    \end{equation}
\end{prop}
\begin{proof} From \eqref{MtPsimple} and \eqref{jointAllocation},
$$
    \mathbb{P}\left( \mathscr{A}_{j}(\gamma_{j},1)=K_j; \, j \in [J] \right) \times \prod_{j=1}^{J} \mathbb{P}\left( C_{j,k}=c_{j,k}; \, k \in [K_j]\right)
$$
is equal to
$$
\frac{\prod_{j=1}^{J}(\gamma_{j})^{n_{j}}}{\prod_{j=1}^{J}K_{j}!\prod_{k=1}^{K_{j}}c_{j,k}!} \mathbb{E}\left[[\sigma_0(1)]^{\tilde{K}} e^{-\sigma_0(1)\sum_{i=1}^{J}\psi_i(\gamma_i)}\right]\prod_{j=1}^J \left(\prod_{k=1}^{K_j} \psi_j^{(c_{j,k})}(\gamma_j)\right)
$$
Divide this quantity by ~\eqref{eq:prob_integral_form_1} to complete the result. 
\end{proof}

\begin{rem}[The Architectural Innovation of Theorem 3.1]\label{rem:Architect}
The preceding propositions are the direct result of the architectural shift enabled by Theorem~\ref{coupledsupertheorem}. Again ${I}_{j},$ and its explicit linkages to $\mathscr{A}_{j}$ and $Z_{j}$ are unique to this work.
To appreciate this, it is essential to revisit the role of the fine block count random variables, $(\mathscr{A}_{j}(\gamma_{j},1))_{j\in[J]}$. We have demonstrated how the information in Theorem~\ref{coupledsupertheorem} negotiates other opaque issues as well. The Markovian nature of the coagulation dependence only through $(\mathscr{A}_{j},j\in[J]),$ which is already easily deduced from the structure here, is reinforced by the description of $\sigma_{0}|(\mathscr{A}_{j},j\in[J])$ in~\cite[Proposition 3.3]{hibp25}.  The distributional explicitness of the  components in Theorem~\ref{coupledsupertheorem} makes the computations in Proposition~\ref{prop:conditional_law_of_finecounts} seem like an elementary exercise.  Furthermore, once that is done, in a conventional framework, one would condition on the event ${\mathscr{A}_{j}(\gamma_{j},1)=K_{j}, \forall j \in [J]}$ for some fixed counts $(K_j)_{j\in[J]}$. One would then face the intractable problem of specifying the law of a coagulation operator whose components $(X_{j,\ell})$ must adhere to the constraint $K_j = \sum_{\ell=1}^{\varphi} X_{j,\ell}.$ The components themselves remain a black box. The innovation here is to completely invert this perspective. As explicitly spelled out in Theorem~\ref{coupledsupertheorem}, our framework specifies the joint law for the entire collection of random variables—the fine counts $(\mathscr{A}_{j}(\gamma_{j},1))_{j\in[J]}$, the number of coarse blocks $\varphi,$ and the coagulation components $(X_{j,\ell})$—within a single, transparent Poissonian construction. Therefore, the identity $\mathscr{A}_{j}(\gamma_{j},1) = \sum_{\ell=1}^{\varphi} X_{j,\ell}$ is not a post-hoc puzzle to be solved, but a structural relationship between random variables whose joint law is known. Conditioning on the event ${\mathscr{A}_{j}(\gamma_{j},1)=K_{j}, \forall j}$ is no longer a conceptual leap; it is a straightforward application of conditional probability on a fully specified probability space. The conditional laws in Propositions~\ref{prop:FullCoagOperator} and~\ref{prop:conditional_law_of_finecounts} are derived by simply dividing one known probability mass by another. The path forward is made clear because there is no black box; every component's distribution was specified from the outset. By contrast, the classical duality of Pitman~\cite{Pit99Coag} is the remarkable special case where an inherent independence property allows one to proceed without ever needing to look inside the box.
\end{rem}
\subsection{The Poissonized Joint EPPF Duality Identity}\label{Sec:jointEPPF}
The detailed derivations of the preceding sections now cast the duality identity into the form of a joint EPPF, arranged subject to exchangeable ordering for the new regime of $J\ge 1$ groups. These developments provide two equivalent perspectives for describing the full joint probability of the fine-grained partition ($\pi_{\vect{n}}^{(1)}$), the coarse-grained partition ($\pi_{\vect{n}}^{(2)}$), and the latent arrival times ($\vect{\gamma}$).

First, Propositions \ref{prop:FullCoagOperator} and \ref{prop:conditional_law_of_finecounts}, and details within, give the fine to coarse component  of the dual in terms of joint $\mathbb{P}(\pi_{\vect{n}}^{(1)},\pi_{\vect{n}}^{(2)} \mid \vect{\gamma})=\mathbb{P}(\pi_{\vect{n}}^{(2)} | \pi_{\vect{n}}^{(1)}, \vect{\gamma})\mathbb{P}(\pi_{\vect{n}}^{(1)}\mid \vect{\gamma}),$
$$
\left(\frac{\prod_{j=1}^{J}K_{j}!}{r!\prod_{\ell=1}^{r}\prod_{j=1}^{J}x_{j,\ell}!} \right)\times  p_{\mathrm{coag}}(\pi_{\vect{n}}^{(2)}\mid \pi_{\vect{n}}^{(1)}, \vect{\gamma})   
\cdot\left(\frac{\prod_{j=1}^{J}n_{j}!}{\prod_{j=1}^{J}K_{j}!\prod_{k=1}^{K_{j}}c_{j,k}!} \right)\times  p_{\text{fine}}(\pi_{\mathbf{n}}^{(1)} | \boldsymbol{\gamma})
$$
In turn, Propositions \ref{prop:FullFragOperator} and \ref{prop:conditional_law_of_fragments} give the coarse to fine component of the duality expressed as  $\mathbb{P}(\pi_{\vect{n}}^{(1)},\pi_{\vect{n}}^{(2)} \mid \vect{\gamma})=\mathbb{P}(\pi_{\vect{n}}^{(1)} | \pi_{\vect{n}}^{(2)}, \vect{\gamma})\mathbb{P}(\pi_{\vect{n}}^{(2)}\mid \vect{\gamma}),$
$$
\prod_{\ell=1}^{r} \prod_{j=1}^{J} \frac{n_{j,\ell}!}{x_{j,\ell}! \prod_{k=1}^{x_{j,\ell}} c_{j,k,\ell}!} \times \prod_{\ell=1}^{r} p_{\mathrm{Frag}}(\mathbf{c}_{\ell}|\vv{n}_{\ell},\vv{\gamma})\cdot\frac{1}{r!} \left( \prod_{j=1}^{J} \frac{n_j!}{\prod_{\ell=1}^{r} n_{j,\ell}!} \right) \times p_{\text{coarse}}(\pi_{\mathbf{n}}^{(2)} | \boldsymbol{\gamma}).
$$
The coarse to fine component required more non-trivial analytic calculations to arrive at its final form, see~\cite{hibp25} for ideas of other possible representations. In any case, the expressions above represent the dual formula for the coupled processes in Theorem~\ref{coupledsupertheorem} given $(I_{j}(\gamma_{j},1)=Z_{j}(\gamma_{j},1)=n_{j})_{j\in[J]}$ allowing for all or some of the $n_{j}=0.$ Furthermore, and importantly, without any undue restrictions on the subordinators $((\sigma_{j};j\in[J]),\sigma_{0}).$ Hence we may treat the expressions above as the full expressions for the relevant joint distributions expressing the duality in the coupled PHIBP. Hereafter we shall dispense with the combinatorial terms,  for brevity, and also to allow us to focus on the joint EPPFs. Hereafter, we shall assume $n_{j}\ge 1, \forall j\in[J],$ which allows us to formally switch to conditioning on the $n_{j}$-th arrival times of "animals/customers" in each group as expressed in Corollary~\ref{jointdenwaitingtimes} with random variable denoted as $\mathbf{T}_{1,\mathbf{n}} = (T_{1,n_1}, \ldots, T_{1,n_J}),$ and density expressed in \eqref{eq:groupswitchingresult2} hereafter denoted as 
$f_{\mathbf{T}_{1,\vect{n}}}(\vect{\gamma}).$

\begin{rem}
To maintain full generality, we follow the principle established in Proposition 6.1 of~\citep{PitmanPoissonMix}, as mentioned in Remark~\ref{Remark:PitEPPFadjustment}. This dictates that for general subordinators, we must divide our expressions by the probability $\mathbb{P}(T_{1,n_{j}}<\infty, \forall j\in[J])=\mathbb{P}(\sigma_{j}(\sigma_{0}(1))>0;\forall j\in[J])$. This normalization is necessary to ensure that when the conditional EPPF is integrated with respect to $\mathbf{T}_{1,\mathbf{n}}$, its subsequent sum over all partitions yields a total probability of one. Henceforth, we will proceed with this normalization performed implicitly, omitting the term for the sake of brevity.
\end{rem}

\begin{thm}[The Unified Poissonized Duality Identity]
\label{thm:unified_duality_identity}
The joint law of the nested partition $(\pi_{\vect{n}}^{(1)}, \pi_{\vect{n}}^{(2)})$ and the arrival times $\mathbf{T}_{1,\vect{n}}=\vect{\gamma}$ is given by the following equality.
\begin{align}
\label{eq:duality_summary_full}
& \underbrace{p_{\mathrm{coag}}(\pi_{\vect{n}}^{(2)} | \pi_{\vect{n}}^{(1)}, \vect{\gamma}) \cdot \Big( p_{\mathrm{fine}}(\pi_{\vect{n}}^{(1)} | \vect{\gamma}) \cdot f_{\mathbf{T}_{1,\vect{n}}}(\vect{\gamma}) \Big)}_{\text{Coagulation applied to the (Fine Partition + Times) Law}} \nonumber \\
& \qquad = \underbrace{p_{\mathrm{frag}}(\pi_{\vect{n}}^{(1)} | \pi_{\vect{n}}^{(2)}, \vect{\gamma}) \cdot \Big( p_{\mathrm{coarse}}(\pi_{\vect{n}}^{(2)} | \vect{\gamma}) \cdot f_{\mathbf{T}_{1,\vect{n}}}(\vect{\gamma}) \Big)}_{\text{Fragmentation applied to the (Coarse Partition + Times) Law}}
\end{align}
Specifically, this takes the explicit form:
\begin{align}
\label{eq:PKduality_summary_fullexplicit}
& \underbrace{p^{[\tilde{K}]}(\tilde{x}_1, \ldots, \tilde{x}_r | \tau_0, \sum_{j=1}^J \psi_j(\gamma_j)) \cdot \Big( \frac{\prod_{j=1}^{J}\gamma^{n_{j}-1}_{j}}{\prod_{j=1}^{J}\Gamma(n_{j})} \cdot \mathbb{E}\left[ [\sigma_0(1)]^{\tilde{K}} e^{-\sigma_0(1)\sum_{v=1}^J \psi_v(\gamma_v)} \right] \cdot \prod_{j=1}^J\prod_{k=1}^{K_j} \psi_j^{(c_{j,k})}(\gamma_j) \Big)}_{\text{Coagulation applied to the (Fine Partition + Times) Law}} \nonumber \\
& \qquad = \underbrace{\prod_{\ell=1}^{r} p_{\mathrm{Frag}}(\mathbf{c}_{\ell}|\vv{n}_{\ell},\vv{\gamma}) \cdot \Big( \frac{\prod_{j=1}^{J}\gamma^{n_{j}-1}_{j}}{\prod_{j=1}^{J}\Gamma(n_{j})} \cdot  {e^{-\Psi_0(\sum_{j=1}^{J}\psi_{j}(\gamma_j))} \prod_{\ell=1}^{r} \left(\Psi_0 \circ \sum_{j=1}^{J}\psi_{j}\right)^{(\vv{n}_{\ell})}(\vv{\gamma})} \Big)}_{\text{Fragmentation applied to the (Coarse Partition + Times) Law}}
\end{align}
\end{thm}

For ease of reference, we provide the description of the EPPF of  coag operator appearing in~\eqref{eq:full_coag_operator_expression}
\begin{equation}\label{eq:full_coag_operator_expression2}
 p^{[\tilde{K}]}(\tilde{x}_1, \ldots, \tilde{x}_r | \tau_0, \sum_{j=1}^J \psi_j(\gamma_j))=
    \frac{e^{-\Psi_{0}(\sum_{j=1}^J \psi_j(\gamma_j)} \prod_{\ell=1}^{r} \Psi_0^{(\tilde{x}_\ell)}\left(\sum_{j=1}^J \psi_j(\gamma_j)\right)}{\mathbb{E}\left[[\sigma_0(1)]^{\tilde{K}} e^{-\sigma_0(1)\sum_{j=1}^J\psi_j(\gamma_j)}\right]}
\end{equation}
and the fragmentation EPPF for a single coarse block $\ell$, denoted $p_{\mathrm{Frag}}(\mathbf{c}_{\ell}|\vv{n}_{\ell},\vv{\gamma})$, is given by:
\begin{equation}
\label{eq:p_frag_definition2}
p_{\mathrm{Frag}}(\mathbf{c}_{\ell}|\vv{n}_{\ell},\vv{\gamma}) = \frac{\Psi^{(\tilde{x}_{\ell})}_{0}\left(\sum_{j=1}^{J}\psi_{j}(\gamma_{j})\right)\prod_{j=1}^{J}\prod_{k=1}^{x_{j,\ell}}\psi_{j}^{(c_{j,k,\ell})}(\gamma_{j})}
{\left(\Psi_{0}\circ \sum_{j=1}^{J}\psi_{j}\right)^{(\vv{n}_{\ell})}(\vv{\gamma})}.
\end{equation}

\begin{rem}[A Canonical Realization for the J-Group Duality]\label{rem:gengammspiders}
We note that extensive calculations and computational experiments have been carried out in our companion work \citep{hibp25} for the case where the subordinators are chosen from the Generalized Gamma (GG) family, specified by L\'evy densities:
\[
\tau_0(\lambda) = \frac{\theta_0}{\Gamma(1-\alpha_0)}\lambda^{-\alpha_0-1}e^{-\zeta_0\lambda}, \quad
\tau_j(s) = \frac{\theta_j}{\Gamma(1-\alpha_j)}s^{-\alpha_j-1}e^{-\zeta_j s} \text{ for } j \in [J],
\]
where $\alpha_0, (\alpha_j)_{j\in [J]} \in [0, 1)$, and $\theta_0, (\theta_j)_{j\in[J]} > 0$. For brevity, we do not repeat these calculations here and refer the interested reader to \citep{hibp25} for the full development, including its detailed application to a microbiome dataset and the derivation of prediction rules. The general framework, however, accommodates arbitrary subordinators, allowing one to mix and match components to construct a wide variety of models.
\end{rem}

\subsection{The Marginalized Duality and Joint EPPF}
\label{subsec:marginalized_duality}

By integrating out the random arrival times $\mathbf{T}_{1,\vect{n}}$ over their joint density $f_{\mathbf{T}_{1,\vect{n}}}(\vect{\gamma})$, we transition from the Poissonized joint law in Theorem~\ref{thm:unified_duality_identity} to the marginal joint Exchangeable Partition Probability Function (EPPF) of the nested partition $(\pi_{\vect{n}}^{(1)}, \pi_{\vect{n}}^{(2)})$. This integration preserves the duality, yielding an identity for the partition structure alone.

\begin{cor}[The Joint EPPF Duality Identity]
\label{cor:marginal_eppf_duality}
The joint EPPF of the nested partition, denoted $p(\pi_{\vect{n}}^{(1)}, \pi_{\vect{n}}^{(2)})$, satisfies the following duality relationship:
\begin{equation}
\label{eq:marginal_duality_summary}
\underbrace{p_{\mathrm{coag}}(\pi_{\vect{n}}^{(2)} | \pi_{\vect{n}}^{(1)}) \cdot p_{\mathrm{fine}}(\pi_{\vect{n}}^{(1)})}_{\text{Coagulation Path}} = \underbrace{p_{\mathrm{frag}}(\pi_{\vect{n}}^{(1)} | \pi_{\vect{n}}^{(2)}) \cdot p_{\mathrm{coarse}}(\pi_{\vect{n}}^{(2)})}_{\text{Fragmentation Path}}
\end{equation}
where:
\begin{itemize}
    \item $p_{\mathrm{fine}}(\pi_{\vect{n}}^{(1)})$ is the marginal EPPF of the fine-grained partition, obtained by integrating the corresponding joint law of the partition and times from Equation~\eqref{eq:PKduality_summary_fullexplicit}:
    \begin{equation}
    \label{eq:p_fine_marginal}
    p_{\mathrm{fine}}(\pi_{\vect{n}}^{(1)}) = \int_{\mathbb{R}_{+}^{J}} \left( \frac{\prod_{j=1}^{J}\gamma_j^{n_j-1}}{\prod_{j=1}^{J}\Gamma(n_j)} \cdot \mathbb{E}\left[ [\sigma_0(1)]^{\tilde{K}} e^{-\sigma_0(1)\sum_{v=1}^J \psi_v(\gamma_v)} \right] \cdot \prod_{j=1}^J\prod_{k=1}^{K_j} \psi_j^{(c_{j,k})}(\gamma_j) \right) d\vect{\gamma}
    \end{equation}
    \item $p_{\mathrm{coarse}}(\pi_{\vect{n}}^{(2)})$ is the marginal EPPF of the coarse-grained partition, similarly obtained by integration:
    \begin{equation}
    \label{eq:p_coarse_marginal}
\int_{\mathbb{R}_{+}^{J}} \left( \frac{\prod_{j=1}^{J}\gamma_j^{n_j-1}}{\prod_{j=1}^{J}\Gamma(n_j)} \cdot  e^{-\Psi_0(\sum_{j=1}^{J}\psi_{j}(\gamma_j))} \prod_{\ell=1}^{r} \left(\Psi_0 \circ \sum_{j=1}^{J}\psi_{j}\right)^{(\vv{n}_{\ell})}(\vv{\gamma}) \right) d\vect{\gamma}
    \end{equation}
    \item $p_{\mathrm{coag}}(\pi_{\vect{n}}^{(2)} | \pi_{\vect{n}}^{(1)})$ and $p_{\mathrm{frag}}(\pi_{\vect{n}}^{(1)} | \pi_{\vect{n}}^{(2)})$ are the corresponding conditional EPPFs that describe the probabilistic transitions between the marginal partition laws.
\end{itemize}
\end{cor}

\subsubsection{The Marginal Joint Law for $J=1.$}
The following represents the direct extension of Pitman's~\cite{Pit99Coag} duality for quite arbitrary subordinators. That is to say the case of $J=1.$ It has two equivalent expressions. The coagulation perspective gives the joint EPPF as:

\begin{align*}
p_{\mathrm{coag}}(\pi_{\vect{n}}^{(2)} | \pi_{\vect{n}}^{(1)}) \cdot p_{\mathrm{fine}}(\pi_{\vect{n}}^{(1)})= \int_{0}^\infty & \frac{\gamma_1^{n-1}}{\Gamma(n)} \exp\left({-\Psi_0\left(\psi_1(\gamma_1)\right)}\right) \\
& \times \left[ \prod_{\ell=1}^{r} \Psi_0^{(x_\ell)}\left(\psi_1(\gamma_1)\right) \right] \left[ \prod_{k=1}^{K} \psi_1^{(c_{k})}(\gamma_1) \right] d\gamma_1.
\end{align*}
where $p_{\mathrm{fine}}(\pi_{\vect{n}}^{(1)})$ is computed as in~\eqref{eq:p_fine_marginal} with $J=1.$
The fragmentation perspective gives the same joint EPPF as:

\begin{align*}
p_{\mathrm{frag}}(\pi_{\vect{n}}^{(1)} | \pi_{\vect{n}}^{(2)}) \cdot p_{\mathrm{coarse}}(\pi_{\vect{n}}^{(2)}) = \int_{0}^\infty & \frac{\gamma_1^{n-1}}{\Gamma(n)} \exp\left({-\Psi_0\left(\psi_1(\gamma_1)\right)}\right) \\
& \times \left[ \prod_{\ell=1}^{r} \left( \Psi_0^{(x_{\ell})}\left(\psi_1(\gamma_1)\right) \prod_{k=1}^{x_{\ell}} \psi_1^{(c_{k,\ell})}(\gamma_1) \right) \right] d\gamma_1.
\end{align*}
where $p_{\mathrm{coarse}}(\pi_{\vect{n}}^{(2)})$ is computed as in~\eqref{eq:p_coarse_marginal} with $J=1.$

\section{Poisson-Kingman Distributions Derived from the Couplings}\label{sec:section:PK}
The constructed framework provides a generative model for coupled partition structures. By taking a snapshot of the system at a fixed time, for instance $t=1$, we can analyze the resulting random mass partitions. We assume that the L\'evy measures of the subordinators $\sigma_{j}$ for $j=0,1,\ldots, J$ are such that the processes admit proper probability density functions. Under this condition, we can describe the corresponding Poisson-Kingman (PK) distributions \citep{Pit02}, see also \citep[pp. 108-109]{BerFrag}.
Although the previous section derived the conditional EPPFs for these cases and established various equivalent representations of their dualities, this section has a twofold objective. First, we demonstrate how to obtain these conditional EPPFs from basic methods typically used for PK models. Second, and more importantly, we explicitly name the laws of the relevant conditional mass partitions. These are mixtures of laws, denoted by $\mathbb{P}^{[n]}$ for various integers $n$. See~\citep{JamesStick} which we will utilize later in this work. We first provide information on the generic $\mathbb{P}^{[n]}$ laws.

\subsection{General \texorpdfstring{$\mathbb{P}^{[n]}$}{P[n]} laws}
\label{sec:genPnlaws}
Generically, let $\sigma$ be a subordinator specified by a L\'evy density $\tau$ and Laplace exponent $\Psi.$ We denote the probability density function of its total mass at time $\lambda$, $\sigma(\lambda)$, as 
$$
f(t|\lambda \tau):=\mathbb{P}(\sigma(\lambda)\in dt)/dt.
$$
The  jump points $(s_{k})_{k\ge 1}$ satisfy $\sigma(\lambda)=\sum_{k=1}^{\infty}s_{k}$, the normalized mass partition $(\frac{s_{k}}{\sigma(\lambda)})_{k\ge 1}$ follows a Poisson-Kingman distribution, denoted $\mathrm{PK}(\lambda\tau)$, with L\'evy density $\lambda \tau(s)$.
Now, consider a variable such as $\Gamma_{n}/\sigma(\lambda)$, which may be interpreted as the $n$-th arrival time in a trapping model, among other possibilities. The conditional density of $\sigma(\lambda)$ given $\Gamma_{n}/\sigma(\lambda)=n$ is :
\begin{equation} \label{eq:tilted_density}
f^{[n]}(t|\lambda\tau,\gamma)=\frac{t^{n}e^{-\gamma t}f(t|\lambda \tau)}{\mathbb{E}[(\sigma(\lambda))^{n}e^{-\gamma\sigma(\lambda)}]}
=\frac{t^{n}f^{[0]}(t|\lambda \tau,\gamma)}{\Xi^{[n]}(\lambda\tau,\gamma)}:=h^{[n]}(t|\lambda\tau,\gamma)\mathbb{P}(\sigma(\lambda)\in dt)/dt
\end{equation}
or equivalently
\begin{equation} \label{eq:tilted_density2}
f^{[n]}(t|\lambda\tau,\gamma)=\frac{t^{n}e^{-\gamma t}f(t|\lambda \tau)}{\mathbb{E}[(\sigma(\lambda))^{n}e^{-\gamma\sigma(\lambda)}]}
=\frac{t^{n}e^{-\gamma t}e^{\lambda\Psi(\gamma)}\mathbb{P}(\sigma(\lambda)\in dt)/dt}{\sum_{r=1}^{n}\lambda^{r}\Xi^{[n]}_{r}(\tau,\gamma)}
\end{equation}

This in turn defines a family of mixed laws:
\begin{equation} \label{eq:pk_mixture_law}
\mathbb{P}^{[n]}(\lambda\tau,\gamma)=\int_{0}^{\infty}\mathrm{PK}(\lambda\tau|t)f^{[n]}(t|\lambda\tau,\gamma)dt
=\int_{0}^{\infty}\mathrm{PK}(\lambda\tau|t)h^{[n]}(t|\lambda\tau,\gamma)\mathbb{P}(\sigma(\lambda)\in dt)
\end{equation}

where $\mathrm{PK}(\lambda\tau|t)$ is the regular conditional distribution of $(\frac{s_{k}}{\sigma(\lambda)})_{k\ge 1}$ given $\sigma(\lambda)=t$. 

The laws  are special cases of mixed Poisson-Kingman distributions as described in \citep{Pit02}. Consequently, they produce infinitely exchangeable EPPFs for general sample sizes. However, when the sample size is equal to the index $n$, the EPPF corresponding to a partition of $[n]$ has a finite Gibbs form, given explicitly as:
\begin{equation} \label{eq:gibbs_eppfprime}
p^{[n]}(n_{1},\ldots,n_{r}|\lambda\tau,\gamma)=\frac{e^{-\lambda\Psi(\gamma)}\lambda^{r}\prod_{l=1}^{r}\Psi^{(n_{l})}(\gamma)}{\mathbb{E}[(\sigma(\lambda))^{n}e^{-\gamma\sigma(\lambda)}]}
=\frac{\lambda^{r}\prod_{l=1}^{r}\Psi^{(n_{l})}(\gamma)}{\sum_{j=1}^{n}\lambda^{j}\Xi^{[n]}_{j}(\tau,\gamma)}
\end{equation}
We now demonstrate the representation of $\mathbb{P}^{[n]}(\lambda\tau_{j},\gamma_{j})$  to fit more closely with the precise processes considered here. 

\begin{prop}
\label{mixrepPnbridge}
Let $F^{[n]}(y \mid \lambda, \gamma_{j})$ denote a $\mathbb{P}^{[n]}(\lambda\tau_{j}, \gamma_{j})$ bridge. Then it has the representation for $y \in [0,1]$:
\begin{equation}
F^{[n]}(y \mid \lambda, \gamma_{j}) \overset{d}{=} \frac{\tilde{\sigma}(\lambda y) + \sum_{k=1}^{K_{n}} S_{j,k} \mathbb{I}_{\{\tilde{U}_{j,k} \leq y\}}}{\tilde{\sigma}(\lambda) + \sum_{k=1}^{K_{n}} S_{j,k}},
\end{equation}
where the random quantities on the right-hand side are constructed as follows:
\begin{itemize}
    \item $C_{j,1}, \ldots, C_{j,K_{n}}$ is a random partition of $[n]$, where $K_{n} \leq n$ is the random number of blocks, with law following the finite Gibbs partition $p^{[n]}(C_{j,1}, \ldots, C_{j,K_n} \mid \lambda\tau_{j}, \gamma_{j})$.

    \item The random variables $S_{j,k} \mid C_{j,k} = c_{j,k}$ have a conditional density given in~\eqref{MtPsimple}
   \begin{equation}
\label{MtPsimple3}
\frac{s^{c_{j,k}}e^{-s\gamma_j}\tau_{j}(s)}{\psi^{(c_{j,k})}_{j}(\gamma_j)}
\end{equation}
and are otherwise conditionally independent over $k=1, \ldots, K_{n}$.

    \item $\tilde{\sigma}(\lambda y)$ denotes a subordinator with L\'evy density $\lambda e^{-\gamma_{j}s}\tau_{j}(s)$.

    \item The normalized jumps of $\tilde{\sigma}(\lambda y)$ follow a Poisson-Kingman distribution with parameter measure $ \lambda e^{-\gamma_{j}s}\tau_{j}(s) \, ds$.

    \item The entire construction is conditional on the event $\{\Gamma_{n}/\sigma_{j}(\lambda) = \gamma_{j}\}$, which is equivalent to the event $\{\mathscr{P}(\gamma_{j}\sigma_{j}(\lambda)) = n\}$. And for $n=0,$ the conditioning is on $\{\mathscr{P}(\gamma_{j}\sigma_{j}(\lambda)) = 0\}.$
\end{itemize}
\end{prop}
\begin{proof} This is an instance of a mixture representation as described in \cite[Section 5]{James2002}. A more transparent approach is first to apply Theorem 1 or Theorem 2 of~\cite{JLP2} setting, in the notation of that work,  $U_{n}=\gamma_{j},$ then randomizing according to the conditional EPPF. One could also use the results in Proposition \ref{prop:posterior_process_decomp}, see \cite{hibp25} for related details.
\end{proof}
We specialize our results to the case of a stable subordinator which we will use in Section~\ref{sec:tagged_fragmentation_example}

\begin{prop}[cf. \cite{JamesStick}]\label{PropStablePowerlaw}
Let $\sigma_{\alpha}$ denote a stable subordinator with L\'evy density $\tau(s)=\rho_{\alpha}(s)=\alpha s^{-\alpha-1}/\Gamma(1-\alpha)$ for $s>0$, and Laplace exponent $\Psi(\gamma)=\psi_{\alpha}(\gamma)=\gamma^{\alpha}$.
\begin{enumerate}
\item The Poisson-Kingman law in \eqref{eq:pk_mixture_law} is denoted as $\mathbb{P}^{[n]}_{\alpha}(\lambda^{1/\alpha}\gamma)$ and given by
\begin{equation}
\label{StablePowerdist}
\mathbb{P}^{[n]}_{\alpha}(\lambda^{1/\alpha}\gamma) = \int_{0}^{\infty} \mathrm{PD}(\alpha|t) \dfrac{ (\lambda^{1/\alpha}\gamma)^{n-\alpha} t^n e^{-\lambda^{1/\alpha}\gamma t} e^{\lambda\gamma^{\alpha}} f_{\alpha}(t) }{\alpha\sum_{j=1}^{n}\mathbb{P}_{\alpha}^{(n)}(j)\frac{(\lambda\gamma^{\alpha})^{j-1}\Gamma(n)}{\Gamma(j)} } dt.
\end{equation}

\item Furthermore, To connect with results in section 2, see also ~\cite{HoJamesLau2025,JamesStick},
set
\begin{equation}
\label{hforStablepowerbias}
h^{[n]}_{\alpha,\lambda\gamma^{\alpha}}(t)=\dfrac{ (\lambda^{1/\alpha}\gamma)^{n-\alpha} t^n e^{-\lambda^{1/\alpha}\gamma t} e^{\lambda\gamma^{\alpha}}}{\alpha\sum_{j=1}^{n}\mathbb{P}_{\alpha}^{(n)}(j)\frac{(\lambda\gamma^{\alpha})^{j-1}\Gamma(n)}{\Gamma(j)} }
=\frac{t^{n}e^{-\lambda^{1/\alpha}\gamma t}}{\mathbb{E}[{(\sigma_{\alpha}(1))}^{n}e^{-\lambda^{1/\alpha}\gamma\sigma_{\alpha}(1)}]} 
\end{equation}
and hence $\mathbb{P}^{[n]}_{\alpha}(\lambda^{1/\alpha}\gamma) =\mathrm{PK}_{\alpha}(h^{[n]}_{\alpha,\lambda\gamma^{\alpha}}\cdot f_{\alpha})$
\item The EPPF of a partition of $[n]$ is 
\begin{equation}
\label{StablepowerEPPF}
p^{[n]}_{\alpha}(n_{1},\ldots,n_{r}|\lambda\gamma^{\alpha}):=p^{[n]}(n_{1},\ldots,n_{r}|\lambda\rho_{\alpha},\gamma)=\frac{p_{\alpha}(n_{1},\ldots,n_{r})\frac{{(\lambda\gamma^{\alpha})}^{r}}{\Gamma(r)}
}{\sum_{j=1}^{n}\mathbb{P}_{\alpha}^{(n)}(j)\frac{(\lambda\gamma^{\alpha})^{j}}{\Gamma(j)} }.
\end{equation}

\item Let $\tilde{\sigma}_{\alpha}$ denote a generalised gamma subordinator with L\'evy density $e^{-s}\rho_{\alpha}(s)$. Then a $\mathbb{P}^{[n]}_{\alpha}(\lambda^{1/\alpha}\gamma)$ bridge can be represented as   
\begin{equation}
F^{[n]}_{\alpha}(y \mid \lambda\gamma^{\alpha}) \overset{d}{=} \frac{\tilde{\sigma}_{\alpha}(\lambda\gamma^{\alpha} y) + \sum_{k=1}^{K^{[\alpha]}_{n}(\lambda\gamma^{\alpha})} \tilde{G}_{N_{k}-\alpha} \mathbb{I}_{\{\tilde{U}_{k} \leq y\}}}{\tilde{\sigma}_{\alpha}(\lambda\gamma^{\alpha}) + \sum_{k=1}^{K^{[\alpha]}_{n}(\lambda\gamma^{\alpha})} \tilde{G}_{N_{k}-\alpha}},
\end{equation}
where $\tilde{G}_{N_{k}-\alpha}$ are conditionally independent $\mathrm{Gamma}(N_{k}-\alpha,1)$ variables, and the random number of jumps $K^{[\alpha]}_{n}(\lambda\gamma^{\alpha})$ has distribution
$$
\mathbb{P}(K^{[\alpha]}_{n}(\lambda\gamma^{\alpha})=r)=\frac{\mathbb{P}_{\alpha}^{(n)}(r)\frac{{(\lambda\gamma^{\alpha})}^{r}}{\Gamma(r)}
}{\sum_{j=1}^{n}\mathbb{P}_{\alpha}^{(n)}(j)\frac{(\lambda\gamma^{\alpha})^{j}}{\Gamma(j)} }.
$$
where to be precise $(N_{1},\ldots, N_{K^{[\alpha]}_{n}(\lambda\gamma^{\alpha})})$ are random, not fixed, quantities in the representations above, specified by the EPPF in~\eqref{StablepowerEPPF}.

\item Let, for fixed $n$, $K_{n,m}$ denote the number of distinct blocks of a $\mathrm{PK}_{\alpha}(h^{[n]}_{\alpha,\lambda \gamma^{\alpha}} \cdot f_{\alpha})$ partition of $[m]$. Then, from~\cite{Pit02}, 
$m^{-\alpha} K_{n,m}$ converges almost surely as $m \rightarrow \infty$ to its $\alpha$-diversity, say $T_{n,\alpha}^{-\alpha}(\lambda \gamma^{\alpha})$, a random variable equating in distribution to that of $\sigma^{-\alpha}_{\alpha}(1)$ given $\mathscr{P}(\lambda \gamma^{1/\alpha} \sigma_{\alpha}(1)) = n$. In particular, from~\cite[p. 8, eq. (2.8)]{JamesStick},

\begin{equation}
T_{n,\alpha}(\lambda\gamma^{\alpha})\overset{d}=\frac{\tilde{\sigma}_{\alpha}(\lambda\gamma^{\alpha}) + \tilde{G}_{n-K^{[\alpha]}_{n}(\lambda\gamma^{\alpha})\alpha} }{\lambda^{1/\alpha}\gamma}\overset{d}=\frac{\tilde{\sigma}_{\alpha}\left(\lambda\gamma^{\alpha} + \tilde{G}_{\frac{n}{\alpha}-K^{[\alpha]}_{n}(\lambda\gamma^{\alpha})}\right) }{\lambda^{1/\alpha}\gamma}
\end{equation}
This equates to the inverse local time at one of  $\mathbb{P}^{[n]}_{\alpha}(\lambda^{1/\alpha}\gamma)$ and has the density $h^{[n]}_{\alpha,\lambda\gamma^{\alpha}}(t)f_{\alpha}(t)$
\end{enumerate}
\end{prop}

\begin{rem} In general $\mathbb{P}^{[n]}$ classes are precisely representable in terms of size biased sampling with replacement based on a mixed Poisson sample of size $n,$ i.e. given $\mathscr{P}(\gamma\sigma(\lambda))=n.$ See related concepts in~\cite[Lemma 4.4 and Theorem 4.5]{PPY92}, as well as~\cite{PY92}, which equates with results for $n=0,1.$ As to our results for the stable case in Proposition~\ref{PropStablePowerlaw}, the case of  $n=0$ is the generalized gamma case and is the setting for the application of~\cite[Proposition 21]{PY97}. The generalizations of  that result for $n=1,2,\ldots$ can be read from~\cite{JamesStick} which we will discuss again in Section~\ref{sec:tagged_fragmentation_example}.
\end{rem}

\subsection{Laws of the Duality Components}\label{sec:PKdualitylaws}
Now, returning to our relevant constructions, the mass partitions are constructed from the normalized jumps of the underlying processes. Let us define the jump masses at $t=1$: the jumps of the base process $\sigma_0$ are $(\lambda_l)_{l \ge 1}$, and the jumps of the fine-grained process $\sigma_j$ generated within the time interval corresponding to $\lambda_l$ are denoted by $(s_{j,k,l})_{k \ge 1}$. These serve as the fundamental building blocks, and by construction we have the following identities:
\begin{itemize}
    \item The total mass of type $j$ within coarse block $l$: $\sum_{k \ge 1} s_{j,k,l} = \sigma_{j,l}(\lambda_l)$.
    \item The total mass of the base process: $\sum_{l \ge 1} \lambda_l = \sigma_0(1)$.
    \item The total mass of type $j$ across all blocks: $\sum_{l \ge 1} \sum_{k \ge 1} s_{j,k,l} = \sigma_j(\sigma_0(1))$.
\end{itemize}
From these masses, we define the four key random partitions that constitute the duality:

\begin{enumerate}
    \item The Fine-Grained Mass Partitions for $j=1,\ldots, J$ are defined by
    \begin{equation} \label{eq:PKfine_partition}
        \mathbf{P}_{j} := \left( P_{j,k}:=\frac{s_{j,k}}{\sigma_j(\sigma_0(1))} \right)_{k \ge 1}.
    \end{equation}
    The corresponding bridges have dual representations given by
    \begin{equation} \label{eq:fine_bridges}
        F_{j}(y) = \sum_{k=1}^{\infty} P_{j,k} \mathbb{I}_{\{ U_{j,k}\le y\}} = \sum_{l=1}^{\infty} V_{j,l} F_{j,\pi(l)}(y).
    \end{equation}
    Conditional on the realization of the base process mass, $\sigma_{0}(1)=b$, the partitions $(\mathbf{P}_{j}, j\in[J])$ are independent, with each $\mathbf{P}_j$ following a Poisson-Kingman law, denoted $\mathrm{PK}(b\tau_{j})$, as can otherwise be found in~\citep{Pit02,Pit06}.

    \item The Global Coarse-Grained Mass Partition (the coagulator) is
    \begin{equation} \label{eq:coarse_partition_global}
        \mathbf{Q}_{0} := \left( Q_{0,l}:=\frac{\lambda_{l}}{\sigma_0(1)} \right)_{l \ge 1}\sim \mathrm{PK}(\tau_{0}),
    \end{equation}
    with corresponding bridge
    \begin{equation} \label{eq:coarse_bridge_global}
        G_{0}(y) = \sum_{l=1}^{\infty} Q_{0,l} \mathbb{I}_{\{ Y_{l}\le y\}}.
    \end{equation}

    \item The Intra-Block Fragmenting Masses for $j=1,\ldots,J$ and $l=1, 2, \ldots$ are
    \begin{equation} \label{eq:intrablock_partition}
        \mathbf{Q}_{j,l} := \left( Q_{j,k,l}:=\frac{s_{j,k,l}}{\sigma_{j,l}(\lambda_{l})} \right)_{k \ge 1},
    \end{equation}
    with corresponding bridge
    \begin{equation} \label{eq:intrablock_bridge}
        F_{j,l}(y) = \sum_{k=1}^{\infty}Q_{j,k,l}\mathbb{I}_{\{U_{j,k,l}\leq y\}}.
    \end{equation}
    It follows that for a fixed $l$, given the sequence of base jumps $(\lambda_{l})_{l\ge 1}$, the partitions $(\mathbf{Q}_{j,l},j\in[J])$ are conditionally independent, with each $\mathbf{Q}_{j,l}$ following a $\mathrm{PK}(\lambda_{l}\tau_{j})$ distribution.

    \item The Coarse-Grained Mass Partitions for $j=1,\ldots, J$ are
    \begin{equation} \label{eq:coarse_partition_j}
        \mathbf{V}_{j} := \left( V_{j,l}:=\frac{\sigma_{j,\pi(l)}(\lambda_{\pi(l)})}{\sigma_j(\sigma_0(1))} \right)_{l \ge 1},
    \end{equation}
    where the mass of a coarse block is the sum of its constituent fine-grained masses, $\sigma_{j,\pi(l)}(\lambda_{\pi(l)})=\sum_{k=1}^{\infty}s_{j,k,\pi(l)}$. The corresponding bridges are given by the composition
    \begin{equation} \label{eq:coarse_bridges_j}
        G_{j}(y) = \sum_{l=1}^{\infty} V_{j,l} \mathbb{I}_{\{ Y_{l}\le y\}} = (F_{j}\circ G_{0})(y).
    \end{equation}
    This structure represents a more complex distribution whose masses are governed jointly by the composed L\'evy measure $\Lambda_{[J]}$ from \eqref{eq:levy_measure_def}. However, it is built from the composition of conditionally PK laws. In particular, the marginal distribution of $\mathbf{V}_{j}$ is $\mathrm{PK}(\tilde{\Lambda}_{[j]})$, where $\tilde{\Lambda}_{[j]}$ is the L\'evy measure of the composed subordinator $\sigma_j \circ \sigma_0$.
\end{enumerate}

\begin{rem}\label{remarkdecomplaws}
It is rather straightforward to properly normalize the subordinator representations in Proposition~\ref{prop:posterior_process_decomp}, combined with the results in Section~\ref{sec:posterior_decomp}, to obtain various coupled representations of these processes under the relevant conditional laws.
\end{rem}

We can summarize the relevant results as follows, with the understanding that we can do this using basic conditioning arguments operating within a typical PK framework except for that of $(\mathbf{V}_{j},j\in[J]),$ which we inevitably obtain due to explicit identification of the other three components.
\begin{enumerate}
\item It follows, see also ~\cite[Proposition 3.3]{hibp25}, that the distribution of $G_{0}|(\mathscr{A}_{j}(\gamma_{j},1)=K_{j}; j\in[J])$ has the distribution of a 
$\mathbb{P}^{[\tilde{K}]}(\tau_{0},\sum_{j=1}^{J}\psi_{j}(\gamma_{j})$ bridge with corresponding coag operator 
\begin{equation}
\label{PKcoagoperator}
p^{[\tilde{K}]}(\tilde{x}_1, \ldots, \tilde{x}_r | \tau_0, \sum_{j=1}^J \psi_j(\gamma_j))
\end{equation}
\item Following Remark~\ref{rem:lawofFrags}, the mass partition associated with the fragmentation operator  $(\tilde{Q}_{j,k,\ell})_{k\ge 1}$, is such that given $(\mathscr{F}^{(H_{\ell})}_{j,\ell}(\gamma_{j},1)=n_{j,\ell})_{j\in[J]}$ and  $H_{\ell}=\lambda$ each follows the law $\mathbb{P}^{[n_{j,\ell}]}(\lambda\tau_{j},\gamma_{j})$, which in turn generates the Gibbs EPPF $p^{[n_{j,\ell}]}(\cdot|\lambda \tau_{j},\gamma_{j})$. That is to say the fragmentation operator is
\begin{equation}\label{eq:PKfrag_operator}
\prod_{\ell=1}^{r}p_{\mathrm{Frag}}(\mathbf{c}_{\ell}|\vv{n}_{\ell},\vv{\gamma}) = \prod_{\ell=1}^{r}\frac{\Psi^{(\tilde{x}_{\ell})}_{0}\left(\sum_{j=1}^{J}\psi_{j}(\gamma_{j})\right)\prod_{j=1}^{J}\prod_{k=1}^{x_{j,\ell}}\psi_{j}^{(c_{j,k,\ell})}(\gamma_{j})}
{\left(\Psi_{0}\circ \sum_{j=1}^{J}\psi_{j}\right)^{(\vv{n}_{\ell})}(\vv{\gamma})}.
\end{equation}

\item For the fine grained mass partitions $(\mathbf{P}_{j};j\in [J])$ by conditioning, it follows that, 
    \begin{equation}
    \label{PKFinegrainedEPPF}
    p_{\text{fine}}((c_{j,k}), (K_j) | \vect{n}, \vect{\gamma})=\int_{0}^{\infty} \left( \prod_{j=1}^J p^{[n_j]}(c_{j,1}, \ldots, c_{j,K_j} | b\tau_j, \gamma_j) \right) \mathbb{P}(\sigma_0(1) \in db | \vect{n}, \vect{\gamma}),
    \end{equation}
    where $p^{[n_j]}(\cdot | b\tau_j, \gamma_j)$ is the finite Gibbs EPPF from (5.9) for a fixed time scale $b$, and 
    $$
    \mathbb{P}(\sigma_0(1) \in db | \vect{n}, \vect{\gamma}) = \frac{\left(\prod_{j=1}^{J}\mathbb{E}[\sigma_j(b)^{n_j} e^{-\gamma_j \sigma_j(b)}]\right) \mathbb{P}(\sigma_{0}(1)\in db)}{\mathbb{E}\left[\prod_{j=1}^J[\sigma_j(\sigma_0(1))]^{n_j} e^{-\sum_{v=1}^J \sigma_v(\sigma_0(1))\gamma_v}\right]}
    $$
    is the conditional density of $\sigma_0(1)$ given the arrival times $(T_{j,n_{j}}=\gamma_{j}; j\in[J])$.

Hence, the law of the corresponding fine-grained bridges $(F_1, \ldots, F_J)$ has mixture representation such that on the  conditional distribution $\sigma_0(1)=b$, the bridges are independent $\mathbb{P}^{[n_j]}(b\tau_j, \gamma_j)$-bridges
\item For the coarse grained partition the conditional EPPF is not simple to obtain using basic properties of PK distributions. However the explicit forms of the EPPFs in \eqref{PKcoagoperator} and \eqref{PKFinegrainedEPPF} forming the fine to coarse side of the duality coupled with explict fragmentation operator \eqref{eq:PKfrag_operator} leads to a "pincer move" to identify the coarse conditional EPPF 
\begin{equation}
\label{eq:PK_coarse_eppf_J_groups}
p_{\text{coarse}}(\pi_{\mathbf{n}}^{(2)} | \boldsymbol{\gamma}) := \frac{\exp\left(-\Psi_0\left(\sum_{j=1}^{J}\psi_{j}(\gamma_j)\right)\right) \prod_{\ell=1}^{r} \left(\Psi_0 \circ \sum_{j=1}^{J}\psi_{j}\right)^{(\vv{n}_{\ell})}(\vv{\gamma})}{\mathbb{E}\left[\prod_{j=1}^{J}[\sigma_{j}(\sigma_{0}(1))]^{n_{j}} e^{-\sum_{j=1}^{J}\sigma_{j}(\sigma_{0}(1))\gamma_j}\right]}.
\end{equation}
As mentioned in Remark~\ref{remarkdecomplaws} the relevant distributions of the corresponding bridges and mass partitions can be explicitly described in a generative (coupled) fashion.
\end{enumerate}

\section{Duality for the General Class of \texorpdfstring{$(\alpha, \beta)$}{(alpha, beta)} stable-Gibbs Partitions}
\label{sec:general_gibbs_duality}
As an important interlude before the general framework, we complete unfinished business from our recent work ~\citep{HoJamesLau2025}. There, we established transformation laws under fragmentation and constructed three of the four duality components for stable-Gibbs partitions. Here, we demonstrate that the coagulation operator—left implicit in that work—admits explicit characterization at the same resolution as Pitman's classical result \citep[Theorem 12]{Pit99Coag}. A subtle point to note is that the $\mathrm{PD}(\alpha,-\beta)$ fragmentation operator in that work is independent and taken external to the system, meaning it is not subect to the change of measure arguments we used. This serves dual purposes: (i) it provides a self-contained extension of the $\mathrm{PD}(\alpha,\theta)$ duality to the broader stable-Gibbs class, and (ii) it establishes a cross-validation checkpoint for the general PHIBP framework developed in Sections~\ref{sec:coupledmixedPoisson}-\ref{sec:tagged_fragmentation_example}. The stable case is tractable via both marginal change-of-measure arguments (this section) and coupled Poissonian construction (Section~\ref{sec:tagged_fragmentation_example})—agreement between these independent approaches validates both methods.
In \citep{HoJamesLau2025}, which is part of a broader trilogy \citep{HJL, HJL2} aimed at interpreting these distributions beyond the PD case, we established the key components for this duality. However, the focus was on the transformation of laws under Pitman's fragmentation operation, and while we constructed all four components, we did not provide an explicit description of the coagulating partition given its input. Here, we complete this characterization. By assembling the pieces from our previous EPPF calculations, we now provide that explicit description, achieving a characterization of the duality at the same resolution as \citep[Theorem 12]{Pit99Coag}. We first proceed with some details of the construction and pertinent facts about $\mathrm{PK}_{\beta}(h \cdot f_{\beta})$ laws. Let $\sigma_{\beta}:=(\sigma_{\beta}(t):t\ge 0)$ denote a stable subordinator of index $0<\beta<\alpha<1$ with Laplace exponent $\psi_{\beta}(\gamma)=\gamma^{\beta}$ and L\'evy density $\rho_{\beta}(s)=\beta s^{-\beta-1}/\Gamma(1-\beta).$ $\sigma_{\beta}(1):=T_{\beta}$ has density (pdf) denoted simply as $f_{\beta}(t).$ And $\sigma^{-\beta}_{\beta}(1)\sim \mathrm{ML}(\beta,0)$ denotes a Mittag-Leffler distribution with density $g_{\beta}.$ We shall assume that there are independent stable subordinators $\sigma_{\alpha}$ and $\sigma_{\frac{\beta}{\alpha}},$ otherwise with the same form of L\'evy density $\rho_{\alpha},$
$\rho_{\frac{\beta}{\alpha}}$
so that for each $t,$ $\sigma_{\beta}(t)=\sigma_{\alpha}(\sigma_{\frac{\beta}{\alpha}}(t)),$ otherwise writing $\sigma_{\beta}=\sigma_{\alpha}\circ\sigma_{\frac{\beta}{\alpha}}$

Setting $\sigma_{\beta}(1)=\sum_{l=1}^{\infty}z_{l}$ where $(z_{l})$ are the ranked jumps of $\sigma_{\beta}$ over $[0,1].$ These form the mass partitions $(V_{1,l}:=\frac{z_{l}}{\sigma_{\beta}(1)})\sim \mathrm{PD}(\beta,0)$ and $(V_{1,l})_{l\ge 1}|\sigma_{\beta}(1)=v$ has the distribution, $\mathrm{PD}(\beta|v).$ We now consider non-negative functions $h(t)$ such that 
$\mathbb{E}[h(\sigma_{\beta}(1))]=1$ forming densities $h(t)f_{\beta}(t).$ This serves as a mixing density to create the stable Poisson-Kingman with mixing distributions denoted as 
$(V_{1,l})_{l\ge 1} \sim \mathrm{PK}_{\beta}(h\cdot f_{\beta}),$ where
\begin{equation}
\mathrm{PK}_{\beta}(h\cdot f_{\beta})=\int_{0}^{\infty}\mathrm{PD}(\beta|v)h(v)f_{\beta}(v)dv=\int_{0}^{\infty}\mathrm{PD}(\beta|s^{-\frac{1}{\beta}})h(s^{-\frac{1}{\beta}})g_{\beta}(s)ds.
\label{PKmodel}
\end{equation} 
There is a corresponding bridge $G_{1}=\sum_{l=1}^{\infty}V_{1,l}\delta_{Y_{l}}$

\begin{rem}
The \(\mathrm{PD}(\beta,\theta) \) distribution, for \( 0 \leq \beta<\alpha < 1 \) and \( \theta > -\beta \), arises for the choice of $h(t)=t^{-\theta}/\mathbb{E}[T^{-\theta}_{\beta}].$
Here $T_{\beta,\theta}$ is the analogue of $\sigma_{\beta}(1):=T_{\beta}$ with density $f_{\beta,\theta}(t)=t^{-\theta}f_{\beta}(t)/\mathbb{E}[T^{-\theta}_{\beta}],$ and 
$T^{-\beta}_{\beta,\theta}\sim \mathrm{ML}(\beta,\theta)$ are general Mittag-Leffler variables with parameters $(\beta,\theta).$ 
\end{rem}

As in \citep[Section 5]{HoJamesLau2025}, \cite{Pit02,Pit06} shows that  the Gibbs partition of $[n]$ described by the $\mathrm{PD}(\beta|v)-\mathrm{EPPF},$ is,
\begin{equation}
p_{\beta}(n_{1},\ldots,n_{r}|v):=\frac{f^{(n-r\beta)}_{\beta,r\beta}(v)}{f_{\beta}(v)} p_{\beta}(n_{1},\ldots,n_{r}),
\label{GibbsalphadeltaEPPF}
\end{equation}
where, as in~\cite{HJL,HJL2}, 
$$
\frac{f^{(n-r\beta)}_{\beta,r\beta}(v)}{f_{\beta}(v)} = \mathbb{G}^{(n,r)}_{\beta}(v)\frac{{\beta}^{1-r}\Gamma(n)}{\Gamma(r)},
$$
with, from~\cite{Gnedin06,Pit02,Pit06},
\begin{equation}
\label{bigG}
\mathbb{G}_{\beta}^{(n,r)}(t) =
\frac{\beta^{r}t^{-n}}{\Gamma(n-r\beta)f_{\beta}(t)}
\left[\int_{0}^{t}f_{\beta}(v)(t-v)^{n-r\beta-1}dv\right],
\end{equation}
and $f^{(n-r\beta)}_{\beta,r\beta}(v)$ being the conditional density of $\sigma_{\beta}(1)|K^{[\beta]}_{n}=r,$ 
 where $K^{[\beta]}_{n}$ denotes the random number of blocks according to a $\mathrm{PD}(\beta,0)$ partition of $[n], $equating to the distribution 
of random variable denoted as 
$Y^{(n-r\beta)}_{\beta,k\beta},$ with density $f^{(n-r\beta)}_{
\beta,r\beta}(t),$ such that~(pointwise), as in~\cite[eq.~(2.13), p.~323]{HJL2},
\begin{equation}
Y^{(n-r\beta)}_{\beta,r \beta}\overset{d}=\frac{T_{\beta,r\beta}}{B_{r\beta,n-r\beta}}=\frac{T_{\beta,n}}{B^{\frac{1}{\beta}}_{r,\frac{n}{\beta}-r}},
\label{jamesidspecial}
\end{equation}
where variables in each ratio are independent, and throughout, $B_{a,b}$ denotes a $\mathrm{Beta}(a,b)$ random variable. Where furthermore, 
$T_{\beta,\theta}$ is a variable with density $f_{\beta,\theta}(t)=t^{-\theta}f_{\sigma_{\beta}(1)}(t)/\mathbb{E}[\sigma^{-\theta}_{\beta}(1)]$ and hence $T^{-\beta}_{\beta,\theta}\sim \mathrm{ML}(\beta,\theta)$ denotes a generalized Mittag-Leffler variable. Hence, the EPPF of a $\mathrm{PK}_{\beta}(h\cdot f_{\beta})$ partition of $[n]$ with $1\leq r\leq n$ blocks of size $(n_{1},\ldots,n_{r})$ can be expressed as 
\begin{equation}
 p_{\beta}(n_{1},\ldots,n_{r}) \Phi^{[\beta]}_{n,r}
\label{EPPFbetaGibbs} 
\end{equation}
where 
\begin{equation}
    \Phi^{[\beta]}_{n,r} = \mathbb{E}\left[h(\sigma_{\beta}(1)) \mid K_{n}^{[\beta]} = r\right]=\int_{0}^{\infty} h(v) f^{(n-r\beta)}_{\beta,r\beta}(v) \,dv=\mathbb{E}[h(Y^{(n-r\beta)}_{\beta,r \beta})].
\end{equation}

Set mass partitions $\mathbf{V}_{1}:=(V_{1,l})_{l\ge 1},$
$\mathbf{P}_{1}:=(P_{1,k})_{k\ge 1}$ and $\mathbf{Q}_{0}:=(Q_{l})_{l\ge 1}$ with corresponding bridges $(G_{1},F_{1},G_{0})$
Then the following points are established in~\citep{HoJamesLau2025}.
\begin{enumerate}
\item The fine mass partition $\mathbf{P}_{1}\overset{d}=\mathrm{Frag}_{\alpha,-\beta}(\mathbf{V}_{1}),$ where the iid collection of  $\mathrm{PD}(\alpha,-\beta)$ fragmenting masses are independent of $\mathbf{V}_{1}\sim \mathrm{PK}_{\beta}(h\cdot f_{\beta}),$ has a marginal distribution
\begin{equation}
\label{lawofPDfine}
\int_{0}^{\infty}\mathrm{PD}(\alpha|s )\left[ \int_{0}^{\infty}h(sy^{1/\alpha})f_{\frac{\beta}{\alpha}}(y)dy \right] f_{\alpha}(s)ds
\end{equation}
with an EPPF for a partition of $[n]$ into $K$ realized blocks, $C_{1},\ldots,C_{K}$, with sizes $(c_{1},\ldots,c_{K}),$ of
$$
p_{\alpha}(c_{1},\ldots,c_{K}) \sum_{r=1}^{K}\mathbb{P}^{(K)}_{\frac{\beta}{\alpha}}(r)\Phi^{[\beta]}_{n,r}.
$$
\item Where there are the identities using $K^{[\beta]}_{n}=K^{[\frac{\beta}{\alpha}]}_{K^{[\alpha]}_{n}},$
\begin{enumerate}
\item For $K^{[\frac{\beta}{\alpha}]}_{K}$ having pmf $\mathbb{P}^{(K)}_{\frac{\beta}{\alpha}}(r)$ for $r=1,\ldots,K,$ there is the distributional identity
$$
Y^{(n-K^{[
{\beta}/{\alpha}]}_{K}\beta)}
_{\beta,K^{[
{\beta}/{\alpha}]}_{K}\beta}\overset{d}=Y^{(n-K\alpha)}_{\alpha,K\alpha}\times T^{\frac{1}{\alpha}}_{\frac{\beta}{\alpha}},
$$
\item and  $\mathbb{E}[h(\sigma_{\beta}(1)) \mid K_{n}^{[\alpha]}=K]$ equal to
\begin{equation} \label{eq:gibbs_factor_mixing}
\sum_{j=1}^{K}\mathbb{P}^{(K)}_{\frac{\beta}{\alpha}}(j)\Phi^{[\beta]}_{n,j}=\int_{0}^{\infty}\int_{0}^{\infty}h(sy^{1/\alpha})f_{\frac{\beta}{\alpha}}(y)f^{(n-K\alpha)}_{\alpha,K\alpha}(s)dsdy
\end{equation}  
\end{enumerate} 
\item $(\mathbf{P}_{1},\mathbf{Q}_{0})$ are specified to have the joint law 
\begin{equation}
\mathrm{P}^{\frac{\beta}{\alpha}}_{\alpha}(h):=\int_{0}^{\infty}\int_{0}^{\infty}\mathrm{PD}(\alpha|s)\mathrm{PD}({\beta}/
{\alpha}|y)h\big(sy^{\frac{1}{\alpha}}\big)f_{\frac{\beta}{\alpha}}(y)f_{\alpha}(s)dyds.
\label{jointmeasure2}
\end{equation}
such that $G_{1}(y)\overset{d}=F_{1}(G_{0}(y))$ for $y\in[0,1]$
\item The distribution of the coagulator $\mathbf{Q}_{0}|\mathbf{P}_{1}$ or $\mathbf{Q}_{0}|C_{1},\ldots,C_{K}$ or its EPPF are not described. 
\end{enumerate}

\begin{rem}As noted in~\cite{HoJamesLau2025},it follows that with respect to the distribution in~\eqref{jointmeasure2}, when $h(t)=t^{-\theta}/\mathbb{E}\big[T^{-\theta}_{\beta}\big]$ for $\theta>-\beta,$ $\mathbf{P}_{1}\sim \mathrm{PD}(\alpha,\theta)$ is independent of $\mathbf{Q}_{0}\sim \mathrm{PD}(
{\beta}/{\alpha},
{\theta}/{\alpha}),$ which corresponds to the specifications in~\citep{Pit99Coag}. That is
$$
h(sy^{1/\alpha}) = \frac{s^{-\theta}}{\mathbb{E}\left[T^{-\theta}_\alpha\right]} \times \frac{y^{-\frac{\theta}{\alpha}}}{\mathbb{E}\left[T^{-\frac{\theta}{\alpha}}_{\beta/\alpha}\right]}
$$
\end{rem}

The following result describing formally a duality for a large $(\alpha,\beta)$-Gibbs class extending that of~\citep{Pit99Coag}, while not appearing explicitly in our work in~\citep{HoJamesLau2025} is now seen to follow from those results.

\begin{prop}\label{Gibbshduality}[cf.~\citep{HoJamesLau2025}]
Consider the mass partitions $(\mathbf{P}_{1}, \mathbf{V}_{1})$, where the coarse partition $\mathbf{V}_{1} \sim \mathrm{PK}_{\beta}(h \cdot f_{\beta})$ and the fine partition is generated as $\mathbf{P}_{1} \overset{d}{=} \mathrm{Frag}_{\alpha,-\beta}(\mathbf{V}_{1})$. Let the partition process evolve according to the dynamics laid out in \textbf{Setup~\ref{setup:canonical_duality}}. Then the following result holds with the representation $p_{\text{coag}}\times p_{\text{fine}}=p_{\text{frag}}\times p_{\text{coarse}},$ as indicated read from left to right The joint EPPF of the pair $(\pi_n^{(1)}, \pi_n^{(2)})$ satisfies the following identity, equating its dual coagulation and fragmentation forms:
\begin{equation}
\begin{aligned}
\label{Gibbsduality}
    &\frac{ p_{\frac{\beta}{\alpha}}(x_{1},\ldots,x_{r}) \Phi^{[\beta]}_{n,r} }{ \sum_{j=1}^{K}\mathbb{P}^{(K)}_{\frac{\beta}{\alpha}}(j)\Phi^{[\beta]}_{n,j} } \times p_{\alpha}(c_{1},\ldots,c_{K}) \sum_{j=1}^{K}\mathbb{P}^{(K)}_{\frac{\beta}{\alpha}}(j)\Phi^{[\beta]}_{n,j}\\
    &\qquad = \prod_{l=1}^{r} p_{\alpha,-\beta}(\mathbf{c}_{l}) \times p_{\beta}(n_{1},\ldots,n_{r}) \Phi^{[\beta]}_{n,r}.
\end{aligned}
\end{equation}
\end{prop}
\begin{proof}
Due to the independence of the $\mathrm{Frag}_{\alpha,-\beta}$ operator from its input $\mathbf{V}_{1}$ and the known EPPF of $\mathbf{V}_{1}$ the right hand side of \eqref{Gibbsduality} is established. Since the marginal distribution of $\mathbf{P}_{1}$ is established in ~\citep{HoJamesLau2025} these identify the distribution of the coagulation operator.  
\end{proof}
\section{The Stable Case: Distributional Operations and Fragmentation Invariance}
\label{sec:tagged_fragmentation_example}

This section completes the circle begun in Section~\ref{sec:general_gibbs_duality}, which established stable-Gibbs duality through marginal change-of-measure arguments. Here, we derive the same results via the coupled Poissonian construction with $J=1$ and stable subordinators.
The agreement between these independent approaches provides crucial cross-validation. Section~\ref{sec:general_gibbs_duality} shows \emph{what} the duality laws must be by integrating over marginals with an externally chosen fragmentation operator. This section reveals \emph{why} those laws arise from the endogenous generative mechanism. This perspective uncovers structure invisible to marginal methods. Proposition~\ref{prop:fraginvariance} establishes invariance of the $\mathrm{PD}(\alpha,-\beta)$ fragmentation law: it remains unchanged whether viewed unconditionally or conditioned on fragment counts $\mathscr{F}^{(H_{\ell})}_{\ell}(\zeta^{1/\beta},1) = n_{\ell}$. This invariance, unexplained from the marginal perspective, emerges naturally from the Poissonized construction as an intrinsic property.
Beyond validation, this section provides a concrete example of the general machinery from Sections~\ref{sec:coupledmixedPoisson}--\ref{sec:section:PK} and leverages specialized distributional tools from~\cite{JamesStick}. Specifically, we employ general projective distributional maps (Proposition 1.1, Corollary 2.1, Lemma 2.1) applicable to arbitrary subordinators. These maps enable the vast array of distributions in Theorem~\ref{thm:master_duality} and Proposition~\ref{prop:fraginvariance}, establishing a principle extendable to any subordinator via $\mathbb{P}^{[n]}$ laws.

\subsection{A Tale of Two Views: The Coupled Process Construction}
We begin by constructing the coupled processes that form the heart of the framework, presenting them from two complementary perspectives.

\subsubsection{Fine-to-Coarse View}
Let $(s_{k},U_{k})_{k\ge 1}$ be the points of a Poisson random measure (PRM) with mean measure $\rho_{\alpha}(s)ds\,du$. Independently, let $(\lambda_{l},Y_{l})_{l\ge 1}$ be the points of a PRM with mean measure $\rho_{\frac{\beta}{\alpha}}(\lambda)d\lambda\,dy$, where $(\lambda_{l})$ are the ranked jumps of a stable subordinator $\sigma_{\frac{\beta}{\alpha}}$. Precisely,  for our construction, the corresponding unconditioned mass partitions are
$$
\mathbf{P}_{\alpha}:=\left(\frac{s_{k}}{\sigma_{\alpha}(\sigma_{\beta/\alpha}(1))}\right)_{k\ge 1}\sim \mathrm{PD}(\alpha,0), \quad 
\mathbf{Q}_{\frac{\beta}{\alpha}}:=\left(\frac{\lambda_{l}}{\sigma_{\frac{\beta}{\alpha}}(1)}\right)_{l\ge 1}\sim \mathrm{PD}(\frac{\beta}{\alpha},0).
$$
where here the  $(s_{k})$ are the ranked jumps of  stable subordinator $\sigma_{\alpha},$ run up to the random time $\sigma_{\beta/\alpha}(1).$ Note that the setup here is similar to \cite[Section 4.4.4]{BerFrag} where it follows from 
Lemma 4.11 of that work, see also~\cite{PY97}, that $\mathbf{P}_{\alpha}$ has the $\mathrm{PD}(\alpha,0)$ law under this scaling independent of ~$\sigma_{\beta/\alpha}(1).$ The composition $\sigma_\beta = \sigma_{\alpha} \circ \sigma_{\frac{\beta}{\alpha}}$ defines the coarse-grained process. The normalized jumps of this composed subordinator, $(z_l)_{l \ge 1}$, form the coarse-grained mass partition which is the ultimate target of this view:
$$
\mathbf{P}_{\beta}:=\left(\frac{z_{l}}{\sigma_{\beta}(1)}\right)_{l\ge 1}\sim \mathrm{PD}(\beta,0).
$$
We build the coupled processes $(I,Z)$ for $y \in [0,1]$ and a time/intensity parameter $\zeta > 0$:
\begin{equation}
\label{stablecoursetofine}
(I(\zeta^{1/\beta},y),Z(\zeta^{1/\beta},y)):= \left( \sum_{k=1}^{\infty}\mathscr{P}_{1,k}(s_{k}\zeta^{1/\beta})\mathbb{I}_{\{U_{k}\leq y\}}, \sum_{k=1}^{\infty}\mathscr{P}_{1,k}(s_{k}\zeta^{1/\beta})\mathbb{I}_{\{w_{k}\leq y\}} \right),
\end{equation}
where $(\mathscr{P}_{1,k})$ are independent standard Poisson processes and $w_{k}=G^{-1}_{\frac{\beta}{\alpha},0}(U_{k})$ maps the uniform labels of the fine partition to the labels of the coarse partition.

\subsubsection{Coarse-to-Fine View}
We now view the same system through the lens of composition. The composed $\beta$-stable subordinator $\sigma_\beta=\sigma_{\alpha}\circ\sigma_{\frac{\beta}{\alpha}}$ and its jumps $(z_l)_{l \ge 1}$ are the starting point. Each coarse jump $z_l$ can be represented as $z_l \overset{d}{=} \sigma_{\alpha}(\lambda_l)$, where $(\lambda_l)$ are the jumps of the directing subordinator $\sigma_{\frac{\beta}{\alpha}}$. More formally, for each $l$, the jump $z_l$ is the sum of points $(s_{k,l})_{k \ge 1}$ from a Poisson process with L\'evy density $\lambda_{l}\rho_{\alpha}(s)ds$. This gives the coarse-to-fine construction:
\begin{equation}
\label{stableCoarsetoFine}
(I(\zeta^{1/\beta},y),Z(\zeta^{1/\beta},y)):= \left( \sum_{l=1}^{\infty}\sum_{k=1}^{\infty}\mathscr{P}_{1,k,l}(s_{k,l}\zeta^{1/\beta})\mathbb{I}_{\{U_{k,l}\leq y\}}, \sum_{l=1}^{\infty}\left[\sum_{k=1}^{\infty}\mathscr{P}_{1,k,l}(s_{k,l}\zeta^{1/\beta})\right]\mathbb{I}_{\{Y_{l}\leq y\}} \right).
\end{equation}
From this view, we can construct the fragmenting mass partitions for each coarse block $l$, whose law is the central object of interest:
$$
\left(Q_{k,l}=\frac{s_{k,l}}{\sigma_{\alpha,l}(\lambda_{l})}\right)_{k\ge 1} \quad \text{with corresponding bridge} \quad F_{1,l}(y)=\sum_{k=1}^{\infty}Q_{k,l}\mathbb{I}_{\{U_{k,l}\leq y\}}.
$$

\begin{rem}[The Law of the Fragments]\label{rem:stablefFragdist}
In this stable setting, one can formally show that the fragmenting bridges $(F_{1,l})_{l\ge 1}$ are i.i.d.\ with law $\mathrm{PD}(\alpha,-\beta)$. This follows from analyzing the joint law of a composed jump $\sigma_{\alpha}(\lambda)$ and the underlying jump $\lambda$, which is $\mathbb{P}(\sigma_{\alpha}(\lambda)\in dz)\rho_{\frac{\beta}{\alpha}}(\lambda)d\lambda$. A change of variable $t=z\lambda^{-1/\alpha}$ reveals the conditional law of the scaled jump to be governed by the $\beta$-tilted density $f_{\alpha,-\beta}(t)=t^{\beta}f_{\alpha}(t)/\E[\sigma^{\beta}_{\alpha}(1)]$. This argument, a variant of~\cite[Proposition 33, (iii)]{PY97}, confirms that the operator is indeed $\mathrm{Frag}_{\alpha,-\beta}$.
\end{rem}

\subsection{The Four-Component Process in the Stable Setting}
We begin by constructing the four coupled processes that form the heart of the PHIBP framework, specializing the general construction from Section~\ref{sec:coupledmixedPoisson} to the stable case for a single group ($J=1$). The building blocks are two independent stable subordinators:
\begin{itemize}
    \item The coagulating subordinator, $\sigma_{\beta/\alpha}$, is a stable subordinator of index $\beta/\alpha$, with $0 < \beta < \alpha < 1$. Its Laplace exponent is $\Psi_0(\gamma) = \gamma^{\beta/\alpha}:=\psi_{\frac{\beta}{\alpha}}(\gamma)$, and its L\'evy density is $\rho_{\beta/\alpha}(\lambda)$.
    \item The fine-grained subordinator, $\sigma_{\alpha}$, is a stable subordinator of index $\alpha$. Its Laplace exponent is $\psi_1(\gamma) = \gamma^\alpha:=\psi_{\alpha}(\gamma)$, and its L\'evy density is $\rho_{\alpha}(s)$.
\end{itemize}
The composition $\sigma_{\alpha} \circ \sigma_{\beta/\alpha}$ is a stable subordinator of index $\alpha \cdot (\beta/\alpha) = \beta$, which we denote $\sigma_\beta$.

Following Theorem~\ref{coupledsupertheorem}, we define the four-component process for a time/intensity parameter $\zeta^{1/\beta} > 0$ and a label-space parameter $y \in [0,1]$.

\begin{equation} \label{eq:stable_4_tier_representation}
\begin{pmatrix}
I(\zeta^{1/\beta}, y) \\[2.5ex]
\mathscr{A}(\zeta^{1/\beta}, y) \\[2.5ex]
\left(\mathscr{F}^{(H_{\ell})}_{\ell}(\zeta^{1/\beta}, y)\right)_{\ell\ge 1} \vphantom{\begin{cases} \sum_{k=1}^{X_{\ell}} C_{k,\ell} \mathbb{I}_{\{\tilde{U}_{k,\ell}\leq y\}} \\ 0 \end{cases}} \\[2.5ex]
Z(\zeta^{1/\beta}, y)
\end{pmatrix}
:=
\begin{pmatrix}
\sum_{k=1}^{\mathscr{A}(\zeta^{1/\beta},1)}C_{k}\mathbb{I}_{\{\tilde{U}_{k}\leq y\}} \\[2.5ex]
\sum_{\ell=1}^{\varphi}X_{\ell}\mathbb{I}_{\{\tilde{Y}_{\ell}\leq y\}} \\[2.5ex]
\begin{cases}
    \sum_{k=1}^{X_{\ell}} C_{k,\ell} \mathbb{I}_{\{\tilde{U}_{k,\ell}\leq y\}} & \text{for } \ell \in \{1, \dots, \varphi\} \\
    0 & \text{for } \ell > \varphi
\end{cases}\\[2.5ex]
\sum_{\ell=1}^{\varphi}\left( \sum_{k=1}^{X_{\ell}} C_{k,\ell} \right) \mathbb{I}_{\{\tilde{Y}_{\ell}\leq y\}}
\end{pmatrix}
\end{equation}
The distributions of the random variables are direct specializations of the general framework, driven by the latent variable $\zeta$:
\begin{itemize}
    \item The number of coarse blocks, $\varphi$, follows a 
    $$
    \mathrm{Poisson}(\psi_{\frac{\beta}{\alpha}}(\psi_{\alpha}(\zeta^{1/\beta}))) = \mathrm{Poisson}(\psi_{\beta}(\zeta^{1/\beta})) = \mathrm{Poisson}(\zeta)
    $$
    distribution.
    
    \item The counts of fine blocks within each coarse block, $(X_\ell)_{\ell \ge 1}$, are i.i.d. from $\mathrm{MtP}(\rho_{\beta/\alpha}, \psi_{\alpha}(\zeta^{1/\beta})) = \mathrm{MtP}(\rho_{\beta/\alpha}, \zeta^{\alpha/\beta})$. With
    $$
    \psi_{\frac{\beta}{\alpha}}^{(x_{\ell})}(\zeta^{\alpha/\beta})=\int_{0}^{\infty}s^{x_{\ell}}e^{-\lambda\zeta^{\alpha/\beta}}\rho_{\frac{\beta}{\alpha}}(\lambda)d\lambda=\frac{\beta\Gamma(x_{\ell}-\frac{\beta}{\alpha})\zeta^{1-\frac{\alpha}{\beta}x_{\ell}}}{\alpha\Gamma(1-\frac{\beta}{\alpha})}
    $$
    and
    $$
    \mathbb{P}(X_{\ell}=x_{\ell})=\frac{\beta\Gamma(x_{\ell}-\frac{\beta}{\alpha})}{\alpha\Gamma(1-\frac{\beta}{\alpha})x_{\ell}!}
    $$
    \item The counts of individuals within each fine block, $(C_k)$ and equivalently $(C_{k,\ell})$, are i.i.d. from $\mathrm{MtP}(\rho_\alpha, \zeta^{1/\beta})$, With
    $$
    \psi_{{\alpha}}^{(c_{k,\ell})}(\zeta^{1/\beta})=\int_{0}^{\infty}s^{c_{k,\ell}}e^{-\lambda\zeta^{1/\beta}}\rho_{{\alpha}}(s)ds=\frac{\alpha\Gamma(c_{k,\ell}-{\alpha})\zeta^{\frac{\alpha}{\beta}-\frac{1}{\beta}c_{k,\ell}}}{\Gamma(1-{\alpha})}
    $$
    and
    $$
    \mathbb{P}(C_{k,\ell}=c_{k,\ell})=\frac{\alpha\Gamma(c_{k,\ell}-{\alpha})}{\Gamma(1-{\alpha})c_{k,\ell}!}
    $$
\end{itemize}
This construction places the stable duality within our general framework. The processes $I$ and $Z$ correspond to the fine and coarse count processes, $\mathscr{A}$ is the allocation process defining the fine partition, and $\mathscr{F}$ defines the fragmentation operator.

\begin{prop} The distribution of the fragmentation counts $(\mathscr{F}^{(H_{\ell})}_{\ell}(\zeta^{1/\beta},1),\ell\in[r])$ are iid $\mathrm{MtP}(\rho_{\beta},\zeta^{1/\beta})$
where 
$$
   {(\Psi_{0}\circ\psi_{1})}^{(n_{\ell})}(\zeta^{1/\beta})= \psi_{\beta}^{(n_{\ell})}(\zeta^{1/\beta})=\int_{0}^{\infty}z^{n_{\ell}}e^{-z\zeta^{1/\beta}}\rho_{{\beta}}(z)dz=\frac{\beta\Gamma(n_{\ell}-{\beta})\zeta^{1-\frac{1}{\beta}n_{\ell}}}{\Gamma(1-{\beta})}
    $$
    and
    $$
    \mathbb{P}(\mathscr{F}^{(H_{\ell})}_{\ell}(\zeta^{1/\beta},1)=n_{\ell})=\frac{\beta\Gamma(n_{\ell}-{\beta})}{\Gamma(1-{\beta})n_{\ell}!}
    $$
 
    \begin{enumerate}
    \item The distribution of $\zeta^{\alpha/\beta}H_{\ell}$ given $X_{\ell}=x_{\ell}$ and $\mathscr{F}^{(H_{\ell})}_{\ell}(\zeta^{1/\beta},1)=n_{\ell}$ is $\mathrm{Gamma}(x_{\ell}-\frac{\beta}{\alpha},1)$ 
    \item The distribution of $X_{\ell}$ given $\mathscr{F}^{(H_{\ell})}_{\ell}(\zeta^{1/\beta},1)=n_{\ell}$ equates to the number of distinct blocks in a $\mathrm{PD}(\alpha,-\beta)$ partition of $[n_{\ell}]$ denote this random variable as $K^{[\alpha,-\beta]}_{n_{\ell}}$ with probability mass function for $k=1,\ldots,n_{\ell},$
\begin{equation}
\label{alphabetafragblockdist}
     \mathbb{P}^{(n_{\ell})}_{\alpha,-\beta}(k)=\frac{\Gamma(n_{\ell})\Gamma(k-\frac{\beta}{\alpha})\Gamma(1-\beta)}{\Gamma(k)\Gamma(1-\frac{\beta}{\alpha})
     \Gamma(n_{\ell}-\beta)}\mathbb{P}^{(n_{\ell})}_{\alpha}(k)
    \end{equation}
    \item It follows that $\zeta^{\alpha/\beta}H_{\ell}$ given $\mathscr{F}^{(H_{\ell})}_{\ell}(\zeta^{1/\beta},1)=n_{\ell}$ is equal in distribution to 
    $$
    G_{K^{[\alpha,-\beta]}_{n_{\ell}}-\frac{\beta}{\alpha}}\overset{d}={G^{\alpha}_{n_{\ell}-\beta}}\times {T^{-\alpha}_{\alpha,-\beta}}
    $$
    where $T^{-\alpha}_{\alpha,-\beta}\sim \mathrm{ML}(\alpha,-\beta)$ independent of $G_{n_{\ell}-\beta}\sim \mathrm{Gamma}(n_{\ell}-\beta,1)$ More details of this distributional identity can be found in \cite[Corollary 2.1 and section 4.4]{JamesStick}.
    \item That is its density can be expressed as
\begin{equation}
\label{scaleHdenstiystablecase}
f_{\zeta^{\alpha/\beta}H_{\ell}}(w|n_{\ell},\zeta^{1/\beta})=\sum_{x_{\ell}=1}^{n_{\ell}}   \mathbb{P}^{(n_{\ell})}_{\alpha,-\beta}(x_{\ell})\frac{w^{x_{\ell}-\frac{\beta}{\alpha}-1}e^{-w}}{\Gamma(x_{\ell}-\frac{\beta}{\alpha})}
\end{equation}
\end{enumerate}
\end{prop}
\begin{proof}The distribution of the $(\mathscr{F}^{(H_{\ell})}_{\ell}(\zeta^{1/\beta},1),\ell\in[r])$  follows as a special case of  Lemma~\ref{post:marginalofN} utilizing the fact that $\sigma_{\alpha}\circ\sigma_{\frac{\beta}{\alpha}}=\sigma_{\beta}.$ Item 1. can be read from 1. of Proposition~\ref{HgivenXN}. Item 2.follows from part 2. of that result where the explicit distribution of $\mathscr{F}^{(H_{\ell})}_{\ell}(\zeta^{1/\beta},1)$ makes this more transparent. Items 3. and 4. follows from 1. and  2. 
\end{proof}

\begin{rem}For the result above we used the known fact that, for the stable case, the sum's distribution is given by generalized Stirling numbers, $S_{\alpha}(n,\ell)$:
\begin{equation}
\label{genStirling}
\mathbb{P}\left(\sum_{k=1}^{x_{\ell}}C_{k,\ell}=n_{\ell}\right)=\frac{\alpha^{x_{\ell}}x_{\ell}!}{n_{\ell}!}S_{\alpha}(n_{\ell},x_{\ell}).
\end{equation}
See \cite[p. 66, 72]{Pit06} and also~\cite{JamesStick,Pit97}.
\end{rem}

We can now describe the law of the fragmenting mass partition, which has the remarkable invariance property that its uncondontional and conditional laws are the same. This will also give us an opportunity to provide some clarifying calculations.

\begin{prop}\label{prop:fraginvariance}
Recall the definition from \eqref{StablePowerdist} (setting $n=n_{\ell}$):
\begin{equation}
\label{StablePowerdist2}
\mathbb{P}^{[n_{\ell}]}_{\alpha}(\lambda^{1/\alpha}\gamma) = \int_{0}^{\infty} \mathrm{PD}(\alpha|t) \dfrac{ (\lambda^{1/\alpha}\gamma)^{n_{\ell}-\alpha} t^{n_{\ell}} e^{-\lambda^{1/\alpha}\gamma t} e^{\lambda\gamma^{\alpha}} f_{\alpha}(t) }{\alpha\sum_{j=1}^{n_{\ell}}\mathbb{P}_{\alpha}^{(n_{\ell})}(j)\frac{(\lambda\gamma^{\alpha})^{j-1}\Gamma(n_{\ell})}{\Gamma(j)} } dt.
\end{equation}
\begin{enumerate}
    \item The distribution of $(Q_{k,\ell})_{k \geq 1} \sim \mathrm{PD}(\alpha, -\beta)$.
    
    \item The distribution of $(Q_{k,\ell})_{k \geq 1} \mid \mathscr{F}^{(H_{\ell})}_{\ell}(\zeta^{1/\beta}, 1) = n_{\ell}$ is:
    \begin{equation}
    \label{PDidFrag}
    \mathbb{E}\left[\mathbb{P}^{[n_{\ell}]}_{\alpha}\Big(G^{1/\alpha}_{K^{[\alpha, -\beta]}_{n_{\ell}} - \frac{\beta}{\alpha}}\Big)\right] = \mathrm{PD}(\alpha, -\beta).
    \end{equation}

    \item The distribution of $(Q_{k,\ell})_{k\ge 1}$ given $K^{[\alpha,-\beta]}_{n_{\ell}}=k$ (or equivalently given $X_{\ell}=k$) and $\mathscr{F}^{(H_{\ell})}_{\ell}(\zeta^{1/\beta}, 1) = n_{\ell}$ is
    \[
    \mathbb{E}\left[\mathbb{P}^{[n_{\ell}]}_{\alpha}\Big(G^{1/\alpha}_{k - \frac{\beta}{\alpha}}\Big)\right].
    \]
\end{enumerate}
\end{prop}

\begin{proof}
Statement (i) has already been verified. Statement (ii) follows from \cite[Corollary 2.1 and section 4.4]{JamesStick}. However, since these results are not well known and will be invoked again, we provide details for clarity. Using the density in \eqref{scaleHdenstiystablecase}, we may focus on terms in \eqref{StablePowerdist2} involving directly its argument:
$$
 \frac{ (w^{1/\alpha})^{n_{\ell}-\alpha} t^{n_{\ell}} e^{-w^{1/\alpha} t} e^{w} f_{\alpha}(t) }{\alpha\sum_{j=1}^{n_{\ell}}\mathbb{P}_{\alpha}^{(n_{\ell})}(j)\frac{w^{j-1}\Gamma(n_{\ell})}{\Gamma(j)} }\times f_{\zeta^{\alpha/\beta}H_{\ell}}(w|n_{\ell},\zeta^{1/\beta}).
$$
Given the definition in~\eqref{alphabetafragblockdist}, after cancellation and integration with respect to $w$, the expression simplifies to: 
$$
\frac{\Gamma(1-\beta)}{\Gamma(1-\frac{\beta}{\alpha})
     \Gamma(n_{\ell}-\beta)}\int_{0}^{\infty}y^{n_{\ell}-\beta-1}e^{-yt}dy \times t^{n_{\ell}}f_{\alpha}(t)
$$
which equates to the density of $T_{\alpha,-\beta}$, completing the proof.
\end{proof}

\begin{rem}
The proof above demonstrates a particular case of the projective distributional maps developed in full generality in \cite[Proposition 1.1]{JamesStick}. These maps apply to any power-biased variable and can therefore be used for all subordinators in this work, not only for the stable case. What we showed above was a projection from a power $n$ index to a power $\beta$ index. More generally, the maps hold for any pair of indices $\vartheta_2 > \vartheta_1$, provided the relevant variables exist, but offer particular advantages in terms of explicit representations when working with an integer index $n > \vartheta_2$. Application to the general $J$-group case is also possible, although we do not explore that possibility here.  See \cite{JamesStick} for more details.
\end{rem}
\subsection{Poissonization and the Bridge to Finite Samples}
The total number of individuals observed in the system, $I(\zeta^{1/\beta},1)$, is a mixed Poisson variable. Following~\cite{PitmanPoissonMix}, we connect this to the arrival time of the $n$-th individual, $T_{n}=\Gamma_{n}/\sigma_{\beta}(1)$, via the fundamental relationship:
\begin{equation}
\{I(\zeta^{1/\beta},1) \ge n\} \iff \{T_{n} \le \zeta^{1/\beta}\}.
\end{equation}
This leads to the switching identity, $\mathbb{P}(T_{n}\in d\gamma)/d\gamma = (n/\gamma)\mathbb{P}(I(\gamma,1)=n)$, which guarantees $\mathbb{P}(T_n < \infty)=1$. The stable case provides an explicit form for the law of this arrival time, which is the final component needed for our main theorem.
\begin{lem}[The Arrival Time Distribution]\label{lem:stablearrival}
As established in~\cite[Section 4.4]{JamesStick}, the scaled arrival time has the representation $T_{n} \overset{d}{=} G^{1/\beta}_{K^{[\beta]}_{n}}$, where $K^{[\beta]}_{n}$ is the number of blocks in a $\mathrm{PD}(\beta,0)$ partition of $[n]$. It follows that $\mathbb{P}(T_n \leq \zeta^{1/\beta}) = \mathbb{P}(G_{K^{[\beta]}_{n}} \leq \zeta)$, where $G_k \sim \mathrm{Gamma}(k,1)$.
\begin{enumerate}
    \item The density of $G_{K_n^{[\beta]}}$ is given by $f_{G_{K_n^{[\beta]}}}(\zeta) = e^{-\zeta} \sum_{k=1}^{n} \mathbb{P}_{\beta}^{(n)}(k) \frac{\zeta^{k-1}}{\Gamma(k)}$.
    \item The probability of observing exactly $n$ individuals in the coarse partition is
    \begin{equation}\label{stableNpoisson}
    \mathbb{P}(\mathscr{P}(\sigma_{\beta}(\zeta))=n)=\frac{\Gamma(n)}{n!}\beta e^{-\zeta}\sum_{k=1}^{n}\mathbb{P}^{(n)}_{\beta}(k)\frac{\zeta^{k}}{\Gamma(k)}.
    \end{equation}
    \item The probability of observing exactly $K$ blocks in the fine partition is
    \begin{equation}\label{stableKpoisson}
    \mathbb{P}(\mathscr{A}(\zeta^{1/\beta},1)=K)=\mathbb{P}(\mathscr{P}(\sigma_{\frac{\beta}{\alpha}}(\zeta))=K)=\frac{\beta\Gamma(K)}{\alpha K!} e^{-\zeta}\sum_{j=1}^{K}\mathbb{P}^{(K)}_{\frac{\beta}{\alpha}}(j)\frac{\zeta^{j}}{\Gamma(j)}.
    \end{equation}
\end{enumerate}
\end{lem}

\subsection{The Master Duality Equation}
We now assemble these components. The PHIBP framework specifies the joint law of the combinatorial structures conditional on the latent variable $\zeta$. This leads to our central result, which reveals the Poissonized joint EPPF of the system, conditioned on the event that the latent time variable is $G_{K^{[\beta]}_{n}}=\zeta$.

\begin{thm}[The Generative Master Equation of Duality]
\label{thm:master_duality}
Consider the PHIBP processes in the stable setting. The joint law of the combinatorial components, conditioned on the event that the latent time variable is $G_{K^{[\beta]}_{n}}=\zeta$, satisfies the following master identity:
\begin{equation}
\label{PHIBPGibbsduality}
\begin{aligned}
    &\underbrace{\frac{ p_{\frac{\beta}{\alpha}}(x_{1},\ldots,x_{r}) \frac{\zeta^{r}}{\Gamma(r)} }{ \sum_{j=1}^{K}\mathbb{P}^{(K)}_{\frac{\beta}{\alpha}}(j)\frac{\zeta^{j}}{\Gamma(j)}}}_{\text{Coagulator EPPF, } p(\pi^{(2)}|\pi^{(1)},\zeta)} \times 
    \underbrace{\frac{p_{\alpha}(c_{1},\ldots,c_{K}) \sum_{j=1}^{K}\mathbb{P}^{(K)}_{\frac{\beta}{\alpha}}(j)\frac{\zeta^{j}}{\Gamma(j)}}{\sum_{k=1}^{n}\mathbb{P}^{(n)}_{\beta}(k)\frac{\zeta^{k}}{\Gamma(k)}}}_{\text{Fine Partition EPPF, } p(\pi^{(1)}|n,\zeta)} \times f_{G_{K_n^{[\beta]}}}(\zeta)\\
    &\qquad = \underbrace{\prod_{l=1}^{r} p_{\alpha,-\beta}(\mathbf{c}_{l})}_{\text{Fragmentation EPPF, } p(\pi^{(1)}|\pi^{(2)})} \times \underbrace{\frac{p_{\beta}(n_{1},\ldots,n_{r}) \frac{\zeta^{r}}{\Gamma(r)}}
    {\sum_{k=1}^{n}\mathbb{P}^{(n)}_{\beta}(k)\frac{\zeta^{k}}{\Gamma(k)}}}_{\text{Coarse Partition EPPF, } p(\pi^{(2)}|n,\zeta)}
    \times f_{G_{K_n^{[\beta]}}}(\zeta).
\end{aligned}
\end{equation}
This equation, expressed in terms of the EPPFs, shows that the fine-to-coarse (LHS) and coarse-to-fine (RHS) perspectives are balanced by the latent time variable $\zeta$. The fragmenter term is independent of $\zeta$, a special feature of the stable case. This master equation has the following consequences:

\begin{description}
    \item[1. Foundational Duality:] To recover the classic duality for $\mathrm{PD}(\alpha,0)$ and $\mathrm{PD}(\beta,0)$, one integrates both sides of the master equation with respect to the mixing density $f_{G_{K_n^{[\beta]}}}(\zeta)d\zeta$ over $(0, \infty)$. This operation marginalizes out the latent time, collapsing the dynamic, Poissonized structure into the static identity.

    \item[2. Pitman-Yor Family :] To recover the duality involving a $\mathrm{PD}(\beta,\theta)$ process for $\theta>-\beta$, the mixing density $f_{G_{K_n^{[\beta]}}}(\zeta)$ is replaced on both sides with the density of the latent time variable corresponding to the $\mathrm{PD}(\beta,\theta)$ process. As detailed in~\cite[Section 4.4]{JamesStick}, this variable has the law of $G_{\frac{\theta}{\beta}+K_{n}^{(\beta,\theta)}}\overset{d}=G^{\beta}_{\theta+n}\times T^{-\beta}_{\beta,\theta}$, where $K_{n}^{(\beta,\theta)}$ is the number of distinct blocks in a $\mathrm{PD}(\beta,\theta)$ partition of $[n]$. 

    \item[3. Landscape of New Dualities:] The principle in (2.) can be generalized. The framework of~\cite[Corollary 2.1, Lemma 2.1]{JamesStick} provides a calculus of conditioned processes, $\mathbb{P}^{[n]}$, each with its own characteristic latent time distribution. By substituting the appropriate density into the master equation, one can generate a new duality relation for any of these processes. This transforms the master equation into a general tool for producing such results.
\end{description}
\end{thm}
\begin{proof}For some further details, as otherwise the results can be deduced from our calculations above, it follows that the coarse partition, which corresponds generally to $p_{\mathrm{coarse}}$ in \eqref{eq:PK_coarse_eppf_J_groups} for $J=1,$ can be expressed as 
$e^{-\psi_{\beta}(\zeta^{1/\beta})}\prod_{\ell=1}^{r}\psi^{(n_{\ell})}_{\beta}(\zeta^{1/\beta})$ divided by $\mathbb{E}[(\sigma_{\beta}(1))^{n}e^{-\zeta^{1/\beta}\sigma_{\beta}(1)}].$ This equates with $p^{[n]}_{\beta}(n_{1},\ldots,n_{r}|\zeta)$ as represented in~\eqref{StablepowerEPPF}. Directly the coag operator is $p^{[K]}_{\frac{\beta}{\alpha}}(x_{1},\ldots,x_{r}|\zeta).$ The fragmentation operator arises as a consequence of Proposition~\ref{prop:fraginvariance} or can be computed directly using our general formulae. As to the fine partition EPPF, as a check against our general formulae,  we can derive it directly from Lemma~\ref{lem:stablearrival} where its exchangeably ordered form is given by $\prod_{k=1}^{K}\mathbb{P}(C_{k}=c_{k})\times \mathbb{P}(\mathscr{P}(\sigma_{\frac{\beta}{\alpha}}(\zeta))=K)$ divided by  $\mathbb{P}(\mathscr{P}(\sigma_{\beta}(\zeta))=n)$ yielding the expression in this Theorem. 
\end{proof}

\subsection{Identification of the  Duality \texorpdfstring{$\mathbb{P}^{[n]}$}{P[n]} Laws}
For fixed $\zeta$, Theorem~\ref{thm:master_duality} serves to further solidify what we have established in our developments in 
Section~\ref{sec:section:PK}. The surprise is, of course, the invariance of the $\mathrm{PD}(\alpha, -\beta)$ laws that act as the fragmenters, subject to conditioning under the counts $\mathscr{F}^{(H_{\ell})}_{j,\ell}$. We summarize the conditioned laws of the other components
Although we can do this otherwise from Theorem~\ref{thm:master_duality} and our developments in Section~\ref{sec:PKdualitylaws}, it is noteworthy that we can use the following fact with $h^{[n]}_{\beta,\zeta}(t)$ defined in~\eqref{hforStablepowerbias}
$$
\mathbb{P}^{[n]}_{\beta}(\zeta^{1/\beta}) =\mathrm{PK}_{\beta}(h^{[n]}_{\beta,\zeta}\cdot f_{\beta}){\mbox { with }}
\Phi^{[\beta]}_{n,r}:=\frac{\frac{\zeta^{r}}{\Gamma(r)}}
    {\sum_{k=1}^{n}\mathbb{P}^{(n)}_{\beta}(k)\frac{\zeta^{k}}{\Gamma(k)}}
$$
in conjunction with~\cite{HoJamesLau2025} as displayed in Section~\ref{sec:general_gibbs_duality}, to immediately conclude the following results:
\begin{enumerate}
\item The distribution of the coarse mass partition $\mathbf{V}_{\beta}$ given $I(\zeta^{1/\beta},1)=Z(\zeta^{1/\beta},1)=n,$ is
$$
\mathbb{P}^{[n]}_{\beta}(\zeta^{1/\beta})
$$
\item The distribution of the fine mass partition $\mathbf{P}_{\alpha}=\mathrm{Frag}_{\alpha,-\beta}(\mathbf{V}_{\beta})$  given $I(\zeta^{1/\beta},1)=Z(\zeta^{1/\beta},1)=n,$ is, as a special case of~\eqref{lawofPDfine},
\begin{equation}
\label{lawofPDfinepowern}
\int_{0}^{\infty}\mathrm{PD}(\alpha|s) \left[ \int_{0}^{\infty}h^{[n]}_{\beta,\zeta}(sy^{1/\alpha})f_{\frac{\beta}{\alpha}}(y)dy \right] f_{\alpha}(s)ds.
\end{equation}

\item The distribution of the coagulator $\mathbf{Q}_{\frac{\beta}{\alpha}}$ given $\mathscr{A}(\zeta^{1/\beta},1)=K,$ and $I(\zeta^{1/\beta},1)=n,$
is 
$$
\mathbb{P}^{[K]}_{\frac{\beta}{\alpha}}(\zeta^{\alpha/\beta})
$$
\end{enumerate}

\begin{rem}
See \cite[Section 6]{HoJamesLau2025} for related details.
\end{rem}

\section{The Dynamic Framework: An Extension of the Duality}\label{sec:DynamicCouplin}
We now formalize the continuous-time framework previewed in Section~\ref{section1:Simtimefragcoag}. The static, four-part duality established in Theorem~\ref{coupledsupertheorem} is built upon an intrinsic L\'evy-It\^o structure. Recognizing this foundational property allows for the immediate and natural extension of the entire framework to continuous time.
This extension is achieved by leveraging the established theory of homogeneous fragmentations developed by Bertoin~\cite{BerFrag}. The move to continuous time becomes a direct matter of identifying our mean measure $\nu(d\mathbf{t}, d\lambda) = \left( \prod_{j=1}^J \mathbb{P}(\sigma_j(\lambda) \in dt_j) \right) \tau_0(\lambda) d\lambda,$ see~\eqref{eq:joint_measure_def} and   \ref{eq:joint_measure_decomp}, as being "analogous" to the  dislocation measure within his framework. This act of identification—a consequence of our framework's design—alleviates the technical burden of constructing and verifying the time dynamics from first principles. The importance of this dynamic extension lies in its implications. It yields tractable and novel duality couplings between $\Lambda$-Fleming-Viot processes, structured $\Lambda$-coalescents, and Homogeneous Fragmentation Processes (HFP), with computational tractability enabled by the mixed Poisson representations of Theorem~\ref{coupledsupertheorem}. Furthermore, this dynamic view solidifies a fundamental shift in the modeling perspective: it moves the focus from time-indexed mass partitions, which describe relative abundances, to a more robust model of absolute abundance rates. These rates are dictated by the jumps of the underlying subordinators over disjoint time intervals, providing a direct, mechanistic description of innovation and population growth through time. In essence, this is a continuous-time realization of the ever-present but almost forgotten species sampling component articulated by Pitman~\cite{PitmanPoissonMix}. 

To amplify our use of the term analogous in reference to $\nu,$ it induces laws on spaces of mass partitions which is the usual framing of dislocation measures. When subordinators do not permit zeros this is effectively one to one. However, as a subtle point our mean abundance rates perspective via jumps endures when proper definition of relative rates based on zero total mass does not. In effect one can push to normalization on the thinned spaces off of zero sets as described in Section~\ref{sec:posterior_decomp}. There is also the subtlety of issues addressed by a rates based approach over a comopositional approach as discussed in the static PHIBP application to microbiome data in~\cite{hibp25}.

\subsection{The Two-Layered Temporal Structure}
The conceptual foundation for this dynamic view is the recognition that our hierarchical model is built upon a family of subordinated L\'evy processes. The system's evolution is not governed by a single, absolute time $t.$ Instead, its temporal structure is two-layered:

\begin{itemize}
\item A single global subordinator, $\sigma_{0}(t)$ serves as a system-wide master clock. Its jumps correspond to major events that affect all populations simultaneously, representing a shared evolutionary or demographic history. 
\item A family of independent local subordinators, $\sigma_j(t)$, act as local clocks. Each is specific to a population $j\in [J]$ and marks events relevant only to that population's unique trajectory.
\end{itemize}
The critical mechanism is subordination: the local clocks do not advance according to absolute time $t$, but according to the time measured by the master clock $\sigma_{0}(t).$ The complete description of the system's state is captured by the joint evolution:
\begin{equation}
\label{eq:subordination_revised}
\left( \sigma_j(\sigma_0(t)), \quad j \in {1, \dots, J}, \quad \sigma_0(t) \right)_{t \ge 0}.
\end{equation}
Including $\sigma_{0}(t)$ explicitly is essential, as it carries the information about the timing and magnitude of shared, system-wide events.
\subsection{The Poisson Random Measure Construction}
The formal bridge from this conceptual structure to a well-defined, continuous-time stochastic process is forged by extending the model's foundational Poisson Random Measure (PRM) from a static space of potential jumps to a dynamic space of time and jumps. This yields the central mathematical construction:
\begin{equation}
\label{homfragdynamics}
dt \otimes \nu(d\mathbf{t}, d\lambda).
\end{equation}
This object is the engine of our entire dynamic theory:
\begin{itemize}
\item The term $dt$ represents the infinitesimal flow of objective, external time.
\item The L\'evy measure $\nu(d\mathbf{t},d\lambda)$ on  the product space encodes the full universe of potential changes: $\mathbf{t}$ represents event types (which sub-population is affected), while $\lambda$ represents event magnitudes (the size of a fragmentation or coalescence).
\item The tensor product $\otimes$ constructs a new measure on the product space of time and events, signifying that for any infinitesimal interval $dt,$ the process has access to the entire set of possible events specified by $\nu$.
\end{itemize}
This construction is rigorously justified by the L\'evy-It\^o decomposition. This theorem guarantees that a process driven by a PRM with mean measure \eqref{homfragdynamics} is a L\'evy process, inheriting the properties of stationary and independent increments. The resulting dynamics are time-homogeneous, providing a principled foundation for modeling steady-state evolutionary dynamics in the spirit of Bertoin's~\cite{BerFrag} framework for homogeneous fragmentations (HFP).

\subsection{Coupled duality for $\Lambda$-Fleming Viot, $\Lambda$-Coalescents and HFP processes}
The static construction of Theorem~\ref{coupledsupertheorem} extends naturally to continuous time. At $t=1,$ the coupled processes were defined via the jumps $(s_{j,k})_{k\ge 1},$ $(\lambda_{l})_{l\ge 1},$ $(\sigma_{j,l}(\lambda_{l}))_{l\ge 1}$ and $((s_{j,k,l})_{k\ge 1})_{l\ge 1}=(s_{j,k})_{k\ge 1}$ and corresponding masses $(U_{j,k,l},Y_{l})$ encode corresponding couplings and dualities on spaces of random bridges and mass partitions as illustrated in Section~\ref{sec:PKdualitylaws}.  We extend this to continuous time by working with $((\sigma_{j}\circ \sigma_{0})(t); j\in[J],\sigma_{0}(t))_{t\ge 0}$ so that constructs over any non-overlapping intervals are independent and otherwise  emulate the static constructions, where, for instance,  
$(I_{j}, \mathscr{A}_{j}, (\mathscr{F}^{(\lambda_l)}_{j,l})_{l \ge 1}, Z_j, j\in [J])([0,t])=(I_{j}, \mathscr{A}_{j}, (\mathscr{F}^{(\lambda_l)}_{j,l})_{l \ge 1}, Z_j, j\in [J])([0,s])+(I_{j}, \mathscr{A}_{j}, (\mathscr{F}^{(\lambda_l)}_{j,l})_{l \ge 1}, Z_j, j\in [J])((s,t])$ is a decomposition over time into two independent pieces where at any fixed time the static properties hold. Hence, this leads to a corresponding coupled sets of bridges say $(F_{j}(\sigma_{0}(t),y), G_{0}(t,y),(F_{j,l}(t,y))_{l\ge 1}, G_{j}(t,y))_{j\in [J]},$ and their coupled mass partitions,  which are a flow of bridges in the sense of \cite{BerLegall00}, satisfying the compositional semi-group property,  and hence are $\Lambda$-Fleming Viot processes that are generator dual to corresponding $\Lambda$-coalescents and by time reversal HFP. Note we have expressed the fine bridges $(F_{j}(\sigma_{0}(t),y), j\in[J])$ as time changed processes, this description carries over to their corresponding HFP, viewed as HFP in the sense of Bertoin~\cite{BerFrag} subject to a further time change. One can work with any of these coupled constructs directly in continuous time to obtain equivalent representations of the same underlying process. The mixed Poisson formulation extends the framework to subordinators that can hit zero with positive probability—a regime where bridge-based constructions are not well-defined. As detailed in Section~\ref{sec:posterior_decomp}, this framework does not require conditioning the process away from zero. Instead, bridges are understood to be well-defined at the jump times of the underlying Poisson process, as these are precisely the moments when the total mass is guaranteed to be positive and normalization is possible. The conditioning on an observation, which is central to the decomposition, inherently singles out these jump events without artificially altering the state space of the underlying subordinator. This is exploited in the fixed time application to microbiome modelling in~\cite{hibp25}.

\begin{rem}[Modularity-extensions]
\label{rem:modularity}
The modularity of our L\'evy-It\^o system allows extensions to abstract Polish spaces $\Omega$ via a mean measure on that space, say $W(d\omega)$, as follows: 
\[
dt \otimes \nu(dt, d\lambda) \otimes W(d\omega).
\]
For instance, $W$ could be It\^o's excursion measure $n(de)$ on the space of excursions, coupling each lineage with an excursion path, or any other $\sigma$-finite measure subject to standard integrability conditions and the creativity of the modeller.
\end{rem}

This principle of subordinated time, made concrete by the measure in \eqref{homfragdynamics}, brings to life four distinct but intimately coupled stochastic flows.

\subsection{The Four-Component Dynamic Flow}
\begin{enumerate}
    \item \textbf{\texttt{Z(t)}: The Structured Coalescent Process (The Backward View of Ancestry)}\\
    Viewed backward in time, the framework defines a structured coalescent process, $Z(t) = (Z_1(t), \dots, Z_J(t))$, on the lineages sampled from the $J$ groups. This process defines a novel class of structured $\Lambda$-coalescents. The subordination principle manifests with mathematical precision in its coagulation law. A fundamental result in the theory of L\'evy processes states that the Laplace exponent of a subordinated process is the composition of the individual exponents. Accordingly, the exponent governing mergers in group~$j$, $\Psi_{\text{coag}}(\gamma)$, is given by $\Psi_0(\psi_j(\gamma))$. This is the direct spectral consequence of the time-domain subordination in \eqref{eq:subordination_revised}. This structure provides the model's key explanatory power: jumps in the local clock $\sigma_j(t)$ drive local, within-group mergers, whereas jumps in the master clock $\sigma_0(t)$ drive synchronous, system-wide merger events, providing a mechanism for correlated demographic histories.

    \item \textbf{\texttt{I(t)} \& \texttt{A(t)}: The Diversification and Allocation Engine (The Forward View of Ecology)}\\
    Viewed forward in time, the very same underlying PRM generates a dual flow modeling the emergence of diversity.
    \begin{itemize}
        \item The \textbf{Fragmentation Flow \texttt{I(t)}} acts as an engine of novelty, describing the continuous birth of new fine-grained features (e.g., alleles, traits). Its rate is governed by the ticking of the subordinated clocks.
        \item The \textbf{Dynamic Allocation Flow \texttt{A(t)}} evolves as a time-continuous Allocation Process, modeling how these new features are partitioned among populations.
    \end{itemize}
    Together, they provide a mechanistic model for ecological diversification and niche construction. The shared master clock, $\sigma_0(t)$, ensures that major bursts of innovation and reallocation can be coupled across otherwise distinct populations.

\item \textbf{\texttt{F(t)}: The Conditional Lineage-Specific Flow and Post-Colonization Adaptation}\\
The framework provides a granular, conditional view of evolution through the process $F^{(H_{\ell})}_{j,\ell}(t)$, which we interpret as a formal model for post-colonization adaptation. This dynamic unfolds in two stages. First, a jump in the master subordinator, $\sigma_0(t)$, establishes a new lineage~$\ell$ across all environments. This system-wide event is analogous to a founder event~\citep{Mayr1954}, which creates the conditions for subsequent divergent evolution. 

Following this founding, the process $F^{(H_{\ell})}_{j,\ell}(t)$ describes the subsequent, private evolutionary trajectory of that lineage as it undergoes local adaptation to the specific selective pressures of environment~$j$~\citep{Kawecki2004}. Taken together, the overall system provides a generative mathematical process for the patterns of adaptive radiation~\citep{Schluter2000}, whereby a single ancestral lineage diversifies to fill distinct ecological niches.
\end{enumerate}

\subsection{Broader Implications, Validation and Fragmentation Time Invariance}
For further clarity, \( (\mathbf{I}(t), \mathbf{Z}(t)) \), when viewed over time \( t \geq 0 \) moving forwards:
\begin{itemize}
\item \( \mathbf{I}(t) \) are time-changed homogeneous fragmentation processes (HFP), with each component driven by the subordinator \( \sigma_j \) but run over the random time \( \sigma_0(t) \);
\item \( \mathbf{Z}(t) \), viewed marginally, are homogeneous fragmentation processes (HFP) based on the composition \( \sigma_j \circ \sigma_0 \), with \( \sigma_0(1) \) replaced by \( \sigma_0(t) \) throughout.
\end{itemize}
Viewing time backwards, the \( \mathbf{Z} \) processes are marginally \( \Lambda \)-coalescents. That is to say subject to time change for $\mathbf{I},$ the marginal distributions of the  component vectors of \( (\mathbf{I}, \mathbf{Z}) \) are familiar objects appearing in the literature ~\cite{BerFrag, Pit06}.

To add for further clarity
\begin{itemize}
\item \( \mathbf{A}(t) \) is the dynamic Allocation process which is encoded by the partitions of \(\mathbf{G}_{0}(t)\), an HFP driven by \(\sigma_{0}\), where the partitions are restrictions to the total block counts of the partition produced by \(\mathbf{I}(t)\).
\item For illustration, we can take \(J=1\), let \([m]=\{1,2,\ldots,m\}\), set \(\Pi_{\mathbf{I},m}(t)\) to be the partition produced by \(\mathbf{I}(t)\) restricted to \([m]\) with corresponding block count \(K_{\mathbf{I},m}(t)\). \(\mathbf{A}_{m}(t)\) is encoded by the information in \(\Pi_{\mathbf{G}_{0},K_{\mathbf{I},m}(t)}(t)\), which is the partition of \(\mathbf{G}_{0}(t)\) restricted to \([K_{\mathbf{I},m}(t)]\).
\item Perform (simultaneously) the Coag or Assembly based on \((\Pi_{\mathbf{I},m}(t),\mathbf{A}_{m}(t))\) equivalently \((\Pi_{\mathbf{I},m}(t),\Pi_{\mathbf{G}_{0},K_{\mathbf{I},m}(t)}(t))\) to get \(\Pi_{\mathbf{Z},m}(t)\).
\item Apply for time \(t+s\), to verify that viewing \(\mathbf{G}_{0}\) as a forward in time (HFP) is crucial to maintain the usual expected dynamics for \(\Pi_{\mathbf{I},m}(t), \Pi_{\mathbf{I},m}(t+s)\) with \(K_{\mathbf{I},m}(t+s)\ge K_{\mathbf{I},m}(t)\) and similarly for \(\Pi_{\mathbf{Z},m}(t),\Pi_{\mathbf{Z},m}(t+s)\) viewed backward in time. If one mistakenly viewed \(\mathbf{G}_{0}\) as a coalescent in forward time, this consistent evolution would break, as the allocation would be merging blocks while the fine partition is fragmenting.
\item The fragmentation operates in an analogous manner.
\end{itemize}
Crucially, at every time instance, this is not a sequential procedure (as described in the static \textbf{Setup~\ref{setup:canonical_duality}}). All components—the fine fragmentation, the allocation, and the coarse assembly—are generated in lockstep.

\begin{rem} (Dynamic Extension and the Hidden Mixed Poisson Structure).  
The continuous-time process version of the coupled mixed Poisson processes in Section~\ref{sec:coupledconstruct} and Theorem~\ref{coupledsupertheorem} is obtained by running the entire static construction through the random clock \( \sigma_0(t) \) for \( t \geq 0 \), with the dynamics dictated by the underlying Homogeneous Fragmentation Process. Crucially, the mixed Poisson framework reveals itself as the hidden sampling component inherent to any species sampling representation of an HFP. Just as a typically expressed HFP evolves over spaces of time-dependent mass partitions, its coupled mixed Poisson process evolves in parallel, encoding the combinatorial structure of partitions sampled from these evolving masses. The following corollaries characterize the snapshot behavior of this coupled system at each fixed \( t \). The distributional time-invariance  means that the statistical properties of the various components are identical at any snapshot in time $t > 0,$ though the specific realizations of the partitions will, of course, differ.
\end{rem}

\begin{cor}[Time-Invariance of the Fragmentation Kernel]
\label{cor:time_invariance}
Consider the continuous-time lift of the coupled process from Theorem~\ref{coupledsupertheorem}, obtained by replacing the L\'evy measure $\tau_0$ of the master subordinator with $t \cdot \tau_0$.
Then, for any $t > 0$, the following components remain time-invariant:
\begin{itemize}
\item The conditional distribution of the \textbf{global jump sizes} $H_{\ell}$ given the associated counts $\vec{n}_{\ell}$,
\item The distribution of all block counts, except for number of blocks dictated by $\varphi.$
\item The \textbf{fragmentation kernel} $p_{\mathrm{frag}}(\cdot \mid \vec{n}_{\ell}, \vec{\gamma})$.
\end{itemize}
The time-dependent quantities are:
\begin{itemize}
\item The \textbf{total number of species} $\varphi(t) \sim \mathrm{Poisson}\left(t\Psi_0\left(\sum_{j=1}^{J} \psi_j(\gamma_j)\right)\right)$,
\item The \textbf{coagulation operator} $p_{\mathrm{coag}}(\cdot\mid\cdot,\vec{\gamma},t)$, based on jumps of the subordinator $\sigma_{0}$ up till time $t$
\item The evolving states $\mathbf{I}(t)$ and $\mathbf{Z}(t)$. and corresponding EPPFs $p_{fine}$ and $p_{coarse}$
\end{itemize}
The duality in Theorem~\ref{coupledsupertheorem} continues to hold at each time $t$ with reflected through~Theorem~\ref{thm:unified_duality_identity} and more generally Section~\ref{Sec:jointEPPF} with $t\tau_{0}$ in place of $\tau_{0}.$
\end{cor}
\begin{cor}[Time-Invariance of Species-Level Fragmentation]
\label{cor:frag_invariance}
Consider the continuous-time lift of the coupled process from Theorem~\ref{coupledsupertheorem} at each fixed time $t$. While three components evolve with the growing number of species $\varphi(t)$:
\begin{itemize}
\item $I_j(\gamma_j, y) = \sum_{\ell=1}^{\varphi(t)} \mathscr{F}^{(H_\ell)}_{j,\ell}(\gamma_j, y)$,
\item $\mathscr{A}_{j}(\gamma_j, y) = \sum_{\ell=1}^{\varphi(t)} X_{j,\ell} \mathbb{I}_{\{\tilde{Y}_{\ell} \leq y\}}$,
\item $Z_j(\gamma_j, y) = \sum_{\ell=1}^{\varphi(t)} \mathscr{F}^{(H_\ell)}_{j,\ell}(\gamma_j, 1)\mathbb{I}_{\{\tilde{Y}_{\ell} \leq y\}}$,
\end{itemize}
the fourth component remains statistically invariant:
\begin{itemize}
\item $\mathscr{F}^{(H_{\ell})}_{j,\ell}(\gamma_j, y)$ has distribution independent of both $t$ and $\varphi(t)$.
\end{itemize}
The fragmentation process for each species is born with identical statistical properties, unaffected by the total number of species in the system.
\end{cor}

We now provide a remark clarifying why this phenomena holds via scale invariance of $\mathrm{MtP}$ (mixed truncated Poisson laws).

\begin{rem}Remark (Componentwise MtP invariance across the four-part system). The MtP distributional invariance under L\'evy scaling applies uniformly to all four components of the coupled system via the inner count variables $C_{j,k}, (C_{j,k,\ell})$ the allocation counts $X_{j,\ell},$ and the species-level fragmenters $\mathscr{F}^{(H_\ell)}_{j,\ell}$— which is directly seen once $\sigma_{j}\circ\sigma_{0}$ is viewed as a single subordinator. Scaling the parent L\'evy measure $\tau_{0}$ to  $t\times \tau_{0}$ only changes the number of observed macro-jumps $\varphi(t)$; the conditional kernels governing per-lineage refinement and allocation remain unchanged. In effect, the entire four-part architecture is, in reference to the fractal structure of the dolls, “Matryoshka-stable”: thinning and scaling decide how many dolls appear, but the kernel painted on each doll—governing the 4 component system does not depend on $t.$ This yields clean separations for inference (global intensity vs. local refinement), robust simulation (size control without re-tuning kernels), and a snapshot-by-snapshot persistence of the coag/frag duality.
\end{rem}

\subsection{Simultaneous Duality for Beta-Coalescents and Time-Changed $\alpha$-Stable HFP}\label{subsec:beta_simultaneous_duality}

For $J=1$ with subordinators $\sigma_{\alpha}(t)$ and $\sigma_{\beta/\alpha}(t)$ ($0 < \beta < \alpha < 1$), the $\mathrm{Beta}(2-\beta,\beta)$-coalescent $Z_{\beta}(t)$ and time-changed $\alpha$-stable HFP $I_{\alpha}(\sigma_{\beta/\alpha}(t))$ emerge as simultaneous duals satisfying the triple ordering~\eqref{eq:ordering}, with state-dependent operators from Sections~\ref{Sec:jointEPPF}-\ref{subsec:marginalized_duality}.

\begin{rem}(Distinction from Time-Reversal Duality)
Unlike Bertoin's time-reversal duality~\citep{BerFrag} (time-reversed $\beta$-stable HFP = $\mathrm{Beta}(2-\beta,\beta)$-coalescent), we establish forward-time coupling: at any fixed $t$, $Z_{\beta}(t)$ (simultaneously a $\mathrm{Beta}(2-\beta,\beta)$-coalescent backward and $\beta$-stable HFP forward) couples with $I_{\alpha}(\sigma_{\beta/\alpha}(t))$ via Theorem~\ref{coupledsupertheorem}. Crucially, $Z_{\beta}(t)$ evolves on physical time while $I_{\alpha}$ runs on subordinated clock $\sigma_{\beta/\alpha}(t)$.
\end{rem}

\subsubsection{$J$-Spider Extension}
Setting $\tau_j(s) = \frac{\theta_j}{\Gamma(1-\alpha)}s^{-\alpha-1}$, $\tau_0(\lambda) = \rho_{\beta/\alpha}(\lambda)$, yields $J$ coordinated pairs $\mathbf{I}(t) = (I_{\alpha,p_j}(\sigma_{\beta/\alpha}(t)))_{j \in [J]}$ and $\mathbf{Z}(t) = (Z_{\beta,p_j}(t))_{j \in [J]}$ with $p_j = \theta_j / \sum_{k=1}^J \theta_k$, where using $\sum_{j=1}^J \sigma_j(y) \overset{d}{=} \sigma_{\alpha}(y) \times \left[\sum_{j=1}^J \theta_j / \alpha\right]^{1/\alpha}$:

\begin{itemize}
\item \textbf{Consistency:} $\sum_{j=1}^J I_{\alpha,p_j}(\sigma_{\beta/\alpha}(t)) \stackrel{d}{=} I_{\alpha}(\sigma_{\beta/\alpha}(t))$, $\sum_{j=1}^J Z_{\beta,p_j}(t) \stackrel{d}{=} Z_{\beta}(t).$
\item \textbf{Spider structure:} Master clock $\sigma_0(t)$ coordinates all $J$ components~\citep{BPY, JamesYor}, with $p_j$ as ray weights.
\item \textbf{Correlation:} Jumps in $\sigma_0(t)$ induce simultaneous cross-population mergers and fragmentations.
\item \textbf{Complete System} This system is completed by corresponding Allocation and Fragmentation processes (\texttt{A(t)}, \texttt{F(t)}).
\item \textbf{Generalized Gamma extensions} As to modelling capabilities, (ARG) applications or otherwise, a more flexible duality structure of Spiders is enabled by generalized gamma process specifications in Remark~\ref{rem:gengammspiders} which is closer to the specifications in~\cite{JamesYor}. Importantly, complex but implementable computational inferential and predictive procedures can be directly adapted from the explicitly executed developments  in~\cite{hibp25}.
\end{itemize}

\begin{rem}(Parameter Range)
Duality requires $0 < \beta < \alpha < 1$ (ensuring $\beta/\alpha \in (0,1)$). As $\beta \to 0$, rescaling $t \to t/\beta$ for the Kingman limit causes $\sigma_{\beta/\alpha}(t/\beta) \to \infty$ (total fragmentation), destroying duality. Non-trivial coupling exists for fixed $\beta \in (0, \alpha)$.
\end{rem}


\begin{rem}[Analogy to Multi-Deme Population Genetics]
The framework provides a new generative engine for multi-deme models in population genetics. The groups $[J]$ are interpreted as distinct demes, the species $(Y_l)$ as alleles, the local subordinators $(\sigma_j)$ as within-deme evolutionary processes (e.g., genetic drift), and the global process $\sigma_0$ as the shared demographic history of the metapopulation. Within this analogy, the sharing of species across groups implicitly models migration or shared ancestry, providing a structure related to multi-type or structured coalescent processes \citep{BerestyckiEtheridgeVeber2013,Johnston2023,mohle2024multi}, including recent models emerging from individual-based dynamics with demographic bottlenecks \citep{PraEtheridge2025Deme}. More than just an analogy, this construction introduces two important, synergistic innovations for such models. First, the shared global process $\sigma_0$ acts as a tether linking the evolutionary dynamics across all demes. This naturally induces correlated demographic events, such as simultaneous mergers across multiple groups from a single punctuated event—a "lock-step" effect that is difficult to model tractably with classical approaches. Second, the framework exhibits \textit{simultaneous structural duality}. This property ensures that for each deme $j$, the genealogical coagulation process is maintained in a precise, dual relationship with the corresponding species fragmentation process. By combining this per-deme duality with the global tethering, the framework moves beyond describing only the marginal genealogy (the structured coalescent) to provide, for the first time, a tractable, coupled representation of finer-grained ancestral processes (like a structured Ancestral Recombination Graph) that are dynamically correlated across the entire metapopulation. The multi-group PHIBP structure underlying this framework has been successfully applied to complex microbiome data \citep{hibp25}, demonstrating its computational tractability and scalability (allowing for the addition of new demes after initial observations). While that work focused on the species sampling aspects without explicitly invoking the duality, it establishes the practical viability of the four-component architecture that enables the simultaneous structural duality developed here. See Section~\ref{sec:hARG} for broader modeling capabilities.
\end{rem}

\section{Cloud duality: The h-biased L\'evy–It\^o Coupled Duality Constructor}\label{sec:dualitymachine}

We now record an abstract ``constructor'' that isolates the mechanism behind Theorem~\ref{coupledsupertheorem}: an h-biased L\'evy–It\^o/PRM architecture on Polish spaces, showing that the duality is driven by point-process regrouping (geometry layer) together with mixed-Poisson thinning/sampling (sampling layer), rather than features specific to interval partitions.

\subsection{Setup and h-biased framework}\label{subsec:setup}

The architecture of this model is best understood not as a specific solution, but as a general-purpose \emph{L\'evy-It\^o coupled duality constructor}. Built from a single Poisson random measure (PRM) in the spirit of Pitman's \emph{big Poisson} formulations, this construction gains operational clarity and modeling flexibility when recast through the \emph{h-biased sampling} framework of Pitman and Yor~\citep{PY92} and Perman, Pitman and Yor~\citep{PPY92}, which we extend here to conditionally independent general Cox processes driven by a common PRM $N_0$. This perspective provides a principled mechanism for translating interpretable features of abstract objects into the intensity structure of the underlying Poisson random measure.

Our construction is the special case where the `size' function $h$ is the identity: the points of the underlying PRMs are the jump sizes themselves ($s_{j,k,l}$ and $\lambda_l$), so a point's value \emph{is} its size for biasing. The framework's true extensibility is revealed by recognizing that $h$ can be any arbitrary positive measurable function mapping points from an abstract Polish space to a `size' value. To formalize this, let $S_{j}$ for $j = 0, \ldots, J$ be abstract Polish spaces supporting points ${(\mathbf{e}_{j,k})}_{k \ge 1}$ of independent PRMs with respective $\sigma$-finite mean measures $\mathbf{n}_{j}(d\mathbf{e}_{j})$, and let $h_{j}\colon S_{j} \to (0,\infty)$ be measurable. The random measures central to our work are specific instances of the h-biased measures $\sigma_{j}(dv) = \int_{S_{j}} h_{j}(\mathbf{e}_{j})N_{j}(d\mathbf{e}_{j}, dv)$, where $N_j$ are PRMs on $S_j \times [0,1]$ with mean measure $\mathbb{E}[N_{j}(d\mathbf{e}_{j}, dv)]=\mathbf{n}_{j}(d\mathbf{e}_{j})dv$. For example, as indicated in \citep[p. 33]{PPY92}, in the excursion case of their Theorem~3.1, $S$ is the product of an interval on the local time scale and a path space of Markov excursions, with $h(L, \mathbf{e})$ equal to the duration of excursion $\mathbf{e}$ during local time $L$.

We have used $[0,1]$ to accommodate the $\mathrm{Uniform}[0,1]$ variables $(Y_{l},U_{j,k,l})$, but this choice is not essential: one may replace $[0,1]$ by any Polish mark space endowed with a suitable non-atomic distribution, as in~\cite{hibp25}. Via standard quantile transforms, the $\mathrm{Uniform}[0,1]$ variables can be viewed simply as generic marks, so the underlying PRM structure is unchanged. In particular, the lift to continuous time in $t$ remains an \emph{external} time mechanism encoded by the product measure $dt\otimes \nu$, rather than something intrinsic to the mark space. Moreover, the within-group sampling operations of~\citep{hibp25}, together with the various standard PRM decompositions, as well as those developed by Pitman and Yor in the context of excursions, Bessel bridges, and self-similar Markov processes~\citep[e.g.][]{PitmanYor1982BesselBridges,PY2001}, can be incorporated into this framework and continue to apply verbatim under such a change of mark space.

\subsection{The cloud duality principle}\label{subsec:cloudduality}
We introduce here the notion of \emph{cloud duality}: the principle that coagulation-fragmentation dualities, including Pitman's classical construction as a special case, are fundamentally about deterministic reorganizations of point clouds on Poisson random measures, with the various object-level dualities (on partitions, trees, loops, etc.)  arising as projections under different sampling schemes. We develop concrete instantiations of this principle in Section~\ref{subsec:examples}, showing how it covers loop-soup constructions, stable tree dualities, and population genetic models within a single framework.
To make this principle concrete, we show how our four-component system instantiates cloud duality in full generality.

By identifying our jumps as $h_{j}(\mathbf{e}_{j,k,l}) = s_{j,k,l}$ and $h_{0}(\mathbf{e}_{0,l}) = \lambda_{l}$, the entire four-component system operates on arbitrary spaces, and the coupled duality of Theorem~\ref{coupledsupertheorem} extends to this general setting:
\begin{equation} \label{coupledcoarsetofinevectorQUADabstract}
\begin{pmatrix} I_j(\gamma_j, y) \\ \mathscr{A}_j(\gamma_j, y) \\ (\mathscr{F}^{(\lambda_l)}_{j,l}(\gamma_j, y))_{l\ge 1} \\ Z_j(\gamma_j, y) \end{pmatrix}
:=
\begin{pmatrix}
\sum_{l=1}^{\infty}\sum_{k=1}^{\infty}\mathscr{P}_{j,k,l}(h_{j}(\mathbf{e}_{j,k,l})\gamma_{j})\mathbb{I}_{\{U_{j,k,l}\leq y\}} \\
\sum_{l=1}^{\infty}\sum_{k=1}^{\infty}\mathbb{I}_{\{\mathscr{P}_{j,k,l}(h_{j}(\mathbf{e}_{j,k,l})\gamma_{j})>0\}}\mathbb{I}_{\{Y_{l}\leq y\}} \\
\left(\sum_{k=1}^{\infty}\mathscr{P}_{j,k,l}(h_{j}(\mathbf{e}_{j,k,l})\gamma_{j})\mathbb{I}_{\{U_{j,k,l}\leq y\}}\right)_{l\ge 1}\\
\sum_{l=1}^{\infty}\left[\sum_{k=1}^{\infty}\mathscr{P}_{j,k,l}(h_{j}(\mathbf{e}_{j,k,l})\gamma_{j})\right]\mathbb{I}_{\{Y_{l}\leq y\}}
\end{pmatrix}
\end{equation}
The compound Poisson representations and resulting joint EPPFs follow directly from Theorem~\ref{coupledsupertheorem}, expressed in terms of the pushforward of the native measures $\mathbf{n}_j$, as can be seen in the EPPF-related calculations in \citep[Sections 3, 5, 8]{James2002}.
\medskip

\subsubsection{Geometry vs.\ sampling (at the level of point clouds)}
The coupled duality has two logically distinct layers on a single underlying Poisson random measure. The first is geometric: one starts
from the latent \emph{point cloud of atoms}
\[
\Bigl(\mathbf{e}_{0,l},\;(\mathbf{e}_{j,k,l})_{j\in[J],\,k\ge1}\Bigr)_{l\ge1},
\]
where each $\mathbf{e}_{j,k,l} \in S_j$ lives in its respective native space $S_j$, forming a family of \emph{child sub-clouds} (indexed by $l$) attached to a parent cloud. Coagulation
and fragmentation are then deterministic maps on this cloud: coagulation groups atoms into
$l$-packets (``clouds-of-clouds''), while fragmentation opens each packet back into its constituent
atoms; no conditioning, biasing, or thinning is involved in these rearrangements. The second layer is
probabilistic: the mixed Poisson/thinning mechanism (implemented by the marks $Y_l,U_{j,k,l}$ and the
Poisson counts) samples from these clouds according to the size maps $h_j: S_j \to \mathbb{R}_+$. It is this sampling layer---implemented via our mixed-Poisson marking/thinning scheme and the
Poisson decompositions of Section~\ref{sec:posterior_decomp}---that yields the compound Poisson
representations and the (joint) EPPFs. The induced dualities on the native spaces $S_j$ are thus realizations of this rate-based Poisson architecture, with the various partition- and bridge-valued descriptions arising as projections of the same underlying mixed Poisson sampling scheme.

\emph{In particular, the Allocation processes $(\mathscr{A}_j)_{j\in[J]}$ belong to the sampling
layer: they record which parent index $\ell$ is actually ``lit up'' in group $j$ under the mixed
Poisson sampling, dictating for each parent cloud $\ell$ in which native spaces $S_j$ (if any) its sub-clouds are actually sampled. This role becomes especially prominent in more exotic geometric choices of the spaces $S_j$.}
\subsection{Pushforward exponents and thinning decompositions on $S_{j}$}\label{subsec:pushforward}

Here, defining the Laplace exponent
\[
\psi_{j}(\gamma_{j})=\int_{\mathbf{S}_{j}}\left( 1 - e^{-h_{j}(\mathbf{e}_{j}) \gamma_{j}} \right) \mathbf{n}_{j}(d\mathbf{e}_{j}),
\]
the resulting thinning on the basic spaces $\mathbf{S}_{j}$ manifests as follows:
\begin{itemize}
    \item \textbf{The Latent Level ($j=0$):} The measure $\mathbf{n}_0$ decomposes into components
    that are respectively inactive and active relative to the total sampling effort
    $\sum_{j=1}^{J}\psi_j(\gamma_j)$:
    \[
    \mathbf{n}_{0}(d\mathbf{e}_{0})
    =
    e^{-h_{0}(\mathbf{e}_{0}) \sum_{j=1}^{J} \psi_{j}(\gamma_{j})} \mathbf{n}_{0}(d\mathbf{e}_{0})
    +
    \left( 1 - e^{-h_{0}(\mathbf{e}_{0}) \sum_{j=1}^{J} \psi_{j}(\gamma_{j})} \right) \mathbf{n}_{0}(d\mathbf{e}_{0}).
    \]
    The $\varphi$ selected points $(\mathbf{e}_{0,
    \ell},\,\ell\in[\varphi])$ in $\mathbf{S}_{0}$ are sampled from
    the distribution proportional to the active part
    \begin{equation}
    \label{SVPdist}
    \left( 1 - e^{-h_{0}(\mathbf{e}_{0}) \sum_{j=1}^{J} \psi_{j}(\gamma_{j})} \right) \mathbf{n}_{0}(d\mathbf{e}_{0}),
    \end{equation}
    and their associated sizes are $H_{\ell} = h_0(\mathbf{e}_{0,\ell})$. In addition, there is an
    unselected latent cloud $(\mathbf{e}'_{0,l})_{l\ge 1}$ with L\'evy measure
    \[
    e^{-h_{0}(\mathbf{e}_{0}) \sum_{j=1}^{J} \psi_{j}(\gamma_{j})} \mathbf{n}_{0}(d\mathbf{e}_{0}).
    \]

    \item \textbf{The Feature Level ($j=1, \ldots, J$):} For each $j\ge 1$, the measure $\mathbf{n}_j$
    decomposes according to the local sampling effort $\gamma_j$:
    \[
    \mathbf{n}_{j}(d\mathbf{e}_{j})
    =
    e^{-h_{j}(\mathbf{e}_{j})\gamma_{j}} \mathbf{n}_{j}(d\mathbf{e}_{j})
    +
    \bigl(1 - e^{-h_{j}(\mathbf{e}_{j})\gamma_{j}}\bigr) \mathbf{n}_{j}(d\mathbf{e}_{j}).
    \]
    
    The \emph{active} feature points have normalized law
    \[
    \frac{\bigl(1 - e^{-h_{j}(\mathbf{e}_{j})\gamma_{j}}\bigr) \mathbf{n}_{j}(d\mathbf{e}_{j})}{\psi_j(\gamma_j)},
    \qquad
    \psi_j(\gamma_j)
    =
    \int_{\mathbf{S}_j} \bigl(1 - e^{-h_j(\mathbf{e}_j)\gamma_j}\bigr) \mathbf{n}_j(d\mathbf{e}_j),
    \]
    which is the same for all latent parents. Given a latent atom $\mathbf{e}_{0,l}$ with size
    $\lambda_l = h_0(\mathbf{e}_{0,l})$, the associated feature-level atoms $(\mathbf{e}_{j,k,l})_{k\ge 1}$ in $\mathbf{S}_j$ form the union
    of two independent PRMs on $\mathbf{S}_j$ with mean measures
    \[
    \lambda_l e^{-h_j(\mathbf{e}_j)\gamma_j} \mathbf{n}_j(d\mathbf{e}_j)
    \quad\text{and}\quad
    \lambda_l \bigl(1 - e^{-h_j(\mathbf{e}_j)\gamma_{j}}\bigr) \mathbf{n}_j(d\mathbf{e}_j).
    \]
    Thus $\lambda_l$ scales both the inactive and active clouds of feature atoms, and in
    particular the aggregate contribution $\sum_{k\ge1} h_j(\mathbf{e}_{j,k,l})$, whereas the
    conditional law of an individual active atom on $\mathbf{S}_j$ has a shape that does not
    depend on $\lambda_l$ in this unconditional construction. The resulting state
    dependence after Poisson sampling is described in Section~\ref{sec:posterior_decomp}
    (cf.\ Section~2.2). See Salminen, Vallois and Yor~\cite[Theorem~10]{SVY2007} for related
distributions on excursion spaces under an It\^o-type measure tilted as in~\eqref{SVPdist}.
\end{itemize}

This modularity allows for the synthesis of novel dualities for objects like trees, graphs, or spatial processes, each with tractability via the
underlying counting process, in the sense that the joint laws can always be expressed through mixed Poisson counts and the decompositions of Section~\ref{sec:posterior_decomp}. Our results here (in particular Theorem \ref{coupledsupertheorem} expressed as \eqref{coupledcoarsetofinevectorQUADabstract} and the decompositions dictated by Section~\ref{sec:posterior_decomp}) complete the first part of the analysis:
they provide a general L\'evy–It\^o/PRM construction and an abstract coupled duality formulated at the level of the underlying Poisson point clouds, from which the various projected dualities on native spaces are obtained. In direct analogy with the case of operations on spaces of bridges and
mass partitions ($\mathbf{P}_{j} = (P_{j,k})_{k\ge 1}$, $\mathbf{Q}_0 =(Q_{0,l})_{l\ge 1}$, $\mathbf{V}_{j} = (V_{j,l})_{l\ge 1}$, and
$\mathbf{Q}_{j,l} = (Q_{j,k,l}))$ etc., what remains—and is genuinely application-specific—is the second part: formulating and interpreting the
concrete duality operations on the chosen state spaces (e.g. loops, trees, or graphs) that instantiate this abstract scheme. To be clear, this is not the
duality reflected directly on the mixed Poisson space, which is immediate, but rather what that duality projects to in terms of dual operations on the native
space of points.

\subsection{Multiple representations and algorithmic flexibility}\label{subsec:multiplerepresentations}

The $h$-biased mixed Poisson framework developed in Sections \ref{sec:coupledmixedPoisson}--\ref{sec:section:PK} delivers the \emph{hard duality}: explicit compound Poisson representations, joint EPPFs, tractable conditional laws, and the full analytic machinery that makes the duality computationally executable. This calculus emerges automatically from the mixed truncated Poisson structure and the thinning decompositions of Section~\ref{sec:posterior_decomp}.

The \emph{structural} coupling, however, exists at a more primitive level. The fundamental object is a tethered Cox process: given parent atoms $(\mathbf{e}_{0,l})_{l \geq 1}$ with associated marks $(Y_l)_{l \geq 1}$, the child processes $N_j$ have conditional intensity
\[
\mathbb{E}[N_j(d\mathbf{e}_j, dv) \mid (\mathbf{e}_{0,l}, Y_l)_{l \geq 1}] = \left(\sum_{l=1}^\infty f(\mathbf{e}_{0,l})\delta_{Y_l}(dv)\right) \mathbf{n}_j(d\mathbf{e}_j),
\]
where $f$ is any non-negative function satisfying $\sum_{l=1}^\infty f(\mathbf{e}_{0,l}) < \infty$ almost surely.

This Cox structure encodes a pathwise duality that persists regardless of how one samples from $N_j$. The duality arises from how the child atoms $(\mathbf{e}_{j,k,l})_{k \geq 1}$ decompose into clouds governed by $f(\mathbf{e}_{0,l})\mathbf{n}_j(d\mathbf{e}_j)$, versus how they assemble into the aggregate configuration $(\mathbf{e}_{j,k})_{k \geq 1}$ without reference to individual parent contributions.

The PRM framework itself provides substantial structure even without the $h$-biased layer: the pathwise coupling, the decomposition into parent-indexed clouds, coagulation-fragmentation as deterministic regrouping, the lift to continuous time via $dt \otimes \nu$, and the conditional independence structure given the parent configuration. One may sample directly from the native measures $\mathbf{n}_j$ without size-biasing, employ series representations, or use specialized algorithms for particular measure classes.

What the $h$-biased mixed Poisson layer adds is the Poisson calculus---the compound Poisson representations, explicit joint EPPFs, and tractable conditional laws that make the duality analytically executable. Without this layer, the architecture remains intact, but the distributional calculations must be developed independently for each choice of native measures and sampling scheme. If one simply wants a coupled coarse-fine system based on Beta coalescents, the Cox structure tells you how to build it. But deriving the joint partition probabilities or predictive distributions then becomes calculus one must carry out oneself---work that is perfectly legitimate, but work nonetheless.

\begin{rem}
This separation clarifies the scope of Theorem 2.1. The theorem provides both the structural coupling (robust to representation) and the computational machinery (requiring the $h$-biased framework). For applications needing only pathwise simulation---such as sampling $\mathrm{Beta}(2-\delta_j, \delta_j)$ coalescents directly via native measures across $0 < \delta_j < 2$---the Cox structure suffices. For applications requiring explicit partition probabilities or predictive distributions, the mixed Poisson framework substantially simplifies relevant PRM calculus.
\end{rem}

\subsection{Examples and applications}\label{subsec:examples}

We record two   geometric instantiations of the constructor—one in loop-soup/CLE settings and one in stable-tree settings—followed by a third example (ARG/TMRCA) meant primarily as a \emph{modeling} template for domain experts rather than a canonical geometric identification.

\subsubsection{Loop-soup/CLE instantiation}
For instance, in relation to questions of constructive theory, one could take
$(S_{j}, \mathcal{S}_{j}, \mathbf{n}_{j})$ as the Polish space of loops with SLE$_\kappa$
loop measure $\mu_\kappa := \mathbf{n}_{j}$—for example, a conformal loop-soup intensity
corresponding to CLE$_\kappa$ as in \citep{SW12}—or, in the cable-graph setting, as the
space of Brownian loops on a cable graph endowed with the critical loop-soup intensity of
\citep[Section~5.4]{werner2025switchingidentitycablegraphloop}, where each
$\mathbf{e}_{j,k,l}$ represents an individual loop. In this interpretation, the latent Poisson
random measure may be viewed as producing a \emph{cloud of loop-atoms} organized into parent-indexed
\emph{loop packets} $\bigl(\mathbf{e}_{0,l},(\mathbf{e}_{j,k,l})_{j\in[J],\,k\ge1}\bigr)_{l\ge1}$.
Coagulation/fragmentation on the native loop spaces then has a direct geometric reading as
grouping loops into packets (coagulation) and opening packets into their constituent loops
(fragmentation), i.e.\ deterministic rearrangements of the same underlying cloud.
In this loop-soup specialization, the \emph{geometric layer} is simply the deterministic regrouping of the same underlying loop configuration into packets versus individual loops, while the \emph{sampling layer} selects packets/loops according to $h_j$ (e.g.\ length, diameter, occupation time), producing the corresponding compound Poisson representations and EPPFs.

Our general duality result applies directly to such choices: choosing bespoke functions $h_j$
such as conformal radius, loop length, loop diameter, or total occupation time along a
distinguished edge yields novel coupled four-component dual systems, which can then be raised to
continuous-time processes via \eqref{homfragdynamics}. The additional mixed-Poisson
marking/thinning layer of the constructor then selects from these loop clouds according to the
observables $h_j$; it is this sampling layer—together with the Poisson decompositions of
Section~\ref{sec:posterior_decomp}—that produces the compound Poisson representations and the (joint) EPPFs. In particular, the Poissonian framework already supplies the structural duality in these loop-soup settings. Moreover, these sampling schemes admit Poisson decompositions of the same general type as in \citep{PPY92, PY92}.
To incorporate the switching property developed in \cite{werner2025switchingidentitycablegraphloop}---which operates in the same spirit as the Bessel bridge decompositions of \cite{PitmanYor1982BesselBridges}---one may employ a signed decomposition of the parent measure:
\[
\mathbf{n}_{0}(d\mathbf{e}) = \mathbf{n}_{0,+}(d\mathbf{e}) + \mathbf{n}_{0,-}(d\mathbf{e}),
\]
where $\mathbf{n}_{0,+}$ and $\mathbf{n}_{0,-}$ correspond respectively to configurations with an even and odd number of excursions connecting designated points. This creates two coupled $4 \times J$-component dualities over restricted configuration spaces that combine additively. Sampling from this decomposition proceeds either by parity-based colouring (in Werner's terminology, the uniform random even subgraph selection) or by the hierarchical schemes developed in \cite{hibp25}.
The choice of child objects $\mathbf{e}_{j,k,l}$ in this setting is flexible and points to natural pairings. For instance, one could take $S_1$ to be $\alpha$-stable L\'evy trees with $h_1(\text{tree}) = $ total mass, modeling genealogical structures within spatial habitats defined by the parent loops; or $S_1$ could be SLE$\kappa$ curves (for $\kappa \neq 4$) with $h_1(\text{curve}) = $ Loewner capacity, capturing conformal interfaces nested within loop clusters, such as level lines of the Gaussian free field at different heights. In each case, the size function $h_1$ determines the intensity of attachment to parent loops of size $h_0$, and Werner's switching property extends automatically: conditioning on connectivity through the child objects corresponds to overlaying an odd number of connecting structures, with the parity being conditionally independent of the occupation field $\Lambda$.

\subsubsection{Stable-tree instantiation}\label{subsec:stable_tree}
A second geometric specialization—which, like \cite{Haas2},  realizes coarse-fine duality directly at the level of tree structures rather than merely mass partitions—is obtained by taking $S_0$ and $S_1$ to be Polish spaces of compact rooted measured $\mathbb{R}$--trees (or excursions coding such trees), and choosing measurable size maps $h_0,h_1$ (e.g.\ $h(\mathcal T)=\mu(\mathcal T)$) together with $\sigma$--finite measures $\mathbf{n}_0,\mathbf{n}_1$ so that the pushforwards $\mathbf{n}_j\circ h_j^{-1}$ match the relevant stable/PD L\'evy intensities (hence, after projection to masses, reproduce the \emph{laws} of the coarse and fine Poisson--Dirichlet spinal mass partitions in \cite[Corollary~10]{Haas2}). In this setting, the coupled PRM constructor yields a pathwise coarse--fine coupling at the level of tree-marked point clouds. At the cloud level, the coarse--fine duality is realized by deterministic transformations of the same parent-indexed point cloud: coagulation groups child tree-marks into $\ell$-packets, while fragmentation opens each packet back into its constituent marks. The resulting coarse and fine tree-valued objects are then the images of these transformed clouds under a measurable spine-gluing map (geometry layer), with the Poisson--Dirichlet structure entering through the mixed-Poisson sampling layer via $h_0,h_1$. This contrasts with the Haas--Pitman--Winkel construction, which starts from a single stable CRT and \emph{decomposes} it along a distinguished spine; here we instead start from a parent-indexed PRM \emph{cloud of tree-marks} and \emph{assemble} them via the gluing map. After projection to masses, the induced coarse/fine spinal partitions agree in law with \cite[Corollary~10]{Haas2}; we emphasize that our goal is \emph{not} to identify the resulting tree-valued objects with the
particular HPW stable-CRT realization, unless one imposes additional calibration of the mark laws and the spine-gluing map to enforce such an identification. By way of analogy, the mere appearance of $\mathrm{PD}(1/2,0)$ laws for ranked excursion lengths does not characterize Brownian motion
among strong Markov processes: the same Poisson--Dirichlet length statistics can arise from different excursion-shape (It\^o measure) specifications. Likewise, matching the induced Poisson--Dirichlet spinal mass laws fixes only the ``mass layer'' of our construction; the tree-valued geometry remains dependent on the choice of mark laws and the spine-gluing map. Regardless of the specific choice of mark laws and spine-gluing map, we obtain a four-component simultaneous structural duality coupled system on trees that can be lifted to continuous time as in Section~\ref{sec:DynamicCouplin}. 

\begin{rem}[Geometric interpretation]
In general, a refined geometric understanding—namely, expressing the coagulation/fragmentation operations directly in terms of transformations of the underlying parent-indexed point clouds on the native spaces $S_j$, and identifying observables $h_j$ with natural geometric or analytic meaning in specific models—requires additional, model-specific work that lies beyond the scope of the present article.
\end{rem}

\subsubsection{Applied instantiation (ARG/TMRCA)}\label{sec:hARG}
Beyond such geometric settings, and more directly relevant to continuous-time population genetics, consider $(S_j, \mathcal{S}_j, \mathbf{n}_j)$ as a space
of labeled genealogical trees or ancestral recombination graphs (ARGs), where $\mathbf{e}_{j,k,l}$ represents a specific genealogical structure. Rather than
specifying complex L\'evy densities, the modeler can work with interpretable functions $h_j$: for instance, $h_j(\mathbf{e}_{j,k,l})$ could be the total
branch length (relevant for mutation accumulation), the time to most recent common ancestor (TMRCA), the number of recombination events, or a measure of
tree imbalance. This function-based specification translates biological assumptions—such as "lineages with longer branches contribute more
variation"—directly into the model's sampling mechanism, bypassing the indirect task of reverse-engineering the corresponding L\'evy measure. The framework then
automatically generates the coupled four-component system with tractable joint EPPFs, providing both the genealogical structure and its dual fragmentation
dynamics. 

\subsection{The skeletal principle: inducing duality on derived objects}\label{subsec:skeletal}

An instructive feature of the stable-tree instantiation in Section~\ref{subsec:stable_tree} is that the trees themselves are not PRM atoms—they are \emph{assembled} from atom-indexed marks via a spine-gluing map, in the spirit of Haas, Pitman, and Winkel's spinal decompositions \cite{Haas2}. A similar principle underlies the Qian-Werner coupling of Brownian loop soups \cite{QianWerner2019}. These constructions illustrate a general pattern that extends well beyond objects directly built from PRM atoms: if you know how to build one component from a PRM, its coarse and fine companions—obeying the cloud-based simultaneous structural duality in lockstep—are already facilitated by the duality in the PRM skeleton, with the geometric design features remaining degrees of freedom for the builder.

Many stochastic objects on abstract spaces $S_j$—while not themselves point clouds—are either \emph{purely functionally determined} by an underlying PRM, or are determined by a PRM \emph{together with auxiliary structure} (such as local times, gluing maps, or spatial scaffolding) that governs how atoms are assembled. For such objects, our cloud duality framework provides a \emph{rigid but highly flexible skeletal construction}: the duality architecture is inherited automatically from the PRM level, with the induced duality on the derived space arising as a deterministic projection.

More precisely, suppose $\mathcal{X}_j$ is constructed from PRM atoms $(\mathbf{e}_{j,k})_{k \geq 1}$ on $S_j$—either as a direct measurable functional $\mathcal{X}_j = \Phi_j\bigl((\mathbf{e}_{j,k})_{k \geq 1}\bigr)$, or via such a functional composed with auxiliary processes that provide temporal or spatial scaffolding. Then the four-component coupled system of Theorem~\ref{coupledsupertheorem} immediately induces a coupled system $(\mathcal{X}^{\mathrm{fine}}_j, \mathcal{X}^{\mathrm{coarse}}_j)$ on the image space, with:
\begin{itemize}
    \item Coarse and fine companions evolving in lockstep under the cloud-based simultaneous structural duality;
    \item Joint laws computable via the compound Poisson calculus;
    \item Continuous-time coupled dynamics inherited from the lift $dt \otimes \nu$.
\end{itemize}

\begin{exam}[Coupled Bessel processes with PD-CRP counts]\label{exam:bessel_skeletal}
The squared Bessel process of dimension $\delta > 0$ admits a classical representation via its excursion point process—a PRM on path space with intensity governed by the It\^o excursion measure—together with the local time process at zero, which provides the temporal scaffolding along which excursions are concatenated. For dimensions $0 < \delta_2 < \delta_1$, coupling two Bessel processes $R^{(\delta_1)}$ and $R^{(\delta_2)}$ via shared excursion PRMs and compatible local times (with appropriate thinning of excursions) yields nested partition structures on excursion counts. The induced partitions follow Chinese Restaurant Process dynamics with $\mathrm{PD}(\delta_i/2, 0)$ asymptotics, and the joint EPPFs are computable from the stable-case formulas of Section~\ref{sec:tagged_fragmentation_example}. This illustrates how the duality framework applies to excursion-based constructions; the classical work of Pitman and Yor on Bessel process decompositions \citep{PitmanYor1982BesselBridges, PY97} provides natural reference points for specialists pursuing rigorous geometric identification.
\end{exam}

\begin{exam}[Coupled Brownian loop soups]\label{exam:loop_soup_skeletal}
The Brownian loop soup exemplifies a structurally simpler situation: the derived object \emph{is} the PRM itself—requiring no gluing or scaffolding—though the atoms live on an exotic space of continuous paths. The loop soup on a domain $D \subset \mathbb{R}^d$, introduced by Lawler and Werner~\cite{LawlerWerner2004}, is a Poisson point process on the space of unrooted loops with intensity $\alpha \cdot \mu^{\mathrm{loop}}_D$, where $\mu^{\mathrm{loop}}_D$ is the (infinite, $\sigma$-finite) Brownian loop measure. Though each atom is an entire continuous path—not a point in the classical sense—the soup itself is a PRM on the Polish space of loops, and the skeletal principle applies directly.

For intensities $\alpha_2 < \alpha_1$, our framework yields coupled soups $(\mathcal{L}^{[\alpha_1]}, \mathcal{L}^{[\alpha_2]})$ satisfying $\mathcal{L}^{[\alpha_2]} \subset \mathcal{L}^{[\alpha_1]}$ almost surely, with the ``missing'' loops $\mathcal{L}^{[\alpha_1]} \setminus \mathcal{L}^{[\alpha_2]}$ forming an independent soup of intensity $(\alpha_1 - \alpha_2)$. The cluster structure of loops—which loops intersect, which together bound a given region, which contribute to a given local time—then inherits a nested partition structure from the cloud level, with explicit joint laws computable via compound Poisson calculus. This illustrates how the duality framework applies to loop-soup constructions, providing a ``mashup'' perspective that complements the elegant coupling constructed by Qian and Werner \citep{QianWerner2019}.
\end{exam}

The skeletal principle thus effects a clean separation of concerns. The \emph{geometric wiring}—how objects on $S_j$ are constructed from PRM atoms, and what auxiliary structure (if any) governs their assembly—is application-specific and may require substantial domain expertise: excursion theory and local times for Bessel processes, the loop measure for Brownian soups, spine-gluing maps for tree-valued processes. But once this wiring is specified, the \emph{duality architecture} is inherited automatically from the cloud level. The cloud duality framework provides the skeleton; the application supplies the flesh.

\begin{rem}[Scope of the cloud duality principle]\label{rem:scope_cloud}
Our results supply the duality architecture and Poisson calculus, with direct uplift to continuous-time dynamics; the geometric wiring of specific models, and probabilistic elements both within and beyond our framework, remain degrees of freedom for the domain expert. The scope of this generative principle is thus left to the imagination and utility of the user.
\end{rem}

%
%

\begin{funding}
This work was supported in part by RGC-GRF grants 16301521 of the HKSAR.
\end{funding}






\bibliographystyle{imsart-number} 
\bibliography{bibliographyfullupdated}

\begin{thebibliography}{55}

\bibitem{Aldous}
\begin{bincollection}[author]
\bauthor{\bsnm{Aldous},~\bfnm{D.~J.}\binits{D.~J.}}
(\byear{1985}).
\btitle{{Exchangeability and related topics}}.
In \bbooktitle{\'Ecole d'\'Et\'e de Probabilit\'es de Saint-Flour XIII -- 1983},
(\beditor{\bfnm{P.~L.}\binits{P.~L.}~\bsnm{Hennequin}}, ed.).
\bseries{Lecture Notes in Mathematics}
\bvolume{1117}
\bpages{1--198}.
\bpublisher{Springer}, \baddress{Berlin}.
\end{bincollection}
\endbibitem

\bibitem{BPY}
\begin{bincollection}[author]
\bauthor{\bsnm{Barlow},~\bfnm{M.}\binits{M.}}, \bauthor{\bsnm{Pitman},~\bfnm{J.}\binits{J.}} \AND \bauthor{\bsnm{Yor},~\bfnm{M.}\binits{M.}}
(\byear{1989}).
\btitle{{Une extension multidimensionnelle de la loi de l'arc sinus}}.
In \bbooktitle{S\'eminaire de Probabilit\'es XXIII},
(\beditor{\bfnm{J.}\binits{J.}~\bsnm{Azema}}, \beditor{\bfnm{P.~A.}\binits{P.~A.}~\bsnm{Meyer}} \AND \beditor{\bfnm{M.}\binits{M.}~\bsnm{Yor}}, eds.).
\bseries{Lecture Notes in Mathematics}
\bvolume{1372}
\bpages{294--314}.
\bpublisher{Springer}, \baddress{Berlin}.
\end{bincollection}
\endbibitem

\bibitem{BerestyckiEtheridgeVeber2013}
\begin{bincollection}[author]
\bauthor{\bsnm{Berestycki},~\bfnm{N.}\binits{N.}}, \bauthor{\bsnm{Etheridge},~\bfnm{A.~M.}\binits{A.~M.}} \AND \bauthor{\bsnm{Veber},~\bfnm{A.}\binits{A.}}
(\byear{2013}).
\btitle{Genealogies in spatially structured populations}.
In \bbooktitle{Probability in Complex Physical Systems: In Homage to Erwin Bolthausen and J{\"u}rgen G{\"a}rtner},
(\beditor{\bfnm{Jean-Dominique}\binits{J.-D.}~\bsnm{Deuschel}}, \beditor{\bfnm{B{\'a}lint}\binits{B.}~\bsnm{T{\'o}th}}, \beditor{\bfnm{Augusto}\binits{A.}~\bsnm{Teixeira}} \AND \beditor{\bfnm{Wendelin}\binits{W.}~\bsnm{Werner}}, eds.).
\bseries{Springer Proceedings in Mathematics \& Statistics}
\bvolume{11}
\bpages{111--155}.
\bpublisher{Springer}.
\end{bincollection}
\endbibitem

\bibitem{Bertoin1999}
\begin{bbook}[author]
\bauthor{\bsnm{Bertoin},~\bfnm{J.}\binits{J.}}
(\byear{1999}).
\btitle{Subordinators: examples and applications}.
\bseries{Lectures on probability theory and statistics}
\bvolume{1717}.
\bpublisher{Springer}.
\end{bbook}
\endbibitem

\bibitem{BerFrag}
\begin{bbook}[author]
\bauthor{\bsnm{Bertoin},~\bfnm{J.}\binits{J.}}
(\byear{2006}).
\btitle{{Random Fragmentation and Coagulation Processes}}.
\bpublisher{Cambridge University Press}.
\end{bbook}
\endbibitem

\bibitem{BerLegall00}
\begin{barticle}[author]
\bauthor{\bsnm{Bertoin},~\bfnm{J.}\binits{J.}} \AND \bauthor{\bsnm{Le~Gall},~\bfnm{J.~F.}\binits{J.~F.}}
(\byear{2000}).
\btitle{{The Bolthausen-Sznitman coalescent and the genealogy of continuous-state branching processes}}.
\bjournal{Probability Theory and Related Fields}
\bvolume{117}
\bpages{249--266}.
\end{barticle}
\endbibitem

\bibitem{CraneEwens}
\begin{barticle}[author]
\bauthor{\bsnm{Crane},~\bfnm{H.}\binits{H.}}
(\byear{2016}).
\btitle{{The Ubiquitous Ewens Sampling Formula}}.
\bjournal{Statistical Science}
\bvolume{31}
\bpages{1 -- 19}.
\bdoi{10.1214/15-STS529}
\end{barticle}
\endbibitem

\bibitem{PraEtheridge2025Deme}
\begin{barticle}[author]
\bauthor{\bsnm{Dai~Pra},~\bfnm{M.}\binits{M.}}, \bauthor{\bsnm{Etheridge},~\bfnm{A.~M.}\binits{A.~M.}}, \bauthor{\bsnm{Koskela},~\bfnm{J.}\binits{J.}} \AND \bauthor{\bsnm{Wilke-Berenguer},~\bfnm{M.}\binits{M.}}
(\byear{2025}).
\btitle{Multi-type $\Xi$-coalescents from structured population models with bottlenecks}.
\bjournal{arXiv preprint arXiv:2504.11875}.
\end{barticle}
\endbibitem

\bibitem{DonnellyKurtz1999}
\begin{barticle}[author]
\bauthor{\bsnm{Donnelly},~\bfnm{Peter}\binits{P.}} \AND \bauthor{\bsnm{Kurtz},~\bfnm{Thomas~G.}\binits{T.~G.}}
(\byear{1999}).
\btitle{Particle representations for measure-valued population models}.
\bjournal{The Annals of Probability}
\bvolume{27}
\bpages{166--205}.
\end{barticle}
\endbibitem

\bibitem{Ferg1973}
\begin{barticle}[author]
\bauthor{\bsnm{Ferguson},~\bfnm{T.~S.}\binits{T.~S.}}
(\byear{1973}).
\btitle{{A Bayesian Analysis of Some Nonparametric Problems}}.
\bjournal{Ann. Statist.}
\bvolume{1}
\bpages{209-230}.
\end{barticle}
\endbibitem

\bibitem{Gnedin06}
\begin{barticle}[author]
\bauthor{\bsnm{Gnedin},~\bfnm{A.}\binits{A.}} \AND \bauthor{\bsnm{Pitman},~\bfnm{J.}\binits{J.}}
(\byear{2006}).
\btitle{{Exchangeable Gibbs partitions and Stirling triangles}}.
\bjournal{Journal of Mathematical Sciences}
\bvolume{138}
\bpages{5674--5685}.
\end{barticle}
\endbibitem

\bibitem{griffiths1997ancestral}
\begin{bincollection}[author]
\bauthor{\bsnm{Griffiths},~\bfnm{R.~C.}\binits{R.~C.}} \AND \bauthor{\bsnm{Marjoram},~\bfnm{P.}\binits{P.}}
(\byear{1997}).
\btitle{An ancestral recombination graph}.
In \bbooktitle{Progress in Population Genetics and Human Evolution}
(\beditor{\bfnm{Peter}\binits{P.}~\bsnm{Donnelly}} \AND \beditor{\bfnm{Simon}\binits{S.}~\bsnm{Tavar{\'e}}}, eds.)
\bpages{257--270}.
\bpublisher{Springer}.
\end{bincollection}
\endbibitem

\bibitem{Haas2}
\begin{barticle}[author]
\bauthor{\bsnm{Haas},~\bfnm{B.}\binits{B.}}, \bauthor{\bsnm{Pitman},~\bfnm{J.}\binits{J.}} \AND \bauthor{\bsnm{Winkel},~\bfnm{M.}\binits{M.}}
(\byear{2009}).
\btitle{{Spinal partitions and invariance under re-rooting of continuum random trees}}.
\bjournal{Annals of Probability}
\bvolume{37}
\bpages{1381--1411}.
\end{barticle}
\endbibitem

\bibitem{hein2005gene}
\begin{bbook}[author]
\bauthor{\bsnm{Hein},~\bfnm{Jotun}\binits{J.}}, \bauthor{\bsnm{Schierup},~\bfnm{Mikkel~H.}\binits{M.~H.}} \AND \bauthor{\bsnm{Wiuf},~\bfnm{Carsten}\binits{C.}}
(\byear{2005}).
\btitle{Gene Genealogies, Variation and Evolution: A Primer in Coalescent Theory}.
\bpublisher{Oxford University Press}.
\end{bbook}
\endbibitem

\bibitem{HJL}
\begin{bmisc}[author]
\bauthor{\bsnm{Ho},~\bfnm{M.~W.}\binits{M.~W.}}, \bauthor{\bsnm{James},~\bfnm{L.~F.}\binits{L.~F.}} \AND \bauthor{\bsnm{Lau},~\bfnm{J.~W.}\binits{J.~W.}}
(\byear{2007}).
\btitle{{Gibbs partitions (EPPF's) derived from a stable subordinator are Fox $H$ and Meijer $G$ transforms}}.
\end{bmisc}
\endbibitem

\bibitem{HJL2}
\begin{barticle}[author]
\bauthor{\bsnm{Ho},~\bfnm{M.~W.}\binits{M.~W.}}, \bauthor{\bsnm{James},~\bfnm{L.~F.}\binits{L.~F.}} \AND \bauthor{\bsnm{Lau},~\bfnm{J.~W.}\binits{J.~W.}}
(\byear{2021}).
\btitle{{Gibbs partitions, Riemann-Liouville fractional operators, Mittag-Leffler functions, and fragmentations derived from stable subordinators}}.
\bjournal{J. Appl. Probab.}
\bvolume{58}
\bpages{314--334}.
\end{barticle}
\endbibitem

\bibitem{HoJamesLau2025}
\begin{barticle}[author]
\bauthor{\bsnm{Ho},~\bfnm{M.~W.}\binits{M.~W.}}, \bauthor{\bsnm{James},~\bfnm{L.~F.}\binits{L.~F.}} \AND \bauthor{\bsnm{Lau},~\bfnm{J.~W.}\binits{J.~W.}}
(\byear{2025}).
\btitle{Inverse Clustering of Gibbs Partitions via Independent Fragmentation and Dual Dependent Coagulation Operators}.
\bjournal{Journal of Applied Probability}
\bpages{1--18}.
\bnote{available online}.
\bdoi{10.1017/jpr.2025.28}
\end{barticle}
\endbibitem

\bibitem{IJ2001}
\begin{barticle}[author]
\bauthor{\bsnm{Ishwaran},~\bfnm{H.}\binits{H.}} \AND \bauthor{\bsnm{James},~\bfnm{L.~F.}\binits{L.~F.}}
(\byear{2001}).
\btitle{{Gibbs sampling methods for stick-breaking priors.}}
\bjournal{J. Amer. Stat. Assoc.}
\bvolume{96}
\bpages{161-173}.
\end{barticle}
\endbibitem

\bibitem{IJ2003}
\begin{barticle}[author]
\bauthor{\bsnm{Ishwaran},~\bfnm{H.}\binits{H.}} \AND \bauthor{\bsnm{James},~\bfnm{L.~F.}\binits{L.~F.}}
(\byear{2003}).
\btitle{{Generalized weighted Chinese restaurant processes for species sampling mixture models.}}
\bjournal{Statist. Sinica}
\bvolume{13}
\bpages{1211-1235}.
\end{barticle}
\endbibitem

\bibitem{James2002}
\begin{barticle}[author]
\bauthor{\bsnm{James},~\bfnm{L.~F.}\binits{L.~F.}}
(\byear{2002}).
\btitle{Poisson Process Partition Calculus with applications to Exchangeable models and Bayesian Nonparametrics}.
\bjournal{arXiv:math/0205093}.
\end{barticle}
\endbibitem

\bibitem{James2017}
\begin{barticle}[author]
\bauthor{\bsnm{James},~\bfnm{L.~F.}\binits{L.~F.}}
(\byear{2017}).
\btitle{{Bayesian Poisson Calculus for Latent Feature Modeling via Generalized Indian Buffet Process Priors}}.
\bjournal{Ann. Stat.}
\bvolume{45}
\bpages{2016-2045}.
\end{barticle}
\endbibitem

\bibitem{JamesStick}
\begin{barticle}[author]
\bauthor{\bsnm{James},~\bfnm{L.~F}\binits{L.~F.}}
(\byear{2019}).
\btitle{{Stick-breaking Pitman-Yor processes given the species sampling size}}.
\bjournal{arXiv preprint arXiv:1908.07186}.
\end{barticle}
\endbibitem

\bibitem{hibp25}
\begin{bmisc}[author]
\bauthor{\bsnm{James},~\bfnm{L.~F.}\binits{L.~F.}}, \bauthor{\bsnm{Lee},~\bfnm{J.}\binits{J.}} \AND \bauthor{\bsnm{Pandey},~\bfnm{A.}\binits{A.}}
(\byear{2025}).
\btitle{{Poisson Hierarchical Indian Buffet Processes-With Indications for Microbiome Species Sampling Models}}.
\bdoi{10.48550/arXiv.2502.01919}
\end{bmisc}
\endbibitem

\bibitem{JLP2}
\begin{barticle}[author]
\bauthor{\bsnm{James},~\bfnm{L.~F.}\binits{L.~F.}}, \bauthor{\bsnm{Lijoi},~\bfnm{A.}\binits{A.}} \AND \bauthor{\bsnm{Pr\"{u}nster},~\bfnm{I}\binits{I.}}
(\byear{2009}).
\btitle{{Posterior analysis for normalized random measures with independent increments.}}
\bjournal{Scand. J. Stat.}
\bvolume{36}
\bpages{76-97}.
\end{barticle}
\endbibitem

\bibitem{JamesYor}
\begin{bmisc}[author]
\bauthor{\bsnm{James},~\bfnm{L.~F.}\binits{L.~F.}} \AND \bauthor{\bsnm{Yor},~\bfnm{M.}\binits{M.}}
(\byear{2007}).
\btitle{Tilted stable subordinators, Gamma time changes and Occupation Time of rays by Bessel Spiders}.
\bhowpublished{arXiv:math/0701049}.
\end{bmisc}
\endbibitem

\bibitem{Johnston2023}
\begin{barticle}[author]
\bauthor{\bsnm{Johnston},~\bfnm{S.~G.~G.}\binits{S.~G.~G.}}, \bauthor{\bsnm{Kyprianou},~\bfnm{A.}\binits{A.}} \AND \bauthor{\bsnm{Rogers},~\bfnm{T.}\binits{T.}}
(\byear{2023}).
\btitle{Multitype $\Lambda$-coalescents}.
\bjournal{The Annals of Applied Probability}
\bvolume{33}
\bpages{4210--4237}.
\bdoi{10.1214/22-AAP1891}
\end{barticle}
\endbibitem

\bibitem{Kawecki2004}
\begin{barticle}[author]
\bauthor{\bsnm{Kawecki},~\bfnm{T.~J.}\binits{T.~J.}} \AND \bauthor{\bsnm{Ebert},~\bfnm{D.}\binits{D.}}
(\byear{2004}).
\btitle{Conceptual issues in local adaptation}.
\bjournal{Ecology Letters}
\bvolume{7}
\bpages{1225--1241}.
\bdoi{10.1111/j.1461-0248.2004.00684.x}
\end{barticle}
\endbibitem

\bibitem{kingman1975}
\begin{barticle}[author]
\bauthor{\bsnm{Kingman},~\bfnm{J.~F.~C.}\binits{J.~F.~C.}}
(\byear{1975}).
\btitle{Random discrete distributions}.
\bjournal{Journal of the Royal Statistical Society: Series B (Methodological)}
\bvolume{37}
\bpages{1--15}.
\end{barticle}
\endbibitem

\bibitem{Kolchin}
\begin{bbook}[author]
\bauthor{\bsnm{Kolchin},~\bfnm{V.~F.}\binits{V.~F.}}
(\byear{1986}).
\btitle{{Random Mappings}}.
\bseries{Translation Series in Mathematics and Engineering}.
\bpublisher{{Optimization Software Inc., Publications Division}}, \baddress{New York}.
\end{bbook}
\endbibitem

\bibitem{LawlerWerner2004}
\begin{barticle}[author]
\bauthor{\bsnm{Lawler},~\bfnm{Gregory~F.}\binits{G.~F.}} \AND \bauthor{\bsnm{Werner},~\bfnm{Wendelin}\binits{W.}}
(\byear{2004}).
\btitle{The {B}rownian loop soup}.
\bjournal{Probability Theory and Related Fields}
\bvolume{128}
\bpages{565--588}.
\end{barticle}
\endbibitem

\bibitem{Lukacs55CharGamma}
\begin{barticle}[author]
\bauthor{\bsnm{Lukacs},~\bfnm{E.}\binits{E.}}
(\byear{1955}).
\btitle{A Characterization of the Gamma Distribution}.
\bjournal{The Annals of Mathematical Statistics}
\bvolume{26}
\bpages{319--324}.
\bdoi{10.1214/aoms/1177728549}
\end{barticle}
\endbibitem

\bibitem{Mayr1954}
\begin{bincollection}[author]
\bauthor{\bsnm{Mayr},~\bfnm{E.}\binits{E.}}
(\byear{1954}).
\btitle{Change of genetic environment and evolution}.
In \bbooktitle{Evolution as a Process}
(\beditor{\bfnm{Julian}\binits{J.}~\bsnm{Huxley}}, \beditor{\bfnm{A.~C.}\binits{A.~C.}~\bsnm{Hardy}} \AND \beditor{\bfnm{E.~B.}\binits{E.~B.}~\bsnm{Ford}}, eds.)
\bpages{157--180}.
\bpublisher{Allen and Unwin}, \baddress{London}.
\end{bincollection}
\endbibitem

\bibitem{mohle2024multi}
\begin{barticle}[author]
\bauthor{\bsnm{Möhle},~\bfnm{M.}\binits{M.}}
(\byear{2024}).
\btitle{On multi-type Cannings models and multi-type exchangeable coalescents}.
\bjournal{Theoretical Population Biology}
\bvolume{156}
\bpages{103--116}.
\bdoi{10.1016/j.tpb.2024.02.005}
\end{barticle}
\endbibitem

\bibitem{PPY92}
\begin{barticle}[author]
\bauthor{\bsnm{Perman},~\bfnm{M.}\binits{M.}}, \bauthor{\bsnm{Pitman},~\bfnm{J.}\binits{J.}} \AND \bauthor{\bsnm{Yor},~\bfnm{M.}\binits{M.}}
(\byear{1992}).
\btitle{{Size-biased sampling of Poisson point processes and excursions}}.
\bjournal{Probability Theory and Related Fields}
\bvolume{92}
\bpages{21--39}.
\end{barticle}
\endbibitem

\bibitem{Pit96}
\begin{barticle}[author]
\bauthor{\bsnm{Pitman},~\bfnm{J.}\binits{J.}}
(\byear{1996}).
\btitle{{Some developments of the Blackwell-MacQueen urn scheme}}.
\bjournal{Statistics, probability and game theory}
\bvolume{IMS Lecture Notes Monogr. Ser., 30, Inst. Math. Statist., Hayward, CA}
\bpages{245-267}.
\end{barticle}
\endbibitem

\bibitem{Pit97}
\begin{barticle}[author]
\bauthor{\bsnm{Pitman},~\bfnm{J.}\binits{J.}}
(\byear{1997}).
\btitle{{Partition structures derived from Brownian motion and stable subordinators}}.
\bjournal{Bernoulli}
\bvolume{3}
\bpages{79-96}.
\end{barticle}
\endbibitem

\bibitem{Pit99Coag}
\begin{barticle}[author]
\bauthor{\bsnm{Pitman},~\bfnm{J.}\binits{J.}}
(\byear{1999}).
\btitle{{Coalescents with multiple collisions}}.
\bjournal{Ann. Probab.}
\bvolume{27}
\bpages{1870--1902}.
\end{barticle}
\endbibitem

\bibitem{Pit02}
\begin{bincollection}[author]
\bauthor{\bsnm{Pitman},~\bfnm{J.}\binits{J.}}
(\byear{2003}).
\btitle{{Poisson-Kingman partitions}}.
In \bbooktitle{Institute of Mathematical Statistics}
(\beditor{\bfnm{R.}\binits{R.}~\bsnm{Goldstein}}, ed.)
\bpages{1-34}.
\bpublisher{Hayward}, \baddress{California}.
\end{bincollection}
\endbibitem

\bibitem{Pit06}
\begin{bbook}[author]
\bauthor{\bsnm{Pitman},~\bfnm{J.}\binits{J.}}
(\byear{2006}).
\btitle{{Combinatorial stochastic processes}}.
\bseries{Lectures from the 32nd Summer School on Probability Theory held in Saint-Flour, July 7--24, 2002. With a foreword by Jean Picard. Lecture Notes in Mathematics, 1875}.
\bpublisher{Springer-Verlag}, \baddress{Berlin}.
\end{bbook}
\endbibitem

\bibitem{PitmanPoissonMix}
\begin{bunpublished}[author]
\bauthor{\bsnm{Pitman},~\bfnm{J.}\binits{J.}}
(\byear{2017}).
\btitle{{Mixed Poisson and negative binomial models for clustering and species sampling}}.
\bnote{Unpublished Manuscript}.
\end{bunpublished}
\endbibitem

\bibitem{PitmanYor1982BesselBridges}
\begin{barticle}[author]
\bauthor{\bsnm{Pitman},~\bfnm{Jim}\binits{J.}} \AND \bauthor{\bsnm{Yor},~\bfnm{Marc}\binits{M.}}
(\byear{1982}).
\btitle{A decomposition of {Bessel} Bridges}.
\bjournal{Zeitschrift f{\"u}r Wahrscheinlichkeitstheorie und Verwandte Gebiete}
\bvolume{59}
\bpages{425--457}.
\bdoi{10.1007/BF00532802}
\end{barticle}
\endbibitem

\bibitem{PY92}
\begin{barticle}[author]
\bauthor{\bsnm{Pitman},~\bfnm{J.}\binits{J.}} \AND \bauthor{\bsnm{Yor},~\bfnm{M.}\binits{M.}}
(\byear{1992}).
\btitle{{Arcsine laws and interval partitions derived from a stable subordinator}}.
\bjournal{Proceedings of the London Mathematical Society}
\bvolume{s3-65}
\bpages{326--356}.
\end{barticle}
\endbibitem

\bibitem{PY97}
\begin{barticle}[author]
\bauthor{\bsnm{Pitman},~\bfnm{J.}\binits{J.}} \AND \bauthor{\bsnm{Yor},~\bfnm{M.}\binits{M.}}
(\byear{1997}).
\btitle{{The two-parameter Poisson-Dirichlet distribution derived from a stable subordinator}}.
\bjournal{Ann. Probab.}
\bvolume{25}
\bpages{855-900}.
\end{barticle}
\endbibitem

\bibitem{PY2001}
\begin{barticle}[author]
\bauthor{\bsnm{Pitman},~\bfnm{J.}\binits{J.}} \AND \bauthor{\bsnm{Yor},~\bfnm{M.}\binits{M.}}
(\byear{2001}).
\btitle{On the distribution of ranked heights of excursions of a Brownian bridge}.
\bjournal{The Annals of Probability}
\bvolume{29}
\bpages{361--384}.
\end{barticle}
\endbibitem

\bibitem{QianWerner2019}
\begin{barticle}[author]
\bauthor{\bsnm{Qian},~\bfnm{Wei}\binits{W.}} \AND \bauthor{\bsnm{Werner},~\bfnm{Wendelin}\binits{W.}}
(\byear{2018}).
\btitle{Coupling the Gaussian free fields with free and with zero boundary conditions via common level lines}.
\bjournal{Communications in Mathematical Physics}
\bvolume{361}
\bpages{53--80}.
\end{barticle}
\endbibitem

\bibitem{sagitov1999}
\begin{barticle}[author]
\bauthor{\bsnm{Sagitov},~\bfnm{Serik}\binits{S.}}
(\byear{1999}).
\btitle{The general coalescent with asynchronous mergers of ancestral lines}.
\bjournal{Journal of Applied Probability}
\bvolume{36}
\bpages{1116--1125}.
\end{barticle}
\endbibitem

\bibitem{SVY2007}
\begin{barticle}[author]
\bauthor{\bsnm{Salminen},~\bfnm{P.}\binits{P.}}, \bauthor{\bsnm{Vallois},~\bfnm{P.}\binits{P.}} \AND \bauthor{\bsnm{Yor},~\bfnm{M.}\binits{M.}}
(\byear{2007}).
\btitle{On the excursion theory for linear diffusions}.
\bjournal{Jpn. J. Math.}
\bvolume{2}
\bpages{97--127}.
\bdoi{10.1007/s11537-007-0662-y}
\end{barticle}
\endbibitem

\bibitem{Sato2013}
\begin{bbook}[author]
\bauthor{\bsnm{Sato},~\bfnm{K.~i.}\binits{K.~i.}}
(\byear{2013}).
\btitle{Lévy Processes and Infinitely Divisible Distributions},
\bedition{2nd} ed.
\bseries{Cambridge Studies in Advanced Mathematics}
\bvolume{68}.
\bpublisher{Cambridge University Press}, \baddress{Cambridge}.
\end{bbook}
\endbibitem

\bibitem{Schluter2000}
\begin{bbook}[author]
\bauthor{\bsnm{Schluter},~\bfnm{D.}\binits{D.}}
(\byear{2000}).
\btitle{The Ecology of Adaptive Radiation}.
\bpublisher{Oxford University Press}, \baddress{Oxford}.
\end{bbook}
\endbibitem

\bibitem{SW12}
\begin{barticle}[author]
\bauthor{\bsnm{Sheffield},~\bfnm{Scott}\binits{S.}} \AND \bauthor{\bsnm{Werner},~\bfnm{Wendelin}\binits{W.}}
(\byear{2012}).
\btitle{Conformal loop ensembles: the {M}arkovian characterization and the loop-soup construction}.
\bjournal{Annals of Mathematics}
\bvolume{176}
\bpages{1827--1917}.
\bdoi{10.4007/annals.2012.176.3.8}
\end{barticle}
\endbibitem

\bibitem{TavareEwens}
\begin{barticle}[author]
\bauthor{\bsnm{Tavar\'e},~\bfnm{S.}\binits{S.}}
(\byear{2021}).
\btitle{{The magical Ewens sampling formula}}.
\bjournal{Bulletin of the London Mathematical Society}
\bvolume{53}
\bpages{1563-1582}.
\end{barticle}
\endbibitem

\bibitem{Taylor2009}
\begin{barticle}[author]
\bauthor{\bsnm{Taylor},~\bfnm{J.~E.}\binits{J.~E.}} \AND \bauthor{\bsnm{V\'{e}ber},~\bfnm{A.}\binits{A.}}
(\byear{2009}).
\btitle{Coalescent processes in subdivided populations subject to recurrent mass extinctions}.
\bjournal{Electronic Journal of Probability}
\bvolume{14}
\bpages{242--288}.
\bdoi{10.1214/EJP.v14-595}
\end{barticle}
\endbibitem

\bibitem{tsilevich2000distinguished}
\begin{barticle}[author]
\bauthor{\bsnm{Tsilevich},~\bfnm{N}\binits{N.}}, \bauthor{\bsnm{Vershik},~\bfnm{A}\binits{A.}} \AND \bauthor{\bsnm{Yor},~\bfnm{M}\binits{M.}}
(\byear{2000}).
\btitle{Distinguished properties of the gamma process, and related topics. Pr{\'e}publication du Laboratoire de Probabilit{\'e}s et Modeles Al{\'e}atoires. No. 575, 2000}.
\bjournal{Mars}.
\end{barticle}
\endbibitem

\bibitem{werner2025switchingidentitycablegraphloop}
\begin{bmisc}[author]
\bauthor{\bsnm{Werner},~\bfnm{Wendelin}\binits{W.}}
(\byear{2025}).
\btitle{{A switching identity for cable-graph loop soups and Gaussian free fields}}.
\bdoi{10.48550/arXiv.2502.06754}
\end{bmisc}
\endbibitem

\bibitem{Wood}
\begin{barticle}[author]
\bauthor{\bsnm{Wood},~\bfnm{F.}\binits{F.}}, \bauthor{\bsnm{Gasthaus},~\bfnm{J.}\binits{J.}}, \bauthor{\bsnm{Archambeau},~\bfnm{C.}\binits{C.}}, \bauthor{\bsnm{James},~\bfnm{L.~F.}\binits{L.~F.}} \AND \bauthor{\bsnm{Teh},~\bfnm{Y.~W.}\binits{Y.~W.}}
(\byear{2011}).
\btitle{{The Sequence Memoizer}}.
\bjournal{Communications of the ACM (Research Highlights)}
\bvolume{54}
\bpages{91-98}.
\end{barticle}
\endbibitem

\end{thebibliography}

\end{document}